\documentclass[a4paper,40pt,notitlepage]{article}
\usepackage{amsmath,amsthm,amssymb,mathrsfs}
\usepackage{enumerate}
\usepackage[dvipdfm]{pict2e}
\usepackage{graphicx}
\usepackage[top=3cm, bottom=3.5cm, left=2.5cm, right=2.5cm]{geometry}

\newtheorem{dfn}{Definition}[subsection]
\newtheorem{thm}[dfn]{Theorem}

\newtheorem{lem}[dfn]{Lemma}

\newtheorem{cor}[dfn]{Corollary}

\newtheorem{rem}[dfn]{Remark}
\newtheorem{prop}[dfn]{Proposition}\makeatletter

         \@addtoreset{equation}{section}
\makeatother

\begin{document}
\title{\bf {\large Hausdorff dimension of the scaling limit of loop-erased random walk \\ in three dimensions}}
\author{Daisuke Shiraishi}
\date{}

\maketitle
\begin{abstract}
Let $M_{n}$ be the length (number of steps) of the loop-erasure of a simple random walk up to the first exit from a ball of radius $n$ centered at its starting point. It is shown in \cite{Shi-gr} that there exists $\beta \in (1, \frac{5}{3}]$ such that $E (M_{n} )$ is of order $n^{\beta}$ in 3 dimensions. In the present article, we show that the Hausdorff dimension of the scaling limit of the loop-erased random walk in 3 dimensions is equal to $\beta$ almost surely.

 \end{abstract}
 
 \section{Introduction}
 \subsection{Introduction}
 Loop-erased random walk (LERW) is a simple path obtained by erasing all loops from a random walk path chronologically (see Section \ref{def-lew} for the precise definition), which was originally introduced in \cite{Law-dfn}. In this article, we study the Hausdorff dimension of the scaling limit of LERW in three dimensions.

 It is known that the scaling limit of LERW in $\mathbb{Z}^{d}$ exists for every $d$. Let $S$ be the simple random walk in $\mathbb{Z}^{d}$ started at the origin and $\tau_{n}$ be the first exit time from a ball of radius $n$. We write $\text{LEW}_{n} = \frac{LE (S[0, \tau_{n}])}{n}$ for the rescaled loop-erased random walk obtained by multiplying LERW up to $\tau_{n}$ by $\frac{1}{n}$ (see Section 2.1 for the definition of $LE$). We think of $\text{LEW}_{n}$ as a random element of the metric space of compact subsets in the closed unit ball in $\mathbb{R}^{d}$ endowed with the Hausdorff distance. Then $\text{LEW}_{n}$ converges weakly to a $d$-dimensional Brownian motion for $d \ge 4$ (Theorem 7.7.6 of \cite{Law-book90}), and converges weakly to $\text{SLE}_{2}$ (\cite{Sch}, \cite{LSW}) for $d=2$ (actually, even in a stronger sense). For $d=3$, the sequence $\text{LEW}_{2^{n}}$ is Cauchy in the metric space and it converges weakly to a random compact subset in the closed unit ball in $\mathbb{R}^{3}$, see \cite{Koz}. We denote the weak convergence limit by ${\cal K}$ in $d=3$ and call it the scaling limit of LERW in 3 dimensions. It is also known that ${\cal K}$ is invariant under rotations and dilations, see \cite{Koz}.

 While the scaling limit of LERW for $d \ge 4$ and $d=2$ are well-studied, little is known about ${\cal K}$ when $d=3$. Recently some topological properties of ${\cal K}$ were studied in \cite{SS}. In \cite{SS}, it was proved that ${\cal K}$ is a simple path almost surely, and that the random set obtained by adding the loops of the independent Brownian loop soup of parameter 1 that meet ${\cal K}$ (see \cite{LW} for the Brownian loop soup) to ${\cal K}$, has the same distribution as the trace of Brownian motion (see Section \ref{scaling limit} for details). Furthermore, bounds on the Hausdorff dimension of ${\cal K}$ were also derived in \cite{SS}. Namely, one has 
\begin{equation}
2 - \xi \le \text{dim}_{\text{H}} ( {\cal K} ) \le \beta, \text{ almost surely},
\end{equation}
where $\xi \in (\frac{1}{2}, 1)$ is the intersection exponent for three dimensional Brownian motion (see \cite{Law-haus} for $\xi$) and $\beta \in (1, \frac{5}{3}]$ is the growth exponent for LERW in $d=3$, i.e., if we write $M_{n}$ for the length (the number of steps) of $LE ( S[0, \tau_{n}] )$, then \cite{Shi-gr} shows that the following limit exists in 3 dimensions: 
\begin{equation}
\lim_{n \to \infty} \frac{\log E (M_{n}) }{\log n} = \beta.
\end{equation} 
In particular, we have $1 < \text{dim}_{\text{H}} ({\cal K}) \le \frac{5}{3}$ almost surely.

\medskip
 
In the present article, we will show that $\text{dim}_{\text{H}} ({\cal K}) \ge \beta$, i.e., the main result of the article is the following.
\begin{thm}\label{main result}
Let $d=3$. We have
\begin{equation}\label{main-result-1}
\text{dim}_{\text{H}} ( {\cal K} ) = \beta, \text{ almost surely}.
\end{equation}
\end{thm}

The exact value of $\beta$ is not known or even conjectured. Numerical simulations suggest that $\beta = 1.62 \pm 0.01$, see \cite{GB}, \cite{Wil}. The best rigorous bounds are $\beta \in (1, \frac{5}{3}]$, see \cite{Law LERW99}.

The Hausdorff dimension of the scaling limit of LERW is equal to $2$ for $d \ge 4$ (Theorem 7.7.6 of \cite{Law-book90}), and is equal to $\frac{5}{4}$ for $d=2$ (\cite{LSW}, \cite{Bef}) almost surely. The exponent $\frac{5}{4}$ is called the growth exponent for LERW for $d=2$, that is, it is known that $E(M_{n})$ is of order $n^{\frac{5}{4}}$ in 2 dimensions (see \cite{Ken}, \cite{Mas} and \cite{Law14}).

\subsection{Some words about the proof}
In order to show that the Hausdorff dimension of ${\cal K}$ is bounded below by $\beta$, we will use a standard technique referred to as Frostman's lemma (see Lemma \ref{dim-3}). We explain how to apply it to our situation here.

By Frostman's lemma, we need to construct a positive (random) measure $\mu$ supported on ${\cal K}$ such that its $(\beta - \delta )$-energy $I_{\beta - \delta } ( \mu )$ (see Lemma \ref{dim-3} for the $\beta$-energy) is finite with high probability for any $\delta > 0$.

With this in mind, we partition the unit ball $D$ into a collection of $\epsilon$-cubes formed by $b_{x} = \epsilon \prod_{i=1}^{3} [x_{i}, x_{i} +1]$ for $x = (x_{1}, x_{2}, x_{3} ) \in \mathbb{Z}^{3}$. We first want to construct a random measure $\mu_{ \epsilon }$ which approximates $\mu$ as follows. We introduce a (random) measure $\mu_{\epsilon }$ whose density, with respect to Lebesgue measure, is comparable to $\frac{1}{ P ( {\cal K} \cap b_{x} \neq \emptyset )}$ on each $\epsilon$-cube $b_{x}$ with $\frac{1}{3} \le |\epsilon x| \le \frac{2}{3}$ such that ${\cal K}$ hits $b_{x}$ and assigns measure zero elsewhere (see Section 5.1 for the precise definition of $\mu_{ \epsilon }$). Then the limit of the support of $\mu_{ \epsilon }$ as $\epsilon \to 0$ is contained in ${\cal K}$ almost surely. Therefore we need to show that for every $\delta > 0$ and $r > 0$ there exist constants $c_{r} > 0, C_{\delta, r} < \infty$, which do not depend on $\epsilon$, such that 
\begin{align}
&P \Big( I_{\beta - \delta } ( \mu_{\epsilon} ) \le C_{\delta, r} \Big) \ge 1-r, \label{sketch-1-1} \\
&P \Big( \mu_{\epsilon} (\overline{D} ) \ge c_{r} \Big) \ge 1-r. \label{sketch-1-2}
\end{align}
for all $\epsilon > 0$. Once \eqref{sketch-1-1} and \eqref{sketch-1-2} are proved, let $\mu$ be any weak limit of the $\mu_{\epsilon}$. Then the measure $\mu$ is a positive measure satisfying that its support is contained in ${\cal K}$ and the $(\beta - \delta )$-energy is finite with probability at least $1-r$. Using Frostman's lemma, we get $\text{dim}_{\text{H}} ({\cal K}) \ge \beta - \delta$ with probability $\ge 1-r$, and Theorem \ref{main result} is proved.

Next we explain how to prove \eqref{sketch-1-1} and \eqref{sketch-1-2}. For \eqref{sketch-1-1}, by Markov's inequality, it suffices to show that the first moment of $I_{\beta - \delta } ( \mu_{\epsilon} )$ is bounded above by some constant $C_{\delta}$ uniformly in $\epsilon$. In order to estimate the first moment, by definition of $\mu_{\epsilon}$, we need to give an upper bound of the probability that ${\cal K}$ hits two distinct $\epsilon$-cubes $b_{x}$ and $b_{y}$ with $\frac{1}{3} \le |\epsilon x|, |\epsilon y| \le \frac{2}{3}$ (see the proof of Lemma \ref{dim-6}). Such a bound will be given in the important Theorem \ref{key thm}. Theorem \ref{key thm} roughly claims that $P^{0} \Big( {\cal K} \cap b_{x} \neq \emptyset, \ {\cal K} \cap b_{y} \neq \emptyset \Big)$ is bounded above by $C P^{0} \Big( {\cal K} \cap b_{x} \neq \emptyset \Big) P^{ \epsilon x} \Big( {\cal K} \cap b_{y} \neq \emptyset \Big)$, where $P^{z}$ denotes the probability measure for ${\cal K}$ started at $z$. Since the domain Markov property of ${\cal K}$ has not been established up to now, we will consider the corresponding probability for LERW as follows. As $\text{LEW}_{2^{k}}$ converges to ${\cal K}$ by \cite{Koz}, we can couple them on the same probability space such that the Hausdorff distance between $\text{LEW}_{n}$ and ${\cal K}$ is bounded above by $\epsilon^{2}$ for large $n=2^{k}$ with high probability. Then the problem boils down to estimates of  $P \Big( LE (S[ 0, \tau_{n} ] ) \cap n b_{x} \neq \emptyset , \ LE (S[ 0, \tau_{n} ] ) \cap n b_{y} \neq \emptyset \Big)$. It is crucial to control the dependence of these two events with the help of the domain Markov property for LERW (see Lemma \ref{domain} for the domain Markov property). This key step will be done in Theorem \ref{key thm}. In Theorem \ref{key thm}, we will show that the probability is bounded above by $C P \Big( LE (S[ 0, \tau_{n} ] ) \cap n b_{x} \neq \emptyset \Big) P^{\epsilon n x} \Big(  LE (S[ 0, \tau_{n} ] ) \cap n b_{y} \neq \emptyset \Big)$, and using some results derived in \cite{Shi-gr}, we will derive a bound of this product in terms of escape probabilities defined as follows. Let $R_{1} \le R_{2}$ and let $S^{1}, S^{2}$ be two independent simple random walks started at the origin. We write $\tau^{i}_{R}$ for the first time that $S^{i}$ hits the boundary of the ball of radius $R$. We  define the escape probability $Es (R_{1}, R_{2} )$ by
\begin{equation*}
Es (R_{1}, R_{2} ) = P_{1} \otimes P_{2} \Big( LE ( S^{1} [0, \tau^{1}_{R_{2}} ] ) [s, u] \cap S^{2} [0, \tau^{2}_{R_{2}} ] = \emptyset \Big),
\end{equation*}
where $u$ is the length of $LE ( S^{1} [0, \tau^{1}_{R_{2}} ] )$, and $s$ is its last visit to the ball of radius $R_{1}$ before time $u$ (see Section 2.2 for escape probabilities). In order for $z$ to be in $LE (S[ 0, \tau_{n} ] )$, by definition of $LE$ (see Definition \ref{procedure}), the following two conditions are required: (i) $S$ hits $z$ before $\tau_{n}$. (ii) The loop-erasure of $S$ from the origin to $z$ does not intersect the remaining part of $S$ from $z$ to $S ( \tau_{n})$. Reversing a path, the probability for $z$ to be in the LERW is equal to the probability that with $S^{1} (0) = S^{2} (0) = z$,
\begin{itemize}
\item $S^{1}$ hits the origin before exiting the ball of radius $n$,

\item The loop-erasure of $S^{1}$ from $z$ to the origin does not intersect $S^{2}$ up to exiting the ball.
\end{itemize}
It turns out that this probability is comparable to $n^{-1} Es (0, n)$ if $\frac{n}{3} \le |z| \le \frac{2 n}{3}$. Furthermore, a similar consideration gives that the probability of $LE ( S[ 0, \tau_{n} ] )$ hitting $n b_{x}$ is comparable to $\epsilon Es ( \epsilon n, n)$, which leads that the probability of ${\cal K}$ hitting $b_{x}$ is also comparable to $\epsilon Es ( \epsilon n, n)$. (In fact, we will set $\frac{1}{\epsilon Es ( \epsilon n, n)}$ on each $b_{x}$ hit by ${\cal K}$ for the density of $\mu_{\epsilon}$, where we chose $n$ as an arbitrary large integer so that the distance between $\text{LEW}_{n}$ and ${\cal K}$ is small with high probability as explained above. We also point out that for all large $n$, $Es ( \epsilon n, n)$ is of order $\epsilon^{\alpha + o(1)}$ for some constant $\alpha$, see Theorem \ref{escape-5}.)
Finally Theorem \ref{key thm} concludes that 
\begin{equation}\label{sketch-2}
P \Big( {\cal K} \cap b_{x} \neq \emptyset, \ {\cal K} \cap b_{y} \neq \emptyset \Big) \le  \frac{C \epsilon }{|x -y | } Es ( \epsilon n, n) Es ( \epsilon n , \epsilon n |x- y| ),
\end{equation}
which is a new result to our knowledge. 
Combining \eqref{sketch-2} with estimates for the escape probabilities obtained in \cite{Shi-gr} (see Section 2.2), we get \eqref{sketch-1-1}.

Next we consider \eqref{sketch-1-2}. The definition of $\mu_{\epsilon}$ immediately gives that $\mu_{\epsilon} (\overline{D} )$ is equal to $\frac{ \epsilon^{2} Y^{\epsilon} }{ Es ( \epsilon n, n)}$, where $Y^{\epsilon}$ stands for the number of $\epsilon$-cubes $b_{x}$ with $\frac{1}{3} \le |\epsilon x| \le \frac{2}{3}$ such that ${\cal K}$ hits $b_{x}$. (Recall that we 
choose $n$ large enough so that the Hausdorff distance between $\text{LEW}_{n}$ and ${\cal K}$ is smaller than $\epsilon^{2}$ with high probability in the coupling explained as above.) Therefore, in order to prove \eqref{sketch-1-2}, it suffices to show that for all $r > 0$ there exists $c_{r} > 0$ such that
\begin{equation}\label{rough}
P \Big(  Y^{\epsilon} \ge c_{r} \epsilon^{-2} Es ( \epsilon n, n) \Big) \ge 1- r,
\end{equation}
for all $\epsilon > 0$. 

Since the probability of ${\cal K}$ hitting $b_{x}$ is comparable to $\epsilon Es ( \epsilon n, n)$, the first moment of $Y^{\epsilon}$ is of order $\epsilon^{-2} Es ( \epsilon n, n)$. Using \eqref{sketch-2}, it turns out that the second moment of $Y^{\epsilon}$ is comparable to the square of its first moment. So the second moment method gives that $Y^{\epsilon}$ is bounded below by $c \epsilon^{-2} Es ( \epsilon n, n)$ with positive probability for some $c > 0$ (Corollary \ref{positive-low}). However this is not enough to prove \eqref{rough} and we need more careful considerations that we will explain below.

In order to prove \eqref{rough}, again we use the coupling of ${\cal K}$ and $\text{LEW}_{n}$ explained as above. Then \eqref{rough} boils down to the corresponding estimates for LERW as follows. Let $Y^{\epsilon}_{n}$ be the number of $\epsilon n$-cubes $n b_{x}$ with $\frac{1}{3} \le |\epsilon x| \le \frac{2}{3}$ such that $LE (S[ 0, \tau_{n} ] )$ hits $n b_{x}$. Then \eqref{rough} is reduced to proving that for all $r > 0$ there exists $c_{r} > 0$ such that
\begin{equation}\label{rough-1}
P \Big(  Y^{\epsilon}_{n} \ge c_{r} \epsilon^{-2} Es ( \epsilon n, n) \Big) \ge 1- r.
\end{equation}

% Again using the coupling of ${\cal K}$ and $\text{LEW}_{n}$ as above, the problem boils down to

We will prove \eqref{rough-1} in Proposition 4.2.2 using Markovian-type ``iteration arguments" that we will briefly explain here. In order to prove \eqref{rough-1}, we consider $N$ cubes $A_{i}$ ($i=1, \cdots , N$) of side length $\frac{n}{3} + \frac{i n}{3 N}$. We are interested in a subpath $\gamma_{i}$ of $\gamma := LE (S[ 0, \tau_{n} ] )$ which consists of $\gamma$ between its first visit to $\partial A_{i}$ and that to $\partial A_{i+1}$ (see the beginning of Section 4 for the precise definition of $\gamma_{i}$). We want to show that for all $r > 0$, by choosing $N = N_{r}$ and $c_{r}$ suitably, the probability that at least one of $\gamma_{i}$ hits $c_{r} \epsilon^{-2} Es ( \epsilon n, n)$ $\epsilon n$-cubes is bigger than $1-r$. To achieve this, we prove in Lemma 4.2.1 that given $\gamma_{1}, \cdots , \gamma_{i}$ the probability that $\gamma_{i+1}$ hits $c_{r} \epsilon^{-2} Es ( \epsilon n, n)$ cubes is bigger than some universal constant $c >0$ for each $i$. This enables us to show that the probability in \eqref{rough-1} is bigger than $1- (1-c)^{N}$ and finish the proof of \eqref{rough-1} by taking $N$ such that $(1-c)^{N} < r$. To establish Lemma 4.2.1, it is crucial to deal with some sort of independence of $\gamma_{i}$. The domain Markov property (see Lemma \ref{domain}) tells that we need to study a random walk conditioned not to intersect $\gamma_{1}, \cdots , \gamma_{i}$. We will study such a conditioned random walk in Section 4.1. Then we will prove Lemma 4.2.1 and \eqref{rough} by using results derived there in Section 4.2. To our knowledge the tail estimate of $Y^{\epsilon}$ as in \eqref{rough} is also new. This iteration argument is based on the same spirit of the proof of Theorem 6.7 of \cite{BM} and Theorem 8.2.6 of \cite{Shi-gr} where exponential lower tail bounds of $M_{n}$ were established for $d=2$ (\cite{BM}) and $d=3$ (\cite{Shi-gr}).

\begin{rem}\label{chuui-1}
As we discussed above, $ E( Y^{\epsilon} )$ is comparable to $\epsilon^{-2} Es ( \epsilon n, n)$, which is of order $\epsilon^{-(2- \alpha ) + o (1)}$ for some exponent $\alpha \in [\frac{1}{3}, 1 )$ (see Theorem \ref{escape-5}). It turns out that $\beta$ in Theorem \ref{main result} is equal to $2- \alpha$.
\end{rem}

\begin{rem}\label{chuui}
It is crucial that both the upper bound in the right hand side of \eqref{sketch-2} and the lower bound of $Y^{\epsilon }$ are given in terms of the escape probabilities. Since $Es ( \epsilon n, n)= \epsilon ^{\alpha + o(1)}$ (see Theorem \ref{escape-5}), one may suppose that in order to prove Theorem \ref{main result}, it suffices to show that for every $\delta > 0$, $Y^{\epsilon } \ge c \epsilon^{-(2-\alpha) + \delta}$ with high probability instead of proving \eqref{rough}. However this is not the case. Energy estimates as in Lemma \ref{dim-6} do not work if we rely on only such estimates without using the escape probabilities.
\end{rem}

\subsection{Structure of the paper}
The organization of the paper is as follows. In the next subsection, we will give a list of notation used throughout the paper.

In Section 2, we will review known facts about LERW. We explain some basic properties of LERW in Section 2.1. In order to show Theorem \ref{main result}, the probability that an LERW and an independent simple random walk do not intersect up to exiting a large ball, which is referred to as an escape probability, is a key tool. That probability will be considered in Section 2.2. The precise definition and some properties of ${\cal K}$ will be given in Section 2.3. 

One of the key results in the paper is Theorem \ref{key thm}, which gives an upper bound of the probability that ${\cal K}$ hits two small boxes. The proof of Theorem \ref{key thm} will be given in Section 3.1. Using Theorem \ref{key thm}, we study the number of small boxes hit by ${\cal K}$ in Section 3.2. By the second moment method, we give a lower bound of the number of those boxes hit by ${\cal K}$ in Corollary \ref{positive-low}. 

To establish \eqref{main-result-1} almost surely, we need to show \eqref{rough} which is an improvement of Corollary \ref{positive-low} in Section 4. Following iteration arguments used in the proof of Theorem 6.7 \cite{BM} and  Proposition 8.2.5 of \cite{Shi-gr}, we study a random walk conditioned not to intersect a given simple path in Section 4.1. Using estimates derived there, we will prove \eqref{rough} and \eqref{sketch-1-2} in Section 4.2. 

We will prove \eqref{sketch-1-1} in Section 5.1.
Finally, using Frostman's lemma (see Lemma \ref{dim-3}), we will prove Theorem \ref{main result} in Section 5.2.

 \subsection{Notation}\label{kigou}
 In this subsection, we will give some definitions which will be used throughout the paper. 

Let $\lambda = [ \lambda (0), \lambda (1), \cdots , \lambda (m) ] $ be a sequence of points in $\mathbb{Z}^{d}$. We call it a path if $|\lambda (j-1) - \lambda (j) | =1 $ for all $j$. In that case we say $\lambda$ has a length $m$ and denote the length of $\lambda$ by len$\lambda$. We call $\lambda$ a simple path if $\lambda (i) \neq \lambda (j)$ for all $i \neq j$.

We use $| \cdot |$ for the Euclid distance in $\mathbb{R}^{d}$. For $n \ge 0$ and $z \in \mathbb{Z}^{d}$, define $B_{z,n } =B (z, n) := \{  x \in \mathbb{Z}^{d} \ | \ |x-z| < n \}$. If $z = 0$, we write $B_{0, n} = B(0, n) = B(n)$. We write $D = \{ x \in \mathbb{R}^{d} \ | \ |x| < 1 \}$ and $\overline{D}$ for its closure. For $r > 0$, let $D_{r} = \{ x \in \mathbb{R}^{d} \ | \ |x| < r \}$ and $\overline{D_{r}}$ for its closure.

For a subset $A \subset \mathbb{Z}^{d}$, we let $\partial A = \{ x \notin A \ | \ \text{ there exists } y \in A \text{ such that } |x-y| =1 \}$ and $\partial_{i} A = \{ x \in A \ | \ \text{ there exists } y \notin A \text{ such that } |x-y| =1 \}$. We write $\overline{A} := A \cup \partial A$. Given a subset $A \subset \mathbb{Z}^{d}$ and $r > 0$, we write $r A := \{ ry \ | \ y \in A \}$.
 
Throughout the paper, $S$, $S^{1}$, $S^{2}$, $S^{3}$ and $S^{4}$ denote independent simple random walks on $\mathbb{Z}^{d}$. For the probability law and the expectation of $S$ started at $z$, we use $P^{z}$ and $E^{z}$ respectively. If $z=0$, we write $P^{0} = P$ and $E^{0} = E$.  For the probability law and the expectation of $S^{i}$ started at $z$, we use $P^{z}_{i}$ and $E^{z}_{i}$ respectively. If $z=0$, we write $P^{0}_{i} = P_{i}$ and $E^{0}_{i} = E_{i}$.

Given $n \ge 1$, let $\tau_{n} := \inf \{ k \ | \ S(k) \notin B(n) \}$ and $\tau^{i}_{n} := \inf \{ k \ | \ S^{i}(k) \notin B(n) \}$. For $z \in \mathbb{Z}^{d}$, we write $T_{z, n} := \inf \{ k \ | \ S(k) \in \partial B(z, n) \}$ and $T^{i}_{z, n} := \inf \{ k \ | \ S^{i}(k) \in \partial B(z, n) \}$.

For a subset $A \subset \mathbb{Z}^{d}$, define Green's function in $A$ by $G (x, y, A) = G_{A} (x, y) = E^{x} \Big( \sum_{j= 0}^{\tau - 1} {\bf 1} \{ S(j) = y \} \Big)$ for $x, y \in A$, where $\tau = \inf \{ t \ | \ S(t) \in \partial A \}$.

We use $c, C, C_{1}, \dotsb$ to denote arbitrary positive constants which may change from line to line. If a constant is to depend on some other quantity, this will be made explicit. For example, if $C$ depends on $\delta$, we write $C_{\delta}$. To avoid complication of notation, we do not use $\lfloor r \rfloor$ (the largest integer $\le r$) even though it is necessary to carry it.

 \section{Loop-erased random walk}
 In this section, we will review some known facts about loop-erased random walk (LERW). In Section 2.1, we begin with the definition of loop-erasure and LERW. Then we state the time reversibility and the domain Markov property of LERW. All results in Section 2.1 hold for LERW in $\mathbb{Z}^{d}$ (even in any graphs). 
 
As we discussed in Section 1.2, the probability that an LERW and an independent simple random walk do not intersect up to exiting a large ball, which is referred to as escape probability, is a key tool in the paper. We will define and consider the escape probability for LERW in $\mathbb{Z}^{3}$ in Section 2.2. Most of estimates for escape probabilities stated there are results derived in \cite{Shi-gr}, and those results will be repeatedly used throughout the paper.

We will explain some known results about the scaling limit of LERW in 3 dimensions in Section 2.3. 

\subsection{Basic properties}\label{def-lew}
In this subsection, we first define the loop-erasure of a given path in Definition \ref{procedure}. LERW is a (random) simple path obtained by loop-erasing from a random walk. It satisfies the time reversibility (see Lemma \ref{reversal}). LERW is not a Markov process by definition, but it satisfies the domain Markov property (see Lemma \ref{domain}). Lemma \ref{reversal} and Lemma \ref{domain} hold for LERW in $\mathbb{Z}^{d}$ for all $d$.

\medskip

We begin with the definition of loop-erasure of a path.
\begin{dfn}\label{procedure}
Given a path $\lambda = [\lambda (0), \lambda (1), \cdots , \lambda (m) ] \subset \mathbb{Z}^{d}$, define its loop-erasure $LE ( \lambda )$ as follows. Let 
\begin{equation}\label{procedure-1}
s_{0} := \max \{ t \ | \ \lambda (t) = \lambda (0) \},
\end{equation}
and for $i \ge 1$, let 
\begin{equation}\label{procedure-2}
s_{i} := \max \{ t \ | \ \lambda (t) = \lambda (s_{i-1} +1) \}.
\end{equation}
We write $n = \min \{ i \ | \ s_{i} = m \}$. Then define $LE ( \lambda ) $ by 
\begin{equation}\label{procedure-3}
LE ( \lambda ) = [ \lambda ( s_{0} ), \lambda ( s_{1} ), \cdots , \lambda ( s_{n} ) ].
\end{equation}
\end{dfn}
Throughout the paper, we are interested in the loop-erasure of random walks running until some stopping time, the loop-erased random walk. 

\medskip

For two paths $\lambda_{1} = [\lambda_{1} (0), \lambda_{1} (1), \cdots , \lambda_{1} (m_{1}) ]$ and $\lambda_{2} = [\lambda_{2} (0), \lambda_{2} (1), \cdots , \lambda_{2} (m_{2}) ]$ in $\mathbb{Z}^{d}$ with $\lambda_{1} (m_{1}) = \lambda_{2} (0)$, we write 
\begin{equation}\label{comp}
\lambda_{1} + \lambda_{2} := [\lambda_{1} (0), \lambda_{1} (1), \cdots , \lambda_{1} (m_{1}), \lambda_{2} (1), \cdots , \lambda_{2} (m_{2}) ].
\end{equation}
We will use repeatedly the following notation for $LE ( \lambda_{1} + \lambda_{2} )$. Let $u = \min \{ t \ | \ LE ( \lambda_{1} ) (t) \in \lambda_{2} \}$ and let $s = \max \{ t \ | \ \lambda_{2} (t) = LE ( \lambda_{1} ) (u) \}$. Define 
\begin{equation}\label{comp-1}
LE^{(1)} = LE_{1} ( \lambda_{1}, \lambda_{2} ) := LE ( \lambda_{1} ) [0, u ], \ LE^{(2)} = LE_{2} ( \lambda_{1}, \lambda_{2} ) := LE ( \lambda_{2} [ s, m_{2} ] ).
\end{equation}
Then it is easy to check that $LE ( \lambda_{1} + \lambda_{2} ) = LE^{(1)} + LE^{(2)}$.

\medskip

For a path $\lambda = [\lambda (0), \lambda (1), \cdots , \lambda (m) ] \subset \mathbb{Z}^{d}$, define its time reversal $\lambda^{R}$ by $\lambda^{R} := [\lambda (m), \lambda (m-1), \cdots , \lambda (0) ]$. Note that in general, $LE ( \lambda ) \neq (LE ( \lambda^{R} ) )^{R}$. However, as next lemma shows, the time reversal of LERW has same distribution to the original LERW. Let $\Lambda_{m}$ be the set of paths of length $m$ started at the origin.

\begin{lem}\label{reversal} (Lemma 7.2.1 \cite{Law-book90})
For each $m \ge 0$, there exists a bijection $T^{m} \ : \ \Lambda_{m} \to \Lambda_{m}$ such that for each $\lambda \in \Lambda_{m}$,
\begin{equation}\label{rev-1}
LE ( \lambda ) = (LE ( (T^{m} \lambda )^{R} ) )^{R}.
\end{equation}
Moreover, it follows that $\lambda$ and $T^{m} \lambda $ visit the same edges in the same directions with the same multiplicities.
\end{lem}

\medskip

Note that LERW is not a Markov process. However it satisfies the domain Markov property in the following sense. 

\begin{lem}\label{domain} (Proposition 7.3.1 \cite{Law-book90})
Let $D \subset \mathbb{Z}^{d}$ be a finite subset. Suppose that $\lambda_{i}$ ($i =1, 2$) are simple paths of length $m_{i}$ with $\lambda_{1} \subset D$, $\lambda_{1} ( m_{1} ) = \lambda_{2} (0) $. Suppose also that $\lambda_{1} + \lambda_{2}$ is a simple path from $\lambda_{1} (0)$ terminated at $\partial D$. Let $Y$ be a random walk $R$ started at $\lambda_{2} (0)$ conditioned on $R [1, \sigma^{R}_{D} ] \cap \lambda_{1} = \emptyset$. Here $\sigma^{R}_{D} = \inf \{ t \ | \ R(t) \notin D \}$. Then we have
\begin{equation}\label{domain-1}
P^{\lambda_{1} (0) } \Big( LE ( S[0, \sigma_{D} ] ) = \lambda_{1} + \lambda_{2} \ \big| \ LE ( S[0, \sigma_{D} ] ) [0, m_{1} ] = \lambda_{1} \Big) = P \Big( LE ( Y [0, \sigma^{Y}_{D} ] ) = \lambda_{2} \Big),
\end{equation}
where $\sigma_{D} $ (resp. $\sigma^{Y}_{D}$) is the first exit time from $D$ for $S$ (resp. $Y$).
\end{lem}

\subsection{Escape probabilities}
As we discussed in Section 1.2, the probability that an LERW and an independent simple random walk do not intersect up to hitting a boundary of a large ball is a key ingredient in the present paper. Such a probability is called an escape probability. The escape probability was studied in order to estimate the length of LERW for $d=2$ in \cite{Mas}, \cite{BM} and for $d=3$ in \cite{Shi-gr}.
In this subsection, we will explain it. In this subsection we recall several results proved in \cite{Shi-gr}. Throughout this subsection, we will assume $d=3$.

\begin{dfn}\label{escape}
Let $m < n$. Suppose that $S^{1}$ and $S^{2}$ are independent simple random walks started at the origin on $\mathbb{Z}^{3}$. Define escape probabilities $Es (n)$, $Es^{\star } (n)$ and $Es (m, n)$ as follows: Let
\begin{equation}\label{escape-1-1}
Es (n) := P_{1} \otimes P_{2} \Big( S^{1} [1, \tau^{1}_{n} ] \cap LE ( S^{2} [0, \tau^{2}_{n} ] ) = \emptyset \Big), 
\end{equation}
i.e., $Es (n)$ is the probability that a simple random walk up to exiting $B (n)$ does not intersect the loop erasure of an independent simple random walk up to exiting $B (n)$. Let
\begin{equation}\label{escape-1-2}
Es^{\star } (n) := P_{1} \otimes P_{2} \Big( S^{1} [1, \tau^{1}_{n} ] \cap \eta^{1}_{0, n} \big( LE ( S^{2} [0, \tau^{2}_{4 n} ] ) \big) = \emptyset \Big),
\end{equation}
where $\eta^{1}_{z, n} ( \lambda ) = \lambda [0, u]$ with $u = \inf \{ t \ | \ \lambda ( t) \in \partial B (z, n ) \}$. For $Es^{\star } (n)$, we first consider the loop erasure of a random walk up to exiting $B (4 n)$, then we only look at the loop erasure from the origin to the first visit to $\partial B (n)$. $Es^{\star } (n)$ is the probability that this part of the loop erasure does not intersect an independent simple random walk up to exiting $B (n)$. Finally, let
\begin{equation}\label{escape-1-3}
Es (m, n) :=  P_{1} \otimes P_{2} \Big( S^{1} [1, \tau^{1}_{n} ] \cap \eta^{2}_{0, m, n} \big( LE ( S^{2} [0, \tau^{2}_{n} ] ) \big) = \emptyset \Big),
\end{equation}
where $\eta^{2}_{z, m, n} ( \lambda ) = \lambda [s, u ]$ with $s = \sup \{ t \le u \ | \ \lambda ( t) \in \partial B (z, m ) \}$ ($u$ was defined as above). For $Es (m, n)$, we first consider the loop erasure of a random walk up to exiting $B (n)$, then we only look at the loop erasure after the last visit to $B (m)$. $Es (m, n)$ is the probability that this part of the loop erasure does not intersect an independent simple random walk up to exiting $B (n)$.
\end{dfn}

\medskip

In the next proposition we collect various relations between the escape probabilities on various scales.

\begin{prop}\label{escape-3} (Propositions 6.2.1, 6.2.2, and 6.2.4 \cite{Shi-gr})
Let $d=3$. There exists a constant $C < \infty$ such that for all $l \le m \le n$, 
\begin{align}\label{escape-4}
& \frac{1}{C} Es^{\star } (n) \le Es (n) \le C Es^{\star } (n), \notag \\
& \frac{1}{C} Es (n') \le Es (n) \le C Es (n'), \text{ for all } n \le n' \le 4n, \notag \\
& \frac{1}{C} Es (n) \le Es (m) Es (m, n) \le C Es (n), \notag \\
& \frac{1}{C} Es (l, n) \le Es (l, m) Es (m, n) \le C Es (l, n).
\end{align}
\end{prop}

\medskip

The next theorem deals with the rate of growth for $Es (n)$ and $Es (m , n)$ in $d=3$.
\begin{thm}\label{escape-5} (Theorem 7.2.1 and Lemma 7.2.2 \cite{Shi-gr})
Let $d=3$. There exists $\alpha \in [\frac{1}{3}, 1)$ such that 
\begin{equation}\label{escape-6}
\lim_{n \to \infty } \frac{ \log Es (n) }{\log n} = - \alpha.
\end{equation}
Furthermore, for all $\kappa > 0$ there exists $ c_{ \kappa } > 0$ and $n_{\kappa } \in \mathbb{N}$ such that 
\begin{equation}\label{escape-7}
c_{\kappa } \Big( \frac{n}{m} \Big)^{-\alpha - \kappa } \le Es (m, n) \le \frac{1}{c_{\kappa }} \Big( \frac{n}{m} \Big)^{-\alpha + \kappa },
\end{equation}
for all $n_{\kappa} \le m \le n$.
\end{thm}

\medskip

The next lemma gives bounds of the ratio of escape probabilities, which will be used repeatedly in the paper.
\begin{lem}\label{escape-8} (Lemma 7.2.3 \cite{Shi-gr})
Let $d=3$. For all $\kappa > 0$, there exists $C_{\kappa} < \infty$ such that for all $1 \le m \le n$,
\begin{equation}\label{escape-9}
m^{\alpha + \kappa} Es (m) \le C_{\kappa} n^{\alpha + \kappa} Es (n).
\end{equation}
Furthermore, for $l \le m$, by dividing both sides above by $Es (l)$ and using Proposition \ref{escape-3}, we see that for all $1 \le l \le m \le n$ 
\begin{equation}\label{escape-10}
m^{\alpha + \kappa} Es (l, m) \le C^{2} C_{\kappa} n^{\alpha + \kappa} Es (l, n),
\end{equation}
where $C$ is a constant as in Proposition \ref{escape-3}.
\end{lem}

\medskip

Let $\tau_{n} = \inf \{ t \ | \ S(t) \in \partial B (n) \}$ and let $M_{n} = \text{len} LE ( S[0, \tau_{n}] )$. The next theorem relates the length of LERW with the escape probability.

\begin{thm}\label{escape-11} (Theorem 8.1.4 and Proposition 8.1.5 \cite{Shi-gr})
Let $d=3$. There exists a constant $C < \infty$ such that for all $n \ge 1$,
\begin{equation}\label{escape-12} 
\frac{1}{C} n^{2} Es (n) \le  E ( M_{n} ) \le C n^{2} Es (n).
\end{equation}
In particular, we have
\begin{equation}\label{escape-13} 
\lim_{n \to \infty} \frac{\log E (M_{n} ) }{\log n} = 2- \alpha.
\end{equation}
\end{thm}

\medskip

In the rest of this subsection, we will give some extension of Theorem 6.1.5 \cite{Shi-gr} which is referred to as the ``separation lemma". Let $R \ge 4$, $n \ge 1$ and $Rn \le L \le 4 Rn$. We are interested in the following event.
\begin{equation}\label{escape-14}
F_{L, R, n} := \Big\{ \eta^{2}_{0, n, Rn} \big( LE ( S^{1} [0, \tau^{1}_{L} ] ) \big) \cap S^{2} [0, \tau^{2}_{Rn} ] = \emptyset \Big\},
\end{equation}
where $\eta^{2}$ was defined right after \eqref{escape-1-3} in Definition \ref{escape}. Let
\begin{align}\label{escape-15}
&A^{+}_{R, n} := \big\{ x= (x_{1}, x_{2}, x_{3} ) \in \mathbb{R}^{3} \ | \ x_{1} \ge \frac{2 Rn}{3} \big\}  \cup B ( \frac{3 Rn}{4} ), \notag \\
&A^{-}_{R, n} :=  \big\{ x= (x_{1}, x_{2}, x_{3} ) \in \mathbb{R}^{3} \ | \ x_{1} \le - \frac{2 Rn}{3} \big\} \cup B ( \frac{3 Rn}{4} ).
\end{align}
Define
\begin{equation}\label{escape-16}
\text{Sep}_{L, R, n} := \Big\{ \eta^{2}_{0, n, Rn} \big( LE ( S^{1} [0, \tau^{1}_{L} ] ) \big) \subset A^{-}_{R, n}, \ S^{2} [0, \tau^{2}_{Rn} ] \subset A^{+}_{R, n} \Big\}.
\end{equation}

The next lemma shows that when a simple random walk does not intersect an independent LERW, they are ``well-separated" with positive probability, i.e., the simple random walk lies in $A^{+}_{R, n}$ and the LERW lies in $A^{-}_{R, n}$ with positive conditional probability under the conditioning. The lemma will be used to compare escape probabilities on various scales by attaching paths to the separated paths (see Lemma \ref{escape-25}, Proposition \ref{first-moment-low} and Lemma \ref{imp-25} for the applications of Lemma \ref{escape-17}).

\begin{lem}\label{escape-17}
Let $d=3$. There exists $c > 0$ such that for all $R \ge 4$, $n \ge 1$ and $Rn \le L \le 4 Rn$, we have 
\begin{equation}\label{escape-18}
P_{1} \otimes P_{2} \Big( \text{Sep}_{L, R, n} \ \big| \ F_{L, R, n} \Big) \ge c.
\end{equation}
\end{lem}

\begin{proof}
Throughout the proof of this lemma, let $\gamma := LE ( S^{1} [0, \tau^{1}_{L} ] )$ and $\gamma' := LE ( S^{1} [0, \infty ) )$. Note that $\gamma'$ is well-defined since $S^{1}$ is transient for $d=3$. For $m < k$, let $\Lambda_{m, k}$ be the set of pairs of two paths $(\lambda^{1}, \lambda^{2})$ satisfying that
\begin{itemize}
\item $\lambda^{1}$ is a simple path started at the origin. $\lambda^{2}$ is a path started at the origin.

\item $\lambda^{i} [0, \text{len} \lambda^{i} ] \subset B(k)$ and $\lambda^{i} ( \text{len} \lambda^{i} ) \in \partial B(k)$ for each $i=1,2$.

\item $\eta^{2}_{0, m, k} (\lambda^{1}) \cap \lambda^{2} = \emptyset$.
\end{itemize}

We write $\tau^{\gamma}_{m} = \inf \{ t \ | \ \gamma (t) \in \partial B (m) \}$ and define $\tau^{\gamma'}_{m}$ similarly. Let 
\begin{align*}
&G' := \Big\{ \gamma' [ \tau^{\gamma'}_{\frac{ R n}{8}}, \tau^{\gamma'}_{\frac{ R n}{4}} ] \cap S^{2} [ \tau^{2}_{\frac{ R n}{8}}, \tau^{2}_{\frac{ R n}{4}}] = \emptyset, \ \gamma' [ \tau^{\gamma'}_{\frac{ R n}{8}}, \tau^{\gamma'}_{\frac{ R n}{4}} ] \cap \lambda^{2} = \emptyset, \ S^{2} [ \tau^{2}_{\frac{ R n}{8}}, \tau^{2}_{\frac{ R n}{4}}] \cap \lambda^{1} = \emptyset \Big\}, \\
&H' := \Big\{ \gamma' [ \tau^{\gamma'}_{\frac{ R n}{8}}, \tau^{\gamma'}_{\frac{ R n}{4}} ] \subset A^{-}_{\frac{R}{4}, n}, \ S^{2} [ \tau^{2}_{\frac{ R n}{8}}, \tau^{2}_{\frac{ R n}{4}}] \subset A^{+}_{\frac{R}{4}, n} \Big\}.
\end{align*}
(We define $G$ and $H$ by replacing $\gamma'$ by $\gamma$ above.)
Then (6.15) of \cite{Shi-gr} shows that there exists an absolute constant $c > 0$ such that for all $(\lambda^{1}, \lambda^{2}) \in \Lambda_{n, \frac{R n}{8}}$, we have
\begin{equation}\label{daru-1}
P_{1} \otimes P_{2} \Big( G', \ H' \ \big| \ \gamma' [0, \tau^{\gamma'}_{\frac{ R n}{8}}] = \lambda^{1}, S^{2} [0,  \tau^{2}_{\frac{ R n}{8}}] = \lambda^{2} \Big) \ge c P_{1} \otimes P_{2} \Big( G' \ \big| \ \gamma' [0, \tau^{\gamma'}_{\frac{ R n}{8}}] = \lambda^{1}, S^{2} [0,  \tau^{2}_{\frac{ R n}{8}}] = \lambda^{2} \Big).
\end{equation}
Taking sum for $(\lambda^{1}, \lambda^{2}) \in \Lambda_{n, \frac{R n}{8}}$, we have
\begin{align}\label{daru-2}
&P_{1} \otimes P_{2} \Big( G'_{1}, \ H', \ \eta^{2}_{0, n, \frac{R n}{8}} \big( \gamma' [0, \tau^{\gamma'}_{\frac{ R n}{8}}] \big) \cap S^{2} [0,  \tau^{2}_{\frac{ R n}{8}}] = \emptyset \Big) \notag \\
&\ge c P_{1} \otimes P_{2} \Big( G'_{1}, \ \eta^{2}_{0, n, \frac{R n}{8}} \big( \gamma' [0, \tau^{\gamma'}_{\frac{ R n}{8}}] \big) \cap S^{2} [0,  \tau^{2}_{\frac{ R n}{8}}] = \emptyset \Big),
\end{align}
where
\begin{align*}
&G'_{1} := \Big\{ \gamma' [ \tau^{\gamma'}_{\frac{ R n}{8}}, \tau^{\gamma'}_{\frac{ R n}{4}} ] \cap S^{2} [ \tau^{2}_{\frac{ R n}{8}}, \tau^{2}_{\frac{ R n}{4}}] = \emptyset, \ \gamma' [ \tau^{\gamma'}_{\frac{ R n}{8}}, \tau^{\gamma'}_{\frac{ R n}{4}} ] \cap S^{2} [0,  \tau^{2}_{\frac{ R n}{8}}] = \emptyset, \notag \\
&\ \ \ \ \ \ \ \  \ \  S^{2} [ \tau^{2}_{\frac{ R n}{8}}, \tau^{2}_{\frac{ R n}{4}}] \cap \gamma' [0, \tau^{\gamma'}_{\frac{ R n}{8}}] = \emptyset \Big\}.
\end{align*}
(Again we define $G_{1}$ by replacing $\gamma'$ by $\gamma$ above.) 

But by Proposition 4.4 \cite{Mas}, the distribution of $\gamma' [ 0, \tau^{\gamma'}_{\frac{ R n}{4}} ]$ is comparable to that of $\gamma [ 0, \tau^{\gamma}_{\frac{ R n}{4}} ]$. Therefore,
\begin{align}\label{daru-3}
&P_{1} \otimes P_{2} \Big( G_{1}, \ H, \ \eta^{2}_{0, n, \frac{R n}{8}} \big( \gamma [0, \tau^{\gamma}_{\frac{ R n}{8}}] \big) \cap S^{2} [0,  \tau^{2}_{\frac{ R n}{8}}] = \emptyset \Big) \notag \\
&\ge c P_{1} \otimes P_{2} \Big( G_{1}, \ \eta^{2}_{0, n, \frac{R n}{8}} \big( \gamma [0, \tau^{\gamma}_{\frac{ R n}{8}}] \big) \cap S^{2} [0,  \tau^{2}_{\frac{ R n}{8}}] = \emptyset \Big).
\end{align}
Once $\gamma [ 0, \tau^{\gamma}_{\frac{ R n}{4}} ]$ and $S^{2} [ 0, \tau^{2}_{\frac{ R n}{4}}]$ are separated as in $H$, by attaching paths from $\partial B( \frac{ R n}{4} )$ to $\partial B( L )$, we see that
\begin{equation}\label{daru-4}
P_{1} \otimes P_{2} \Big( \text{Sep}_{L, R, n}, \  F_{L, R, n} \Big) \ge c P_{1} \otimes P_{2} \Big( G_{1}, \ H, \ \eta^{2}_{0, n, \frac{R n}{8}} \big( \gamma [0, \tau^{\gamma}_{\frac{ R n}{8}}] \big) \cap S^{2} [0,  \tau^{2}_{\frac{ R n}{8}}] = \emptyset \Big),
\end{equation}
for some $c > 0$.

Suppose that $\gamma [\tau^{\gamma}_{\frac{ R n}{8}}, \tau^{\gamma}_{R n}] \cap \partial B(n) = \emptyset$. Then $\eta^{2}_{0, n, \frac{R n}{8}} \big( \gamma [0, \tau^{\gamma}_{\frac{ R n}{8}}] \big) \cup \gamma [\tau^{\gamma}_{\frac{ R n}{8}}, \tau^{\gamma}_{\frac{ R n}{4}}] \subset \eta^{2}_{0, n, R n} (\gamma )$. Therefore if we write $\sigma := \max \{ t \ \tau^{\gamma}_{R n} \ | \ \gamma (t ) \in B(n) \}$, then
\begin{align*}
&P_{1} \otimes P_{2} \Big( G_{1}, \ \eta^{2}_{0, n, \frac{R n}{8}} \big( \gamma [0, \tau^{\gamma}_{\frac{ R n}{8}}] \big) \cap S^{2} [0,  \tau^{2}_{\frac{ R n}{8}}] = \emptyset \Big) \\
&\ge P_{1} \otimes P_{2} \Big( \gamma [\tau^{\gamma}_{\frac{ R n}{8}}, \tau^{\gamma}_{R n}] \cap \partial B(n) = \emptyset, \ \eta^{2}_{0, n, R n} (\gamma ) \cap S^{2} [0,  \tau^{2}_{R n}] = \emptyset, \ \gamma [0, \sigma ] \cap S^{2} [ \tau^{2}_{\frac{ R n}{8}}, \tau^{2}_{\frac{ R n}{4}}] = \emptyset \Big) \\
&\ge P_{1} \otimes P_{2} \Big(  \eta^{2}_{0, n, R n} (\gamma ) \cap S^{2} [0,  \tau^{2}_{R n}] = \emptyset, \ \gamma [0, \sigma ] \cap S^{2} [ \tau^{2}_{\frac{ R n}{8}}, \tau^{2}_{\frac{ R n}{4}}] = \emptyset \Big) \\
&- P_{1} \otimes P_{2} \Big( \gamma [\tau^{\gamma}_{\frac{ R n}{8}}, \tau^{\gamma}_{R n}] \cap \partial B(n) \neq \emptyset \Big) \\
&\ge P_{1} \otimes P_{2} \Big(  \eta^{2}_{0, n, R n} (\gamma ) \cap S^{2} [0,  \tau^{2}_{R n}] = \emptyset, \ \gamma [0, \sigma ] \cap S^{2} [ \tau^{2}_{\frac{ R n}{8}}, \tau^{2}_{\frac{ R n}{4}}] = \emptyset \Big) - \frac{C}{R},
\end{align*}
for some $C < \infty$. Here we used Proposition 1.5.10 \cite{Law-book90} in the last inequality. Let 
\begin{equation*}
q := \max \{ k \ | \ \gamma [0, \sigma ] \subset B ( 2^{k} n ) \}.
\end{equation*}
Since $\gamma [0, \sigma ] \subset B ( R n ) $, we have $q \le \log_{2}R +1$. Therefore, 
\begin{align}\label{daru-5}
&P_{1} \otimes P_{2} \Big(  \eta^{2}_{0, n, R n} (\gamma ) \cap S^{2} [0,  \tau^{2}_{R n}] = \emptyset, \ \gamma [0, \sigma ] \cap S^{2} [ \tau^{2}_{\frac{ R n}{8}}, \tau^{2}_{\frac{ R n}{4}}] = \emptyset \Big) - \frac{C}{R} \notag \\
&\ge P_{1} \otimes P_{2} \Big(  F_{L, R, n} \Big) - P_{1} \otimes P_{2} \Big(  \gamma [0, \sigma ] \cap S^{2} [ \tau^{2}_{\frac{ R n}{8}}, \tau^{2}_{\frac{ R n}{4}}] \neq \emptyset \Big) - \frac{C}{R}.
\end{align}
But by Proposition 1.5.10 \cite{Law-book90},
\begin{align*}
&P_{1} \otimes P_{2} \Big(  \gamma [0, \sigma ] \cap S^{2} [ \tau^{2}_{\frac{ R n}{8}}, \tau^{2}_{\frac{ R n}{4}}] \neq \emptyset \Big) \le \sum_{k=1}^{\log_{2}R +1} P_{1} \otimes P_{2} \Big( q=k, \ \gamma [0, \sigma ] \cap S^{2} [ \tau^{2}_{\frac{ R n}{8}}, \tau^{2}_{\frac{ R n}{4}}] \neq \emptyset \Big) \\
&\le \sum_{k=1}^{\log_{2}R +1} P_{1} \otimes P_{2} \Big( q=k, \ B(2^{k} n) \cap S^{2} [ \tau^{2}_{\frac{ R n}{8}}, \tau^{2}_{\frac{ R n}{4}}] \neq \emptyset \Big) \le  \sum_{k=1}^{\log_{2}R +1} C 2^{-k} \frac{2^{k}}{R} \le \frac{C \log R}{R}.
\end{align*}
Combining this with \eqref{daru-4} and \eqref{daru-5}, we have
\begin{equation*}
P_{1} \otimes P_{2} \Big( \text{Sep}_{L, R, n}, \  F_{L, R, n} \Big) \ge c_{1} P_{1} \otimes P_{2} \Big(  F_{L, R, n} \Big) - \frac{C_{1} \log R}{R},
\end{equation*}
for some $c_{1} > 0, C_{1} < \infty$. However, by Corollary 4.2 \cite{Law-cuttimes}, it follows that there exist $c_{2} > 0$ and $\xi \in (\frac{1}{2}, 1)$ such that
\begin{equation*}
P_{1} \otimes P_{2} \Big( F_{L, R, n}  \Big) \ge P_{1} \otimes P_{2} \Big( S^{1} [\tau^{1}_{n}, \tau^{1}_{L} ] \cap S^{2} [0, \tau^{2}_{Rn} ] = \emptyset \Big) \ge c_{2} R^{-\xi },
\end{equation*}
where $\xi$ is referred to as the intersection exponent (see \cite{Law-cuttimes} for $\xi$). Since we know that $\xi < 1$ (see \cite{Law-cuttimes}), there exists $C < \infty $ such that $c_{1} c_{2} R^{-\xi } > \frac{2 C_{1} \log R }{R}$ for all $R \ge C$. Then for all $R \ge C$, we see that $P_{1} \otimes P_{2} \Big( \text{Sep}_{L, R, n} \cap F_{L, R, n} \Big) \ge \frac{c_{1}}{2} P_{1} \otimes P_{2} \Big( F_{L, R, n}  \Big)$, which finishes the proof for $R \ge C$. It is easy to check that the lemma holds for $R \le C$, so we finish the proof of lemma. \end{proof}

\medskip

Once we show Lemma \ref{escape-17}, using the same argument as in the proof of Proposition 6.2.1 \cite{Shi-gr}, we get the following lemma immediately. We shall omit its proof and leave it to the reader.

\begin{lem}\label{escape-25}
Let $d=3$. There exists $C < \infty$ such that for all for all $R \ge 4$, $n \ge 1$ and $Rn \le L \le 4 Rn$, we have 
\begin{equation}\label{escape-26}
\frac{1}{C} P_{1} \otimes P_{2} \Big( F_{L, R, n}  \Big) \le Es (n, L ) \le C P_{1} \otimes P_{2} \Big( F_{L, R, n}  \Big),
\end{equation}
where $F_{L, R, n}$ was defined as in \eqref{escape-14}.
\end{lem}

\subsection{Scaling limit of LERW in three dimensions}\label{scaling limit}
In this subsection, we will review some known facts about the scaling limit of LERW in three dimensions. As we explain in Section 1.1, the scaling limit of LERW for $d=3$ exists \cite{Koz}, and some properties of it were studied in \cite{SS}. We will explain the details here.

Let $D = \{ x \in \mathbb{R}^{3} \ | \ |x| < 1 \}$ and $\overline{D}$ be its closure.  Let 
\begin{equation}\label{rescale}
\text{LEW}_{n} =\frac{LE ( S[0, \tau_{n}] )}{n}. 
\end{equation}
Here $S$ is a simple random walk started at the origin on $\mathbb{Z}^{3}$ and $\tau_{n} = \inf \{ t \ | \ S(t) \in \partial B (n) \}$. 

\medskip

We write ${\cal H} (\overline{D} )$ for the metric space of the set of compact subsets in $\overline{D}$ with the Hausdorff distance $d_{\text{H}}$. Thinking of $\text{LEW}_{n}$ as random elements of ${\cal H} (\overline{D} )$, let $P^{(n)}$ be the probability measure on ${\cal H} (\overline{D} )$ induced by $\text{LEW}_{n}$. Then \cite{Koz} shows that $P^{(2^{j})}$ is Cauchy with respect to the weak convergence topology, and therefore $P^{(2^{j})}$ converges weakly. Let $\nu$ be its limit probability measure. We call $\nu$ the scaling limit measure of LERW in three dimensions. We write ${\cal K}$ for the random compact subset associated with $\nu$. We call ${\cal K}$ the scaling limit of LERW in three dimensions. It is also shown in \cite{Koz} that ${\cal K}$ is invariant under rotations and dilations. 

\medskip

Some properties of ${\cal K}$ were studied in \cite{SS}. In \cite{SS}, it is shown that ${\cal K}$ is a simple path almost surely (Theorem 1.2 \cite{SS}). Furthermore, if we let $Y$ be the union of ${\cal K}$ and loops from independent Brownian loop soup in $D$ which intersect ${\cal K}$, more precisely,
\begin{equation}\label{decomp}
Y := {\cal K} \cup \{ \ell \in \text{BS} \ | \ \ell \cap {\cal K} \neq \emptyset \},
\end{equation}
then $Y$ has the same distribution in ${\cal H} (\overline{D} )$ as the trace of three dimensional Brownian motion up to exiting from $D$ (Theorem 1.1 \cite{SS}). Here $\text{BS}$ is the Brownian loop soup in $D$ which is independent of ${\cal K}$ (see \cite{LW} for the Brownian loop soup).

\medskip

We denote the Hausdorff dimension by $\text{dim}_{\text{H}} (\cdot )$. Bounds of $\text{dim}_{\text{H}} ({\cal K})$ were given in Theorem 1.4 \cite{SS} as follows. Let $\xi$ be the intersection exponent for three dimensional Brownian motion (see \cite{Law-haus} for $\xi$). Let $\beta = 2- \alpha$, where $\alpha$ is the exponent as in Theorem \ref{escape-5}. Then Theorem 1.4 \cite{SS} shows that 
\begin{equation}\label{haus-bound}
2- \xi \le \text{dim}_{\text{H}} ({\cal K}) \le \beta, \text{ almost surely}.
\end{equation}
In particular, since $\xi \in (\frac{1}{2}, 1)$ (see \cite{Law-haus}) and $\beta \in (1, \frac{5}{3}]$ (see \cite{Law LERW99}), we have
\begin{equation}\label{haus-bound-1}
1 < \text{dim}_{\text{H}} ({\cal K}) \le \frac{5}{3}, \text{ almost surely}.
\end{equation}

\medskip

The main purpose of the present paper is to show that 
\begin{equation}\label{haus-bound-2}
 \text{dim}_{\text{H}} ({\cal K}) \ge \beta, \text{ almost surely},
\end{equation}
which concludes that $\text{dim}_{\text{H}} ({\cal K}) = \beta$ almost surely.

 \section{The number of small boxes hit by ${\cal K}$} 
 From here to the end of the present paper, we will assume $d=3$. In this section, we will give bounds of the number of small boxes hit by ${\cal K}$. To do it, we will first estimate the probability that ${\cal K}$ hits two distinct small boxes (see Theorem \ref{key thm}), which is one of the key result in the paper. We will show Theorem \ref{key thm} in Section 3.1. Then using the second moment method, we will give some bounds of the number of boxes hit by ${\cal K}$ in Section 3.2. 
 
 \subsection{Probability of ${\cal K}$ hitting two small boxes}
 Recall that $D = \{ x \in \mathbb{R}^{3} \ | \ |x| < 1 \}$ and $\overline{D}$ is its closure. For $ r > 0$, we write $D_{r} = \{ x \in \mathbb{R}^{3} \ | \ |x| < r \}$ and let $\overline{D_{r}}$ be its closure. For $x = (x_{1}, x_{2}, x_{3}) \in \mathbb{Z}^{3}$, let
\begin{equation}\label{box}
B_{x} = \prod_{i=1}^{3} [x_{i}, x_{i}+1]
\end{equation}
In this subsection, we will establish an upper bound of the probability that ${\cal K}$ hits both $\epsilon B_{x}$ and $\epsilon B_{y}$ with $\frac{1}{3} \le |\epsilon x |, |\epsilon y | \le \frac{2}{3}$ and $x, y \in \mathbb{Z}^{3}$ (see Theorem \ref{key thm}). The upper bound will be given in terms of escape probabilities defined in Section 2.2. In the proof of Theorem \ref{key thm}, we will repeatedly use several properties of escape probabilities explained in Section 2.2 as well as Proposition 4.2, 4.4 and 4.6 in \cite{Mas}.

\medskip

Let $\text{LEW}_{n} =\frac{LE ( S[0, \tau_{n}] )}{n}$. Here $S$ is a simple random walk started at the origin on $\mathbb{Z}^{3}$ and $\tau_{n} = \inf \{ t \ | \ S(t) \in \partial B (n) \}$.  Since $\text{LEW}_{2^{j}}$ converges weakly to ${\cal K}$ (see Section \ref{scaling limit}), we can define $\{ \text{LEW}_{2^{j}} \}_{j \ge 1}$ and ${\cal K}$ on the same probability space $(\Omega, {\cal F}, P)$ such that 
\begin{equation}\label{skorohod}
\lim_{j \to \infty} d_{\text{H}} ( \text{LEW}_{2^{j}}, {\cal K} ) = 0, \ P \text{-almost surely, }
\end{equation}
where $d_{\text{H}}$ is the Hausdorff metric on ${\cal H} (\overline{D} )$ (see Section \ref{scaling limit} for ${\cal H} (\overline{D} )$).
 
Take $\epsilon > 0$. By \eqref{skorohod}, for $P$-a.s., $\omega$, there exists $N_{\epsilon} ( \omega ) < \infty $ such that 
\begin{equation*}
d_{\text{H}} ( \text{LEW}_{2^{j}}, {\cal K} ) < \epsilon^{2}, \text{ for all } j \ge N_{\epsilon}.
\end{equation*}
Since $P ( N_{\epsilon} < \infty ) = 1$, there exists $j_{\epsilon}$ such that 
\begin{equation}\label{skorohod-2}
P ( N_{\epsilon} < j_{\epsilon} ) \ge 1- \epsilon^{100}.
\end{equation}
On the event $\{ N_{\epsilon} < j_{\epsilon} \}$, if we write $n_{\epsilon} := 2^{j_{\epsilon }}$, then 
\begin{equation}\label{skorohod-3}
d_{\text{H}} ( \text{LEW}_{n_{\epsilon}}, {\cal K} ) < \epsilon^{2}.
\end{equation}
From now on, we fix $n=n_{\epsilon} = 2^{j_{\epsilon }}$ for each $\epsilon > 0$ such that \eqref{skorohod-2} holds.

\medskip

One of the key results in this paper is the following theorem.

\begin{thm}\label{key thm}
Fix $\epsilon > 0$ and take $n=n_{\epsilon} = 2^{j_{\epsilon }}$ such that \eqref{skorohod-2} holds. Suppose that $x \neq y \in \mathbb{Z}^{3}$ satisfy 
\begin{equation}\label{key thm-1}
\epsilon B_{x} \subset D_{\frac{2}{3}} \setminus D_{\frac{1}{3}} \text{ and } \epsilon B_{y} \subset D_{\frac{2}{3}} \setminus D_{\frac{1}{3}}.
\end{equation}
Let $l := |x-y |$. Then there exists an absolute constant $C < \infty$ such that
\begin{equation}\label{key-2}
P \big( {\cal K} \cap \epsilon B_{x} \neq \emptyset \text{ and } {\cal K} \cap \epsilon B_{y} \neq \emptyset \big) \le C \text{Es} ( \epsilon n , l \epsilon n) \text{Es} ( \epsilon n , n )  \frac{\epsilon}{l}.
\end{equation}
\end{thm}

\medskip

\begin{rem}\label{polish-rem-1}
Since the proof of Theorem \ref{key thm} is quite long, we explain some of its ideas here. Take $n=n_{\epsilon} = 2^{j_{\epsilon }}$ such that \eqref{skorohod-2} holds. Since $\epsilon^{100} \ll \text{Es} ( \epsilon n , l \epsilon n) \text{Es} ( \epsilon n , n )  \frac{\epsilon}{l}$, we may suppose that $d_{\text{H}} ( \text{LEW}_{n}, {\cal K} ) < \epsilon^{2}$. In that case if ${\cal K}$ hits both $\epsilon B_{x}$ and $\epsilon B_{y}$, then $\gamma := LE ( S[0, \tau_{n}] )$ hits both $\epsilon n B'_{x}$ and $\epsilon n B'_{y}$, where $B'_{z} = \prod_{i=1}^{3} [z_{i} -2, z_{i} + 2]$. So we need to estimate
\begin{equation}\label{intui}
P \Big( \gamma \cap \epsilon n B'_{x} \neq \emptyset, \ \gamma \cap \epsilon n B'_{y} \neq \emptyset \Big) \le P \Big( \tau^{\gamma, x} < \tau^{\gamma, y} < \infty \Big) + P \Big( \tau^{\gamma, y} < \tau^{\gamma, x} < \infty \Big).
\end{equation}
Here $\tau^{\gamma, z} := \inf \{ t \ | \ \gamma (t) \in \epsilon n B'_{z}  \}$. We want to show that 
\begin{equation}\label{intui-1}
P \Big( \tau^{\gamma, x} < \tau^{\gamma, y} < \infty \Big) \le C P \Big( \tau^{\gamma, x} <  \infty \Big) P^{\epsilon n x} \Big( \tau^{\gamma, y} < \infty \Big).
\end{equation}
Note that if $\gamma $ were $S[0, \tau_{n}]$, \eqref{intui-1} would hold because of the strong Markov property. However, since $\gamma = LE ( S[0, \tau_{n}] )$ is not a Markov process and the distribution of $\gamma [\tau^{\gamma, x} , \tau^{\gamma, y} ]$ strongly depends on the shape of $\gamma [0, \tau^{\gamma, x}]$. We need to control such dependence and this will be done in Lemma \ref{hosoku-2}, Lemma \ref{key-21} and \ref{hosoku-3}. Then we will prove \eqref{intui-1}. Once \eqref{intui-1} is proved then Theorem \ref{key thm} immediately follows because 
\begin{equation}\label{intui-2}
P \Big( \tau^{\gamma, x} <  \infty \Big) P^{\epsilon n x} \Big( \tau^{\gamma, y} < \infty \Big) \le  C \text{Es} ( \epsilon n , l \epsilon n) \text{Es} ( \epsilon n , n )  \frac{\epsilon}{l},
\end{equation}
and the second probability in RHS of \eqref{intui} can be estimated similarly.
\end{rem}

\medskip

We will split the proof of Theorem \ref{key thm} as follows. Since we want to estimate the probability in LHS of \eqref{intui-1} in terms of escape probabilities, we first rewrite the probability in terms of independent random walks by reversing paths in Lemma \ref{hosoku-1}. Such independent random walks consist of three walks $S^{1}, \cdots, S^{3}$ with $S^{1} (0)=  S^{2} (0) \in \partial (\epsilon n B'_{x})$ and $S^{3}(0) \in \partial (\epsilon n B'_{y})$ (see Figure 1). In order for $\gamma = LE ( S[0, \tau_{n}] )$ to hit $\epsilon n B'_{x}$, the loop erasure of the time reverse of $S^{1}$, say $(S^{1})^{R}$, does not intersect a composition of two walks $S^{2}$ and $S^{3}$. In addition, in order for $\gamma$ to hit $\epsilon n B'_{y}$, the loop erasure of a composition of two walks $(S^{1})^{R}$ and $S^{2}$ does not intersect $S^{3}$. To control the independence, we will replace the latter event by events that the loop erasure of $(S^{2})^{R}$ up to some stopping time does not intersect $S^{3}$ in Lemma \ref{hosoku-2} ($S^{4}$ corresponds to $(S^{2})^{R}$ up to that stopping time in Lemma \ref{hosoku-2}). The distribution of the loop erasure of $(S^{2})^{R}$ up to the stopping time will be studied in Lemma \ref{key-21}, which allows us to think that the latter event is independent from the former one, and to estimate the probability of the latter event in terms of escape probabilities. Finally in Lemma \ref{hosoku-3} we will estimate the probability of the former event using escape probabilities, and then prove Theorem \ref{key thm}.

\begin{figure}
\begin{center}

\includegraphics[width=8.5cm]{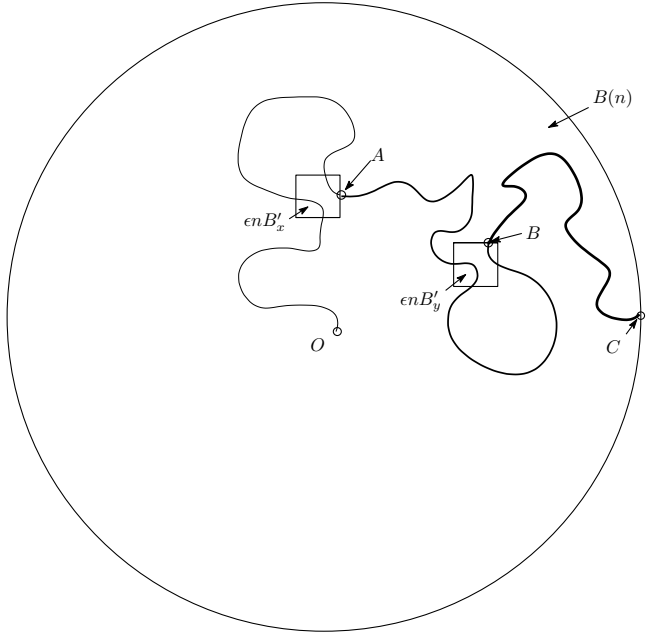}

\end{center}

\caption{A simple path from $O$ to $C$ stands for $\gamma = LE ( S[0, \tau_{n}] )$. Points $A$ and $B$ stand for last visits of $\gamma$ to $\epsilon n B'_{x}$ and $\epsilon n B'_{y}$. Then $A=S^{1} (0)=  S^{2} (0)$ and $B=S^{3}(0)$. We let $S^{1}$ run until it hits $O$, let $S^{2}$ run until it hits $B$, and let $S^{3}$ run until it hits $C$. The simple path from $O$ to $A$ corresponds to the loop erasure of the time reverse of $S^{1}$, say $(S^{1})^{R}$. The simple path from $O$ to $B$ corresponds to the loop erasure of $(S^{1})^{R} + S^{2}$. Finally, $\gamma$ corresponds to the loop erasure of $(S^{1})^{R} + S^{2} + S^{3}$. }

\end{figure}

\medskip

\begin{proof}
It suffices to show \eqref{key-2} for $l \ge 10^{6}$. Indeed, $\text{Es} ( \epsilon n, l \epsilon n ) \ge c$ for $l \le 10^{6}$ and we already showed that 
\begin{equation*}
P ( {\cal K} \cap \epsilon B_{x} \neq \emptyset ) \le C \text{Es} ( \epsilon n, n ) \epsilon.
\end{equation*}
(See the proof of Lemma 7.1 \cite{SS} for this inequality.)  Therefore, 
\begin{equation*}
P \big( {\cal K} \cap \epsilon B_{x} \neq \emptyset \text{ and } {\cal K} \cap \epsilon B_{y} \neq \emptyset \big) \le C \text{Es} ( \epsilon n, n ) \epsilon \le C \text{Es} ( \epsilon n , l \epsilon n) \text{Es} ( \epsilon n , n )  \frac{\epsilon}{l},
\end{equation*}
for $l \le 10^{6}$.

Thus we may assume that $10^{6} \le l \le \frac{2}{\epsilon}$. Note that by \eqref{escape-7},
\begin{equation*}
\epsilon^{100} \le C \text{Es} ( \epsilon n , l \epsilon n) \text{Es} ( \epsilon n , n )  \frac{\epsilon}{l}.
\end{equation*}
So by \eqref{skorohod-2}, 
\begin{align*}
&P \big( {\cal K} \cap \epsilon B_{x} \neq \emptyset, \  {\cal K} \cap \epsilon B_{y} \neq \emptyset, \ d_{\text{H}} ( \text{LEW}_{n}, {\cal K} ) \ge \epsilon^{2} \big) \le P \big( {\cal K} \cap \epsilon B_{x} \neq \emptyset, \  {\cal K} \cap \epsilon B_{y} \neq \emptyset, \ N_{\epsilon} \ge j_{\epsilon} \big) \\
&\le \epsilon^{100} \le C \text{Es} ( \epsilon n , l \epsilon n) \text{Es} ( \epsilon n , n )  \frac{\epsilon}{l}.
\end{align*}
Therefore it suffices to show that
\begin{equation}\label{key-3}
P \big( {\cal K} \cap \epsilon B_{x} \neq \emptyset, \  {\cal K} \cap \epsilon B_{y} \neq \emptyset, \ d_{\text{H}} ( \text{LEW}_{n}, {\cal K} ) < \epsilon^{2} \big) \le C \text{Es} ( \epsilon n , l \epsilon n) \text{Es} ( \epsilon n , n )  \frac{\epsilon}{l}.
\end{equation}

Suppose that ${\cal K} \cap \epsilon B_{x} \neq \emptyset, \  {\cal K} \cap \epsilon B_{y} \neq \emptyset, \ d_{\text{H}} ( \text{LEW}_{n}, {\cal K} ) < \epsilon^{2}$. Let $B'_{x} = \prod_{i=1}^{3} [x_{i} -2, x_{i} + 2]$ and $B'_{y} = \prod_{i=1}^{3} [y_{i} -2, y_{i} + 2]$. Then 
\begin{equation}\label{key-4}
\text{LEW}_{n} \cap \epsilon B'_{x} \neq \emptyset, \  \text{LEW}_{n} \cap \epsilon B'_{y} \neq \emptyset.
\end{equation}
So we have to estimate
\begin{equation*}
P \big( \text{LEW}_{n} \cap \epsilon B'_{x} \neq \emptyset, \  \text{LEW}_{n} \cap \epsilon B'_{y} \neq \emptyset \big) = P \big( LE ( S[0, \tau_{n} ] ) \cap \epsilon n B'_{x} \neq \emptyset, \  LE ( S[0, \tau_{n} ] ) \cap \epsilon n B'_{y} \neq \emptyset \big).
\end{equation*}
Suppose that $LE ( S[0, \tau_{n} ] ) \cap \epsilon n B'_{x} \neq \emptyset, \  LE ( S[0, \tau_{n} ] ) \cap \epsilon n B'_{y} \neq \emptyset$. Then clearly $S[0, \tau_{n} ] \cap \epsilon n B'_{x} \neq \emptyset, \ S[0, \tau_{n} ]  \cap \epsilon n B'_{y} \neq \emptyset$. So we may define
\begin{equation*}
T^{x} = \max \{ t \le \tau_{n} \ | \ S(t) \in \partial (\epsilon n B'_{x}) \}, \  T^{y} = \max \{ t \le \tau_{n} \ | \ S(t) \in \partial  (\epsilon n B'_{y}) \}.
\end{equation*}
Then
\begin{align}\label{key-5}
&P \big( LE ( S[0, \tau_{n} ] ) \cap \epsilon n B'_{x} \neq \emptyset, \  LE ( S[0, \tau_{n} ] ) \cap \epsilon n B'_{y} \neq \emptyset \big) \notag \\
&\le P \big( LE ( S[0, \tau_{n} ] ) \cap \epsilon n B'_{x} \neq \emptyset, \  LE ( S[0, \tau_{n} ] ) \cap \epsilon n B'_{y} \neq \emptyset, T^{x} < T^{y} \big) \notag  \\
&+ P \big( LE ( S[0, \tau_{n} ] ) \cap \epsilon n B'_{x} \neq \emptyset, \  LE ( S[0, \tau_{n} ] ) \cap \epsilon n B'_{y} \neq \emptyset, T^{x} > T^{y} \big)
\end{align}
We will deal with only the first probability in the right hand side of \eqref{key-5}. The second probability can be estimated similarly. 

\medskip

Define
\begin{equation}\label{mendoi-1}
\sigma^{i}_{z} = \max \{ t \le \tau^{1}_{n} \ | \ S^{i} (t) = z \},
\end{equation}
and
\begin{align}\label{mendoi-2}
&\overline{\sigma}_{1} := \inf \{ t \ | \ LE ( S^{1} [0, \sigma^{1}_{z} ] ) (t) \in \partial (\epsilon n B'_{x}) \}, \notag \\
&\overline{\sigma}_{2}:= \inf \{ t \ | \  LE ( S^{1} [0, \sigma^{1}_{z} ] + S^{2} [0, \sigma^{2}_{w}] ) (t) \in \partial (\epsilon n B'_{y}) \}.
\end{align}

\medskip

The estimate of the first probability in the right hand side of \eqref{key-5} will be carried out below, but it is quite long. So we will split it into shorter claims (Lemma \ref{hosoku-1}, \ref{hosoku-2}, \ref{key-21}, and \ref{hosoku-3}). 

\medskip

In order to estimate the first probability in the right hand side of \eqref{key-5} in terms of the escape probabilities, we need to decompose the simple random walk path into three parts; $S$ from the origin to the $\epsilon n$ cube around $x$, $S$ from the $\epsilon n$ cube around $x$ to the cube around $y$, and $S$ from the cube around $y$ to the boundary of $B(n)$. By using a standard technique called ``last exit decomposition" (see Proposition 2.4.1 \cite{Law-book90} for details), Lemma \ref{hosoku-1} below deals with this decomposition. In the Lemma \ref{hosoku-1}, these three parts in the decomposition correspond to $S^{1}$, $S^{2}$, and $S^{3}$ respectively.

\medskip

\begin{lem}\label{hosoku-1}
There exists a $C < \infty$ such that 
\begin{align}\label{hosoku-1-1}
&P \big( LE ( S[0, \tau_{n} ] ) \cap \epsilon n B'_{x} \neq \emptyset, \  LE ( S[0, \tau_{n} ] ) \cap \epsilon n B'_{y} \neq \emptyset, T^{x} < T^{y} \big) \notag \\
&\le C \sum_{z \in \partial (\epsilon n B'_{x})} \sum_{w \in \partial (\epsilon n B'_{y})}  P^{0}_{1} \otimes P^{z}_{2} \otimes P^{w}_{3} \Big( \sigma^{1}_{z} < \tau^{1}_{n}, \  \sigma^{2}_{w} < \tau^{2}_{n}, \notag \\
&\ \ \ \ \ \  S^{2} [1, \sigma^{2}_{w} ] \cap \overline{(\epsilon n B'_{x})} = \emptyset, \ S^{3} [ 1, \tau^{3}_{n} ] \cap  \overline{(\epsilon n B'_{y})}  = \emptyset \notag \\
& \ \ \ \  \  LE ( S^{1} [0, \sigma^{1}_{z} ] ) [0, \overline{\sigma}_{1}] \cap ( S^{2} [ 0 , \sigma^{2}_{w} ] \cup S^{3} [0, \tau^{3}_{n}]) = \emptyset, \ LE ( S^{1} [0, \sigma^{1}_{z} ] + S^{2} [0, \sigma^{2}_{w}] ) [0, \overline{\sigma}_{2}] \cap S^{3} [ 0 , T^{3}_{w, \frac{ \epsilon l n}{4}} ] = \emptyset  \Big).
\end{align}
\end{lem}

\begin{proof}
 Suppose that $LE ( S[0, \tau_{n} ] ) \cap \epsilon n B'_{x} \neq \emptyset$. Let
\begin{equation}\label{key-6}
\sigma_{1}' = \inf \{ t \ | \ LE ( S[0, \tau_{n} ] ) (t) \in \partial (\epsilon n B'_{x}) \}, \ \overline{\sigma}_{1}' = \inf \{ t \ | \ LE ( S[0, T^{x} ] ) (t) \in \partial (\epsilon n B'_{x}) \}.
\end{equation}

Note that 
\begin{equation}\label{key-7}
LE ( S[0, T^{x} ] ) [0, \overline{\sigma}'_{1}] \cap S[ T^{x}, \tau_{n} ] = \emptyset.
\end{equation}
To see this, we let $LE^{(1)} = LE_{1} ( \lambda_{1}, \lambda_{2} )$ and $LE^{(2)} = LE_{2} ( \lambda_{1}, \lambda_{2} )$ where $\lambda_{1} = LE ( S[0, T^{x} ] )$ and $\lambda_{2} = S[ T^{x}, \tau_{n} ] $ (see \eqref{comp-1} for $LE^{(i)}$). Then $LE ( S[0, \tau_{n} ] ) = LE^{(1)} + LE^{(2)}$. Let $u = \inf \{ t \ | \ \lambda_{1} (t) \in \lambda_{2} \}$ and $s = \sup \{ \lambda_{2} (t) = \lambda_{1} (u) \}$. Then $LE^{(1)} = \lambda_{1} [0, u]$ and $LE^{(2)} = LE ( \lambda_{2} [s,  \text{len} \lambda_{2} ]  )$. If $LE ( S[0, T^{x} ] ) [0, \overline{\sigma}'_{1}] \cap S[ T^{x}, \tau_{n} ] \neq \emptyset$, then $u \le \overline{\sigma}'_{1}$. By definition of $\overline{\sigma}'_{1}$, this implies that $LE ( S[0, T^{x} ] ) [0, u] \cap \epsilon n B'_{x} = \emptyset$. Moreover, since $S[ T^{x}, \tau_{n} ] \cap \epsilon n B'_{x} = \emptyset$, we see that $LE^{(2)} \cap \epsilon n B'_{x} = \emptyset$. This implies that $LE ( S[0, \tau_{n} ] ) \cap \epsilon n B'_{x} = \emptyset$ and we get a contradiction. Therefore \eqref{key-7} holds and $\overline{\sigma}'_{1} < u$. Thus $LE ( S[0, \tau_{n} ] ) [0, \overline{\sigma}'_{1} ] = LE ( S[0, T^{x} ] ) [0, \overline{\sigma}'_{1} ] $ and $\sigma'_{1} = \overline{\sigma}'_{1}$.

Thus,
\begin{align}\label{key-8}
& P \big( LE ( S[0, \tau_{n} ] ) \cap \epsilon n B'_{x} \neq \emptyset, \  LE ( S[0, \tau_{n} ] ) \cap \epsilon n B'_{y} \neq \emptyset, T^{x} < T^{y} \big) \notag \\
&\le P \big( LE ( S[0, T^{x} ] ) [0, \overline{\sigma}'_{1}] \cap S[ T^{x}, \tau_{n} ] = \emptyset, \ LE ( S[0, T^{y} ] ) [0, \overline{\sigma}'_{2}] \cap S[ T^{y}, \tau_{n} ] = \emptyset, \ T^{x} < T^{y} < \tau_{n} \big),
\end{align}
where $\overline{\sigma}'_{2} = \inf \{ t \ | \ LE ( S[0, T^{y} ] ) (t) \in \partial (\epsilon n B'_{y}) \}$.

Next we will decompose $S[0, \tau_{n}]$ into three parts, $S[0, T^{x}]$, $S[T^{x}, T^{y}]$ and $S[T^{y}, \tau_{n}]$, using a standard technique called ``last exit decomposition" (see Proposition 2.4.1 \cite{Law-book90} for details). Note that by the Markov property at time $k_{1}$ and $k_{1} + k_{2}$, we have
\begin{align}\label{key-9}
&P \big( LE ( S[0, T^{x} ] ) [0, \overline{\sigma}'_{1}] \cap S[ T^{x}, \tau_{n} ] = \emptyset, \ LE ( S[0, T^{y} ] ) [0, \overline{\sigma}'_{2}] \cap S[ T^{y}, \tau_{n} ] = \emptyset, \ T^{x} < T^{y} < \tau_{n} \big) \notag \\
&= \sum_{k_{1} > 0} \sum_{k_{2} > 0} \sum_{z \in \partial (\epsilon n B'_{x})} \sum_{w \in \partial (\epsilon n B'_{y})} \notag \\
& \ \ \ \ P \Big( T^{x} =k_{1}, \ S(k_{1}) = z, \ T^{y} = k_{1} + k_{2}, \ S(k_{1} + k_{2} ) = w, \ k_{1} + k_{2} < \tau_{n} \notag \\
& \ \ \ \ \  \ \ \ LE ( S[0, T^{x} ] ) [0, \overline{\sigma}'_{1}] \cap S[ T^{x}, \tau_{n} ] = \emptyset, \ LE ( S[0, T^{y} ] ) [0, \overline{\sigma}'_{2}] \cap S[ T^{y}, \tau_{n} ] = \emptyset  \Big) \notag \\
&= \sum_{k_{1} > 0} \sum_{k_{2} > 0} \sum_{z \in \partial (\epsilon n B'_{x})} \sum_{w \in \partial (\epsilon n B'_{y})} \notag \\
& \ \ \ \ P^{0}_{1} \otimes P^{z}_{2} \Big( S^{1} (k_{1}) = z, \ k_{1} < \tau^{1}_{n}, \ S^{2} ( k_{2} ) = w, \ k_{2} < \tau^{2}_{n} \notag \\
& \ \ \ \  \ \ \ \ S^{2} [1, \tau^{2}_{n} ] \cap \overline{(\epsilon n B'_{x})} = \emptyset, \ S^{2} [k_{2} + 1, \tau^{2}_{n} ] \cap \overline{(\epsilon n B'_{y})} = \emptyset \notag \\
& \ \ \ \  \ \ \ \ LE ( S^{1} [0, k_{1} ] ) [0, \overline{\sigma}_{1}] \cap S^{2} [ 0 , \tau^{2}_{n} ] = \emptyset, \ LE ( S^{1} [0, k_{1} ] + S^{2} [0, k_{2}] ) [0, \overline{\sigma}_{2}] \cap S^{2} [ k_{2} , \tau^{2}_{n} ] = \emptyset  \Big) \notag \\
& = \sum_{z \in \partial (\epsilon n B'_{x})} \sum_{w \in \partial (\epsilon n B'_{y})} \frac{1}{p_{z}} \frac{1}{p_{w}} P^{0}_{1} \otimes P^{z}_{2} \otimes P^{w}_{3} \Big( \sigma^{1}_{z} < \tau^{1}_{n}, \  \sigma^{2}_{w} < \tau^{2}_{n}, \notag \\ 
& \ \ \ \  \ S^{2} [1, \sigma^{2}_{w} ] \cap \overline{(\epsilon n B'_{x})} = \emptyset, \ S^{3} [ 1, \tau^{3}_{n} ] \cap (\overline{(\epsilon n B'_{x})} \cup \overline{(\epsilon n B'_{y})} ) = \emptyset \notag \\
& \ \ \ \  \  LE ( S^{1} [0, \sigma^{1}_{z} ] ) [0, \overline{\sigma}_{1}] \cap ( S^{2} [ 0 , \sigma^{2}_{w} ] \cup S^{3} [0, \tau^{3}_{n}]) = \emptyset, \ LE ( S^{1} [0, \sigma^{1}_{z} ] + S^{2} [0, \sigma^{2}_{w}] ) [0, \overline{\sigma}_{2}] \cap S^{3} [ 0 , \tau^{3}_{n} ] = \emptyset  \Big) \notag \\
&\ \ \ (\text{where } p_{z} = P^{z} ( z \notin  S [1, \tau_{n} ] ) )
\end{align}
which finishes the proof of Lemma \ref{hosoku-1}. \end{proof}

\medskip

\begin{rem}\label{polish-rem}
There are six events in the probability in the right hand side of \eqref{hosoku-1-1}. We want to say they are ``independent up to constant". Namely, we will show that the probability in RHS of \eqref{hosoku-1-1} is comparable to the product of six probabilities coming from each of six events. Then we need to estimate each of those probabilities. The first four events are easy to estimate. The fifth event corresponds to the probability that the loop erasure of a random walk from the $\epsilon n$ cube around $x$ to the origin does not intersect a random walk from the cube around $x$ to the boundary of $B(n)$. This probability is comparable to $Es ( \epsilon n, n )$. Similarly we will see that the sixth events corresponds to $Es ( \epsilon n, \epsilon l n )$.

\end{rem}

\medskip

With the strategy in Remark \ref{polish-rem} in mind, we introduce some notation before going to the next lemma.

We write 
\begin{align}\label{key-10}
&F^{1} := \Big\{ \sigma^{1}_{z} < \tau^{1}_{n}, \  \sigma^{2}_{w} < \tau^{2}_{n}, \ S^{2} [1, \sigma^{2}_{w} ] \cap \overline{(\epsilon n B'_{x})} = \emptyset, \ S^{3} [ 1, \tau^{3}_{n} ] \cap  \overline{(\epsilon n B'_{y})}  = \emptyset \notag \\ 
& \ \   LE ( S^{1} [0, \sigma^{1}_{z} ] ) [0, \overline{\sigma}_{1}] \cap ( S^{2} [ 0 , \sigma^{2}_{w} ] \cup S^{3} [0, \tau^{3}_{n}]) = \emptyset, \ LE ( S^{1} [0, \sigma^{1}_{z} ] + S^{2} [0, \sigma^{2}_{w}] ) [0, \overline{\sigma}_{2}] \cap S^{3} [ 0 , \tau^{3}_{n} ] = \emptyset  \Big\}.
\end{align}

By Lemma \ref{hosoku-1}, we have to estimate $P^{0}_{1} \otimes P^{z}_{2} \otimes P^{w}_{3} (F^{1})$. To do so, define $A^{r}_{z}= B(z, 2^{r} \epsilon n ) \setminus B ( z, 2^{r-1} \epsilon n )$ for $r \ge 1$ and $A^{0}_{r}= B( z, \epsilon n)$. Let $u^{(1)} = \text{len} LE ( S^{1} [0, \sigma^{1}_{z} ] )$ and 
\begin{equation}\label{q1}
q^{(1)} = \max \{ r \ge 0 \ | \ LE ( S^{1} [0, \sigma^{1}_{z} ] ) [\overline{\sigma}_{1}, u^{(1)} ] \cap A^{r}_{z} \neq \emptyset \}.
\end{equation}
($q^{(1)}$ is well-defined because $LE ( S^{1} [0, \sigma^{1}_{z} ] ) [\overline{\sigma}_{1}, u^{(1)} ] \cap A^{0}_{r} \neq \emptyset$.) We will first deal with the case of $q^{(1)} \le \log_{2} l -3$ so that $2^{q^{(1)}} \epsilon n \le 2^{-3} l \epsilon n$. So suppose that $q^{(1)} = r \le \log_{2} l -3$. Let
\begin{equation}\label{key-11}
T^{2}_{z, \frac{l \epsilon n}{2}} = \inf \{ t \ | \ S^{2} (t) \in \partial B (z, \frac{l \epsilon n}{2} ) \}.
\end{equation}
Then by the strong Markov property for $S^{2}$ at $T^{2}_{z, \frac{l \epsilon n}{2}}$,
\begin{align}\label{key-12}
&P^{0}_{1} \otimes P^{z}_{2} \otimes P^{w}_{3} (F^{1}, \ q^{(1)}= r) \notag \\
&=\sum_{z' \in \partial B (z, \frac{l \epsilon n}{2} ) } P^{0}_{1} \otimes P^{z}_{2} \otimes P^{w}_{3} (F^{1}, \ q^{(1)}= r, \ S^{2} ( T^{2}_{z, \frac{l \epsilon n}{2}} ) = z' ) \notag \\
&= \sum_{z' \in \partial B (z, \frac{l \epsilon n}{2} ) } P^{0}_{1} \otimes P^{z}_{2} \otimes P^{z'}_{4} \otimes P^{w}_{3}  \notag \\
& \ \ \ \ \ \Big(  \sigma^{1}_{z} < \tau^{1}_{n}, \  \sigma^{4}_{w} < \tau^{4}_{n}, \ (S^{2} [1, T^{2}_{z, \frac{l \epsilon n}{2}} ] \cup S^{4}[0, \sigma^{4}_{w}]) \cap \overline{(\epsilon n B'_{x})} = \emptyset, \ S^{3} [ 1, \tau^{3}_{n} ] \cap  \overline{(\epsilon n B'_{y})}  = \emptyset \notag \\
& \ \ \ \ \ \ \ LE ( S^{1} [0, \sigma^{1}_{z} ] ) [0, \overline{\sigma}_{1}] \cap ( S^{2} [ 0 , T^{2}_{z, \frac{l \epsilon n}{2}} ] \cup S^{4}[0, \sigma^{4}_{w}] \cup S^{3} [0, \tau^{3}_{n}]) = \emptyset \notag \\
& \ \ \ \ \ \ \ LE ( S^{1} [0, \sigma^{1}_{z} ] + S^{2} [0, T^{2}_{z, \frac{l \epsilon n}{2}}] + S^{4}[0, \sigma^{4}_{w}] ) [0, \overline{\sigma}_{2}] \cap S^{3} [ 0 , \tau^{3}_{n} ] = \emptyset, \ q^{(1)}= r, \ S^{2} ( T^{2}_{z, \frac{l \epsilon n}{2}} ) = z' \Big).
\end{align}
We define an event $F^{2}$ by
\begin{align}\label{key-13}
&F^{2} = \Big\{ \sigma^{1}_{z} < \tau^{1}_{n}, \ S^{2} [1, T^{2}_{z, \frac{l \epsilon n}{2}} ] \cap \overline{(\epsilon n B'_{x})} = \emptyset, \notag \\
& \ \ \ \ \  \ \ \ \ LE ( S^{1} [0, \sigma^{1}_{z} ] ) [0, \overline{\sigma}_{1}] \cap  S^{2} [ 0 , T^{2}_{z, \frac{l \epsilon n}{2}} ] = \emptyset, \ q^{(1)}= r, \ S^{2} ( T^{2}_{z, \frac{l \epsilon n}{2}} ) = z' \Big\}.
\end{align}
Define a sequence of stopping times $T_{i}$ by $T_{0} = 0$ and 
\begin{align}\label{ex1}
&T_{2i+1} = \inf \{ t \ge T_{2i} \ | \ S^{4} (t) \in \partial B ( w, \frac{ \epsilon l n}{4} ) \} \notag \\
&T_{2i} = \inf \{ t \ge T_{2i-1} \ | \ S^{4} (t) \in \partial B ( w, \frac{ \epsilon l n}{800} ) \}.
\end{align} 
Let 
\begin{align}\label{darui}
&u'_{4} =u'_{4, i} := \inf \{ t \ | \ LE ( S^{4} [0, T_{2i + 1} ] ) (t) \in \partial B ( w, \frac{\epsilon l n}{1600} ) \} \notag \\
&\overline{\sigma}^{\star }_{4} =\overline{\sigma}^{\star }_{4, i} := \max \{ t \le u'_{4} \ | \ LE ( S^{4} [0, T_{2i + 1} ] ) (t) \in \partial B ( w, 8 \epsilon n ) \}.
\end{align}

\medskip

\begin{rem}\label{polish-rem-2}
Recall that $z$ and $w$ are points in the $\epsilon n$ neighborhood of $x$ and $y$. 
By reversing paths of $S^{1}$ and $S^{2}$ in the probability in RHS of \eqref{key-12}, $S^{1}$ is a random walk from $z$ to the origin, $S^{2}$ is a random walk from $z$ to $z'$, $S^{4}$ is a random walk from $w$ to $z'$, and $S^{3}$ is a random walk from $w$ to $\partial B(n)$. We want to deal with eight events in the probability of \eqref{key-12} as if they were independent. Some technical issues arise when we deal with the fifth, sixth, and seventh events. We will first deal with the sixth event in the next lemma below, by using entrance and exit times defined as in \eqref{ex1}.

\end{rem}

\medskip

We have to estimate the probability in RHS of \eqref{key-12}. With the strategy in Remark \ref{polish-rem-2} in mind, we first deal with the sixth event of the probability in \eqref{key-12}. The sixth event is written in terms of the loop-erasure of three walks $S^{1}$, $S^{2}$ and $S^{4}$. We want to replace it by the loop-erasure of $S^{4}$ only. In the next lemma, we will do the replacement by using entrance and exit times defined in \eqref{ex1}.

% In order to show \eqref{rem-2-1} and \eqref{rem-2-2}, we will prove Lemma \ref{hosoku-2} below. Lemma \ref{hosoku-2} states that we may replace $LE ( S^{1} [0, \sigma^{1}_{z} ] + S^{2} [0, T^{2}_{z, \frac{l \epsilon n}{2}}] + S^{4}[0, \sigma^{4}_{w}] ) [0, \overline{\sigma}_{2}]$ by $LE ( S^{4} [ 0, T_{2i+1} ] ) [\overline{\sigma}^{\star }_{4}, u_{4}']$ after reversing $S^{4}[0, \sigma^{4}_{w}]$ (see \eqref{ex1} and \eqref{darui} for $T_{i}$, $\overline{\sigma}^{\star }_{4}$, and $u_{4}'$).

\medskip

\begin{lem}\label{hosoku-2}
Suppose that $r \le \log_{2} l -3$. Then there exists $C < \infty$ such that 
\begin{align}\label{hosoku-2-1}
&P^{0}_{1} \otimes P^{z}_{2} \otimes P^{w}_{3} (F^{1}, \ q^{(1)}= r, \ S^{2} ( T^{2}_{z, \frac{l \epsilon n}{2}} ) = z' ) \notag \\
&\le C E^{0}_{1} \otimes E^{z}_{2} \Big\{ {\bf 1}_{F^{2}} \notag \\
&\times \Big( \frac{C}{\epsilon n} \frac{1}{\epsilon l n} \sum_{i=0}^{\infty}  P^{w}_{4} \otimes P^{w}_{3} \Big( T_{2i+1} <  \tau_{n}^{4}, \  LE ( S^{4} [ 0, T_{2i+1} ] ) [\overline{\sigma}^{\star }_{4}, u_{4}'] \cap S^{3} [0, T^{3}_{w, \frac{l \epsilon n}{4}} ] = \emptyset \Big) + \frac{C}{( \epsilon l n)^{2}} \Big) \notag \\
&\times \max_{w_{1} \in \partial B ( w, \frac{\epsilon l n}{4} ) } P^{w_{1}}_{3} \Big( LE ( S^{1} [0, \sigma^{1}_{z} ] ) [0, \overline{\sigma}_{1}] \cap S^{3} [ 0 , \tau^{3}_{n}] = \emptyset \Big) \Big\}.
\end{align}
(See \eqref{ex1} and \eqref{darui} for $T_{i}$, $\overline{\sigma}^{\star }_{4}$, and $u_{4}'$)
\end{lem}

\begin{proof} Condition on $S^{1}[0, \sigma^{1}_{z} ]$ and $S^{2}[  0 , T^{2}_{z, \frac{l \epsilon n}{2}} ]$ on $F^{2}$, let 
\begin{align}\label{key-14}
&\gamma = LE ( S^{1}[0, \sigma^{1}_{z} ] + S^{2}[  0 , T^{2}_{z, \frac{l \epsilon n}{2}} ] ), \notag \\
&\gamma_{1} = LE_{1} ( S^{1}[0, \sigma^{1}_{z} ], S^{2}[  0 , T^{2}_{z, \frac{l \epsilon n}{2}} ] ), \notag \\
&\gamma_{2} = LE_{2} ( S^{1}[0, \sigma^{1}_{z} ], S^{2}[  0 , T^{2}_{z, \frac{l \epsilon n}{2}} ] ),
\end{align}
so that $\gamma = \gamma_{1} + \gamma_{2}$. Since $LE ( S^{1} [0, \sigma^{1}_{z} ] ) [0, \overline{\sigma}_{1}] \cap  S^{2} [ 0 , T^{2}_{z, \frac{l \epsilon n}{2}} ]= \emptyset$ on $F^{2}$, we see that $\text{len} \gamma_{1} > \overline{\sigma}_{1}$, $LE ( S^{1} [0, \sigma^{1}_{z} ] ) [0, \overline{\sigma}_{1}] \subset \gamma_{1} $, $\gamma_{2} (0) \in  LE ( S^{1} [0, \sigma^{1}_{z} ] ) [ \overline{\sigma}_{1}, u^{(1)} ]$ and $\gamma_{2} ( \text{len} \gamma_{2} ) = z'$.

Conditioning $S^{1}[0, \sigma^{1}_{z} ]$ and $S^{2}[  0 , T^{2}_{z, \frac{l \epsilon n}{2}} ]$ on $F^{2}$, we will deal with $S^{4}$ and $S^{3}$. Suppose that $LE ( S^{1} [0, \sigma^{1}_{z} ] ) [0, \overline{\sigma}_{1}] \cap S^{4}[0, \sigma^{4}_{w}] = \emptyset$. Let 
\begin{align}\label{key-15}
&\lambda = LE ( S^{1}[0, \sigma^{1}_{z} ] + S^{2}[  0 , T^{2}_{z, \frac{l \epsilon n}{2}} ] + S^{4}[0, \sigma^{4}_{w}] ), \notag \\
&\lambda_{1} = LE_{1} ( \gamma, S^{4}[0, \sigma^{4}_{w}] ), \notag \\
&\lambda_{2} = LE_{2} ( \gamma, S^{4}[0, \sigma^{4}_{w}] ),
\end{align}
so that $\lambda = \lambda_{1} + \lambda_{2}$. Since $LE ( S^{1} [0, \sigma^{1}_{z} ] ) [0, \overline{\sigma}_{1}] \cap S^{4}[0, \sigma^{4}_{w}] = \emptyset$, we see that $\text{len} \lambda_{1} > \overline{\sigma}_{1}$, $\lambda_{2} (0) \in \gamma_{1} [\overline{\sigma}_{1}, \text{len} \gamma_{1} ] \cup \gamma_{2}$ and $LE ( S^{1} [0, \sigma^{1}_{z} ] ) [0, \overline{\sigma}_{1}] \subset \lambda_{1}$. 

Let $u_{2} = \text{len} \lambda_{1}$ and let $T' = \max \{ t \le \sigma^{4}_{w} \ | \ S^{4} (t) = \lambda_{1} (u_{2}) \}$. We see that $S^{4} (T') = \lambda_{1} (u_{2}) = \lambda_{2} (0) \in \gamma_{1} [\overline{\sigma}_{1}, \text{len} \gamma_{1} ] \cup \gamma_{2}$. Suppose that $q^{(1)} = r \le \log_{2} l -3$. Then $\gamma_{1} [\overline{\sigma}_{1}, \text{len} \gamma_{1} ] \subset LE ( S^{1} [0, \sigma^{1}_{z} ] ) [ \overline{\sigma}_{1}, u^{(1)} ] \subset B ( z, 2^{r} \epsilon n ) \subset B( z, \frac{l \epsilon n}{4} )$. Thus $S^{4} ( T' ) \in B( z, \frac{l \epsilon n}{2} )$. 

Note that $u_{2} = \inf \{ t \ | \ \gamma (t) \in S^{4} [ 0, \sigma^{4}_{w} ] \}$ and $T' = \max \{ t \le \sigma^{4}_{w} \ | \ S^{4} (t) = \gamma ( u_{2} ) \}$. Conditioning $\gamma$, we are interested in 
\begin{align}\label{key-16}
&\tilde{p}_{1} := P^{z'}_{4} \otimes P^{w}_{3} \Big( T' \le \sigma^{4}_{w} < \tau_{n}^{4}, \ S^{4} ( T') \in \gamma_{1} [\overline{\sigma}_{1}, \text{len} \gamma_{1} ] \cup \gamma_{2}, S^{3} [ 1, \tau^{3}_{n} ] \cap \overline{(\epsilon n B'_{y})}  = \emptyset, \notag \\
& \ \ \  \ \ \  \ \ \ \ \ \  \  LE ( S^{1} [0, \sigma^{1}_{z} ] ) [0, \overline{\sigma}_{1}] \cap S^{3} [ T^{3}_{w, \frac{l \epsilon n}{4}} , \tau^{3}_{n}] = \emptyset, LE ( S^{4} [T', \sigma^{4}_{w} ] ) [0, \overline{\sigma}_{4}] \cap S^{3} [0, T^{3}_{w, \frac{l \epsilon n}{4}} ] = \emptyset \Big) \notag \\
&(\text{where } \overline{\sigma}_{4} := \inf \{ t \ | \ LE ( S^{4} [T', \sigma^{4}_{w} ] ) (t) \in \partial (\epsilon n B'_{y} \} ) \notag \\
&\le P^{z'}_{4} \otimes P^{w}_{3} \Big( T' \le \sigma^{4}_{w} < \tau_{n}^{4}, \ S^{4} ( T') \in \gamma_{1} [\overline{\sigma}_{1}, \text{len} \gamma_{1} ] \cup \gamma_{2}, \notag \\
& \ \ \ \ \ \ \ \ \ \ \ \ \ \ S^{3} [ 1, T^{3}_{w, \frac{l \epsilon n}{4}} ] \cap \overline{(\epsilon n B'_{y})}  = \emptyset, LE ( S^{4} [T', \sigma^{4}_{w} ] ) [0, \overline{\sigma}_{4}] \cap S^{3} [0, T^{3}_{w, \frac{l \epsilon n}{4}} ] = \emptyset \Big) \notag \\
&\times \max_{w_{1} \in \partial B ( w, \frac{\epsilon l n}{4} ) } P^{w_{1}}_{3} \Big( LE ( S^{1} [0, \sigma^{1}_{z} ] ) [0, \overline{\sigma}_{1}] \cap S^{3} [ 0 , \tau^{3}_{n}] = \emptyset \Big).
\end{align}
We will consider the time reverse of $S^{4} [0, \sigma^{4}_{w} ]$. Note that for each SRW path $\eta = [ \eta (0), \cdots , \eta (m) ]$ with $\eta (0) = z'$ and $\eta (m) = w$, we have $P^{z'}_{4} \big( S^{4} [0, \sigma^{4}_{w} ] = \eta \big) = \frac{p_{w}}{p_{z}} P^{w}_{4} \big( S^{4} [0, \sigma^{4}_{z} ] = \eta^{R} \big)$. Suppose that $S^{4} (0) = w$ and $\sigma^{4}_{z'} < \tau^{4}_{n}$ (this is equivalent to $\tau^{4}_{z'} < \tau^{4}_{n}$). Define $u_{2}' := \inf \{ t \ | \ \gamma (t) \in S^{4} [0, \sigma^{4}_{w} ] \}$ and $T'' := \inf \{ t \ | \ S^{4} (t) = \gamma ( u_{2}' ) \}$. Let $\overline{\sigma}_{4}' := \max \{ t \ | \ LE ( S^{4} [ 0, T'' ] ) (t) \in \partial (\epsilon n B'_{y} ) \}$ and $u_{4} := \text{len} LE ( S^{4} [ 0, T'' ] ) $. Then by the time reversibility of LERW (see Lemma \ref{reversal}), the distribution of $LE ( S^{4} [T', \sigma^{4}_{w} ] ) [0, \overline{\sigma}_{4}]$ under $P^{z'}_{4}$ is same to that of $\big( LE ( S^{4} [ 0, T'' ] ) [\overline{\sigma}_{4}', u_{4}] \big)^{R}$ under $P^{w}_{4}$. Therefore,
\begin{align}\label{key-17}
&P^{z'}_{4} \otimes P^{w}_{3} \Big( T' \le \sigma^{4}_{w} < \tau_{n}^{4}, \ S^{4} ( T') \in \gamma_{1} [\overline{\sigma}_{1}, \text{len} \gamma_{1} ] \cup \gamma_{2}, \notag \\
& \ \ \ \ \ \ \ \ \ \ \ \ \ \ S^{3} [ 1, T^{3}_{w, \frac{l \epsilon n}{4}} ] \cap \overline{(\epsilon n B'_{y})}  = \emptyset, LE ( S^{4} [T', \sigma^{4}_{w} ] ) [0, \overline{\sigma}_{4}] \cap S^{3} [0, T^{3}_{w, \frac{l \epsilon n}{4}} ] = \emptyset \Big) \notag \\
&= \frac{p_{w}}{p_{z}} P^{w}_{4} \otimes P^{w}_{3} \Big( T'' \le \sigma^{4}_{z'} < \tau_{n}^{4}, \ S^{4} ( T'') \in \gamma_{1} [\overline{\sigma}_{1}, \text{len} \gamma_{1} ] \cup \gamma_{2}, \notag \\
& \ \ \ \ \ \ \ \ \ \ \ \ \ \ S^{3} [ 1, T^{3}_{w, \frac{l \epsilon n}{4}} ] \cap \overline{(\epsilon n B'_{y})}  = \emptyset, LE ( S^{4} [ 0, T'' ] ) [\overline{\sigma}_{4}', u_{4}] \cap S^{3} [0, T^{3}_{w, \frac{l \epsilon n}{4}} ] = \emptyset \Big) \notag \\
&\le \frac{c}{\epsilon n} \times \max_{w_{3} \in \partial B ( w, 6 \epsilon n )} P^{w}_{4} \otimes P^{w_{3}}_{3} \Big( T'' \le \sigma^{4}_{z'} < \tau_{n}^{4}, \ S^{4} ( T'') \in \gamma_{1} [\overline{\sigma}_{1}, \text{len} \gamma_{1} ] \cup \gamma_{2}, \notag \\
&  \ \  \ \ \ \ \ \ \ \ \ \ \ \ \ \ \ \ \ \  \ \ \ \ \ \ \  \ \ \ \ \ \ \ \ \ \  LE ( S^{4} [ 0, T'' ] ) [\overline{\sigma}_{4}'', u_{4}] \cap S^{3} [0, T^{3}_{w, \frac{l \epsilon n}{4}} ] = \emptyset \Big),
\end{align}
where $\overline{\sigma}_{4}'' := \max \{ t \ | \ LE ( S^{4} [ 0, T'' ] ) (t) \in \partial B ( w, 8 \epsilon n ) \}$ and we used $P^{w}_{3} \Big(  S^{3} [ 1, T^{3}_{w, 6 \epsilon n} ] \cap \overline{(\epsilon n B'_{y})}  = \emptyset \Big) \le \frac{c}{\epsilon n}$ in the last inequality (see Proposition 1.5.10 \cite{Law-book90} for this). By the Harnack principle (see Theorem 1.7.6 \cite{Law-book90}),
\begin{align}\label{key-18}
&\text{ RHS of \eqref{key-17} } \le \frac{C}{\epsilon n} P^{w}_{4} \otimes P^{w}_{3} \Big( T'' \le \sigma^{4}_{z'} < \tau_{n}^{4}, \ S^{4} ( T'') \in \gamma_{1} [\overline{\sigma}_{1}, \text{len} \gamma_{1} ] \cup \gamma_{2}, \notag \\
&  \ \  \ \ \ \ \ \ \ \ \ \ \ \ \ \ \ \ \ \  \ \ \ \ \ \ \  \ \ \ \ \ \ \ \ \ \  LE ( S^{4} [ 0, T'' ] ) [\overline{\sigma}_{4}'', u_{4}] \cap S^{3} [0, T^{3}_{w, \frac{l \epsilon n}{4}} ] = \emptyset \Big).
\end{align}

Let $\tau' := \inf \{ t \ge T^{4}_{w, \frac{\epsilon l n}{1600}} \ | \ S^{4}(t) \in \partial B (w, 8 \epsilon n ) \}$. Then $P^{w}_{4} \big( \tau' < \sigma^{4}_{z'} < \tau_{n}^{4} \big) \le \frac{C}{l^{2} \epsilon n}$ (see Proposition 1.5.10 \cite{Law-book90}). Therefore,
\begin{align}\label{key-19}
&P^{w}_{4} \otimes P^{w}_{3} \Big( T'' \le \sigma^{4}_{z'} < \tau_{n}^{4}, \ S^{4} ( T'') \in \gamma_{1} [\overline{\sigma}_{1}, \text{len} \gamma_{1} ] \cup \gamma_{2}, \ LE ( S^{4} [ 0, T'' ] ) [\overline{\sigma}_{4}'', u_{4}] \cap S^{3} [0, T^{3}_{w, \frac{l \epsilon n}{4}} ] = \emptyset \Big) \notag \\
&\le P^{w}_{4} \otimes P^{w}_{3} \Big( T'' \le \sigma^{4}_{z'} < \tau_{n}^{4}, \ \tau' > \sigma^{4}_{z'},   \ S^{4} ( T'') \in \gamma_{1} [\overline{\sigma}_{1}, \text{len} \gamma_{1} ] \cup \gamma_{2}, \notag \\
& \ \ \ \ \ \ \ \ \ \ \ \  \ \ \ LE ( S^{4} [ 0, T'' ] ) [\overline{\sigma}_{4}'', u_{4}] \cap S^{3} [0, T^{3}_{w, \frac{l \epsilon n}{4}} ] = \emptyset \Big) + \frac{C}{l^{2} \epsilon n}.
\end{align}

Suppose that $q^{(1)} = r \le \log_{2} l -3$, $T'' \le \sigma^{4}_{z'} < \tau_{n}^{4}$ and $S^{4} ( T'') \in \gamma_{1} [\overline{\sigma}_{1}, \text{len} \gamma_{1} ] \cup \gamma_{2}$. Then $S^{4} ( T'') \in B ( z, \frac{ \epsilon l n}{2} ) \subset B( w, \frac{ \epsilon l n}{3} )^{c}$. Let $i_{0}$ be the unique index $i$ such that $T_{2i+1} < T'' \le \min \{ T_{2i+2}, \sigma^{4}_{z'} \}$. Suppose that $\tau' > \sigma^{4}_{z'}$. Since $T_{2i_{0} + 1} < T'' < T_{2i_{0} + 2}$ and $S^{4} [ T_{2i_{0} + 1}, T''] \cap B ( w, \frac{ \epsilon l n}{800} ) = \emptyset$, we have
\begin{equation*}
\overline{\sigma}_{4}'' = \max \{ t \ | \ LE ( S^{4} [0, T_{2i_{0} + 1} ] ) (t) \in \partial B ( w, 8 \epsilon n ) \}.
\end{equation*}
(Recall that $\overline{\sigma}_{4}'' := \max \{ t \ | \ LE ( S^{4} [ 0, T'' ] ) (t) \in \partial B ( w, 8 \epsilon n ) \}$.) Furthermore, if we let
\begin{equation*}
u''_{4} = \inf \{ t \ | \ LE ( S^{4} [0, T_{2i_{0} + 1} ] ) (t) \in \partial B ( w, \frac{\epsilon l n}{1600} ) \},
\end{equation*}
then $\overline{\sigma}_{4}'' < u''_{4}$, $\overline{\sigma}_{4}'' = \max \{ t \le u''_{4} \ | \ LE ( S^{4} [0, T_{2i_{0} + 1} ] ) (t) \in \partial B ( w, 8 \epsilon n ) \}$ and 
\begin{equation*}
u''_{4} = \inf \{ t \ | \ LE ( S^{4} [ 0, T'' ] ) (t) \in \partial B ( w, \frac{\epsilon l n}{1600} ) \}.
\end{equation*}
Therefore we see that $LE ( S^{4} [ 0, T'' ] ) [\overline{\sigma}_{4}'', u''_{4} ] = LE ( S^{4} [0, T_{2i_{0} + 1} ] ) [ \overline{\sigma}_{4}'', u''_{4} ]$ and 
\begin{align}\label{key-20}
&\text{The first term of RHS of \eqref{key-19} } \notag \\
&\le \sum_{i=0}^{\infty} P^{w}_{4} \otimes P^{w}_{3} \Big( T_{2i+1} < \sigma^{4}_{z'} < \tau_{n}^{4}, \ \tau' > \sigma^{4}_{z'},   \ LE ( S^{4} [ 0, T_{2i+1} ] ) [\overline{\sigma}^{\star }_{4}, u_{4}'] \cap S^{3} [0, T^{3}_{w, \frac{l \epsilon n}{4}} ] = \emptyset \Big) \notag \\
&(\text{Recall that } \overline{\sigma}^{\star }_{4} \text{ and } u_{4}' \text{ were defined as in \eqref{darui}}) \notag \\
&\le \sum_{i=0}^{\infty} \frac{C}{\epsilon l n} P^{w}_{4} \otimes P^{w}_{3} \Big( T_{2i+1} <  \tau_{n}^{4}, \  LE ( S^{4} [ 0, T_{2i+1} ] ) [\overline{\sigma}^{\star }_{4}, u_{4}'] \cap S^{3} [0, T^{3}_{w, \frac{l \epsilon n}{4}} ] = \emptyset \Big),
\end{align}
where we used the strong Markov property and the fact that $\max_{w' \in \partial B( w, \frac{\epsilon l n}{4} )} P^{w'}_{4} \big(  \tau^{4}_{z'} < \infty \big) \le \frac{C}{\epsilon l n}$ in the last inequality (see Proposition 1.5.10 \cite{Law-book90}), and we finish the proof of Lemma \ref{hosoku-2}.  \end{proof}

\medskip

Recall the strategy in Remark \ref{polish-rem-2}. By Lemma \ref{hosoku-2}, we replaced the sixth event in \eqref{key-12} by the event the loop-erasure of $S^{4}$ up to some stopping time  does not intersect $S^{3}$. We want to show that the probability of that event is bounded above by an escape probability, i.e., we want to prove that 
\begin{equation}\label{polish-rem-3-1}
\sum_{i=0}^{\infty}  P^{w}_{4} \otimes P^{w}_{3} \Big( T_{2i+1} <  \tau_{n}^{4}, \  LE ( S^{4} [ 0, T_{2i+1} ] ) [\overline{\sigma}^{\star }_{4}, u_{4}'] \cap S^{3} [0, T^{3}_{w, \frac{l \epsilon n}{4}} ] = \emptyset \Big) \le C Es ( \epsilon n, \epsilon l n ).
\end{equation}
In order to show \eqref{polish-rem-3-1}, we need to study the distribution of $LE ( S^{4} [ 0, T_{2i+1} ] )$. The next lemma compares the distribution of $LE ( S^{4} [ 0, T_{2i+1} ] )$ with that of $LE ( S^{4} [ 0, T_{1} ] )$. Note that the probability of $T_{2i+1} < \infty$ is bounded above by $c~{i}$ for some $c < 1$. The next lemma shows that conditioned on $T_{2i+1} < \infty$, the distribution of $LE ( S^{4} [ 0, T_{2i+1} ] )$ is comparable to that of $LE ( S^{4} [ 0, T_{1} ] )$.

\medskip

\begin{lem}\label{key-21}
There exists a $c \in (\frac{1}{2}, 1)$ such that for all $i \ge 0$ and for every simple path $\eta = [\eta (0), \cdots , \eta (m) ]$ with $\eta (0) = w$ and $\eta \subset \overline{B( w, \frac{\epsilon l n}{1600} )}$, we have 
\begin{equation}\label{key-22}
P^{w}_{4} \Big( T_{2i+1} <  \tau_{n}^{4}, \  LE ( S^{4} [ 0, T_{2i+1} ] ) [0, m] = \eta \Big) \le c^{i} P^{w}_{4} \Big(  LE ( S^{4} [ 0, T_{1} ] ) [0, m] = \eta \Big),
\end{equation}
where $T_{i}$ was defined as in \eqref{ex1}.
\end{lem}

\begin{proof} We will show this sublemma by induction. Take a simple path $\eta = [\eta (0), \cdots , \eta (m) ]$ with $\eta (0) = w$ and $\eta \subset \overline{B( w, \frac{\epsilon l n}{1600} )}$. The inequality \eqref{key-22} trivially holds when $i=0$. So suppose that \eqref{key-22} holds for $c \in (\frac{1}{2}, 1)$ and $i-1$. Note that 
\begin{align}\label{key-23}
&P^{w}_{4} \Big( T_{2i+1} <  \tau_{n}^{4}, \  LE ( S^{4} [ 0, T_{2i+1} ] ) [0, m] = \eta \Big) = P^{w}_{4} \Big( T_{2i} <  \tau_{n}^{4}, \  LE ( S^{4} [ 0, T_{2i+1} ] ) [0, m] = \eta \Big) \notag \\
&= P^{w}_{4} \Big( T_{2i} <  \tau_{n}^{4}, \  LE ( S^{4} [ 0, T_{2i+1} ] ) [0, m] = \eta, \ S^{4} [T_{2i}, T_{2i+1}] \cap \eta = \emptyset \Big) \notag \\
&  \ \ \  + P^{w}_{4} \Big( T_{2i} <  \tau_{n}^{4}, \  LE ( S^{4} [ 0, T_{2i+1} ] ) [0, m] = \eta, \ S^{4} [T_{2i}, T_{2i+1}] \cap \eta \neq \emptyset \Big).
\end{align}

Suppose that $T_{2i} <  \tau_{n}^{4}$, $LE ( S^{4} [ 0, T_{2i+1} ] ) [0, m] = \eta$ and $S^{4} [T_{2i}, T_{2i+1}] \cap \eta = \emptyset$. Let $LE^{(1)} = LE_{1} ( \lambda_{1}, \lambda_{2} )$ and $LE^{(2)} = LE_{2} ( \lambda_{1}, \lambda_{2} )$ where $\lambda_{1} = LE ( S^{4} [ 0, T_{2i-1} ] )$ and $\lambda_{2} = S^{4} [T_{2i-1}, T_{2i+1}]$ (see \eqref{comp-1} for $LE^{(i)}$). Then $LE ( S^{4} [ 0, T_{2i+1} ] ) = LE^{(1)} + LE^{(2)}$. Let $u= \text{len}LE^{(1)} = \inf \{ t \ | \ \lambda_{1} (t) \in \lambda_{2} \}$. Then $u > m$. Indeed, if $u \le m$, then $LE^{(1)} (u) = LE ( S^{4} [ 0, T_{2i+1} ] ) (u) = \eta (u)$. This implies $\eta (u) \in \lambda_{2}$. Since $S^{4} [T_{2i-1}, T_{2i}] \cap \overline{B( w, \frac{\epsilon l n}{1600} )} = \emptyset$, we see that $\eta (u) \in S^{4} [T_{2i}, T_{2i+1}]$, and we get a contradiction. Thus $u > m$. Therefore $\eta = LE ( S^{4} [ 0, T_{2i+1} ] ) [0, m] = LE^{(1)} [0, m] =  LE ( S^{4} [ 0, T_{2i-1} ] ) [0, m]$. So
\begin{align}\label{key-24}
&P^{w}_{4} \Big( T_{2i} <  \tau_{n}^{4}, \  LE ( S^{4} [ 0, T_{2i+1} ] ) [0, m] = \eta, \ S^{4} [T_{2i}, T_{2i+1}] \cap \eta = \emptyset \Big) \notag \\
&\le P^{w}_{4} \Big( T_{2i-1} <  \tau_{n}^{4}, \  LE ( S^{4} [ 0, T_{2i-1} ] ) [0, m] = \eta \Big) \times \max_{w' \in \partial B( w, \frac{ \epsilon l n }{4} ) }  P^{w'}_{4} \Big( t^{1} < \tau_{n}^{4}, \ S^{4} [t^{1}, t^{2}] \cap \eta = \emptyset \Big),
\end{align}
where $t^{1} = \inf \{ t \ | \ S^{4} (t) \in \partial B( w, \frac{\epsilon l n}{800} ) \}$ and $t^{2} = \inf \{ t \ge t^{1} \ | \ S^{4} (t) \in \partial B( w, \frac{\epsilon l n}{4} ) \}$.

Next we will deal with the second term in the RHS of \eqref{key-23}. Suppose that $T_{2i} <  \tau_{n}^{4}$, $S^{4} [T_{2i}, T_{2i+1}] \cap \eta \neq \emptyset$ and $LE ( S^{4} [ 0, T_{2i+1} ] ) [0, m] = \eta$. On this event, we may define $u' := \inf \{ t \ | \ \eta (t) \in S^{4} [T_{2i}, T_{2i+1}] \}$. Then $u' \le m$. Note that
\begin{align}\label{key-25}
&P^{w}_{4} \Big( T_{2i} <  \tau_{n}^{4}, \  LE ( S^{4} [ 0, T_{2i+1} ] ) [0, m] = \eta, \ S^{4} [T_{2i}, T_{2i+1}] \cap \eta \neq \emptyset \Big) \notag \\
&= \sum_{j=0}^{m} P^{w}_{4} \Big( T_{2i} <  \tau_{n}^{4}, \ u'=j, \ LE ( S^{4} [ 0, T_{2i+1} ] ) [0, m] = \eta \Big).
\end{align}

Suppose that $T_{2i} <  \tau_{n}^{4}$, $u' = j$ and $LE ( S^{4} [ 0, T_{2i+1} ] ) [0, m] = \eta$. Since $S^{4} [T_{2i-1}, T_{2i+1}]$ does not intersect $\eta [0, j-1]$, in order for $LE ( S^{4} [ 0, T_{2i+1} ] ) [0, m]$ to be $\eta$, $\eta [0, j]$ must be contained in the loop-erasure of $S^{4}$ up to $T_{2i-1}$. The rest part of $\eta$, say $\eta [j+1, m]$, is constructed by the loop-erasure of $S^{4} [T_{2i}, T_{2i+1}]$. Therefore, we have 
\begin{itemize}
\item $LE ( S^{4} [ 0, T_{2i-1} ] ) [0, j] = \eta [0, j] $,

\item If we let $\tau^{\star}_{\eta (j)} := \inf \{ t \ge T_{2i-1} \ | \ S^{4} (t) = \eta (j) \}$, then $T_{2i} < \tau^{\star}_{\eta (j)} < T_{2i+1}$ and $S^{4} [ T_{2i}, \tau^{\star}_{\eta (j)} ] \cap \eta [0, j-1] = \emptyset$,

\item $LE ( S^{4} [ \tau^{\star}_{\eta (j)}, T_{2i+1} ] ) [0, m-j] = \eta [j, m]$,

\item $S^{4} [ \tau^{\star}_{\eta (j)}, T_{2i+1} ]  \cap \eta [0, j-1 ] = \emptyset $.

\end{itemize}
So the probability in RHS of \eqref{key-25} is bounded above by the probability of four events above as follows.
\begin{align}\label{key-26}
&P^{w}_{4} \Big( T_{2i} <  \tau_{n}^{4}, \ u'=j, \ LE ( S^{4} [ 0, T_{2i+1} ] ) [0, m] = \eta \Big) \notag \\
&\le P^{w}_{4} \Big( T_{2i} <  \tau_{n}^{4}, \ LE ( S^{4} [ 0, T_{2i-1} ] ) [0, j] = \eta [0, j], \ T_{2i} < \tau^{\star}_{\eta (j)} < T_{2i+1}, \ S^{4} [ T_{2i}, \tau^{\star}_{\eta (j)} ] \cap \eta [0, j-1] = \emptyset, \notag \\
& \ \ \ \ \  \ \ \ \  \ LE ( S^{4} [ \tau^{\star}_{\eta (j)}, T_{2i+1} ] ) [0, m-j] = \eta [j, m], \ S^{4} [ \tau^{\star}_{\eta (j)}, T_{2i+1} ]  \cap \eta [0, j-1 ] = \emptyset \Big) \notag \\
&= \sum_{w' \in \partial B( w, \frac{ \epsilon l n }{4} ) } P^{w}_{4} \Big( T_{2i-1} <  \tau_{n}^{4}, \ S^{4} (T_{2i-1}) = w', \ LE ( S^{4} [ 0, T_{2i-1} ] ) [0, j] = \eta [0, j] \Big) \notag \\ 
& \ \   \times P^{w'}_{4} \Big( t^{1} < \tau_{n}^{4}, \ \tau^{4}_{\eta (j) } < t^{2}, \ S^{4} [ t^{1}, \tau^{4}_{\eta (j) } ] \cap \eta [0, j-1] = \emptyset, \notag \\ 
& \ \ \ \ \ \ \ \ LE ( S^{4} [ \tau^{4}_{\eta (j)}, t^{2} ] ) [0, m-j] = \eta [j, m], \ S^{4} [ \tau^{4}_{\eta (j)}, t^{2} ] \cap \eta [0, j-1] = \emptyset \Big),
\end{align}
where $\tau^{4}_{\eta (j) } = \inf \{ t \ | \ S^{4} (t) = \eta (j) \}$. Since $\eta \subset \overline{B( w, \frac{\epsilon l n}{1600} )}$, in order for $S^{4}$ to hit $\eta$, $S^{4}$ must intersect $\partial B( w, \frac{\epsilon l n}{800} )$ before $\tau^{4}_{\eta (j) }$. So by using the strong Markov property at $t^{1}$ first, then using it again at $\tau^{4}_{\eta (j) }$, we have 
\begin{align}\label{key-27}
&P^{w'}_{4} \Big( t^{1} < \tau_{n}^{4}, \ \tau^{4}_{\eta (j) } < t^{2}, \ S^{4} [ t^{1}, \tau^{4}_{\eta (j) } ] \cap \eta [0, j-1] = \emptyset, \notag \\ 
& \ \ \ \ \ \ \ \ LE ( S^{4} [ \tau^{4}_{\eta (j)}, t^{2} ] ) [0, m-j] = \eta [j, m], \ S^{4} [ \tau^{4}_{\eta (j)}, t^{2} ] \cap \eta [0, j-1] = \emptyset \Big) \notag \\
&= \sum_{w'' \in \partial B( w, \frac{\epsilon l n}{800} ) } P^{w'}_{4} \big( S^{4} ( t^{1} ) = w'' \big)  \notag \\ 
& \ \ \times P^{w''}_{4} \Big( \tau^{4}_{\eta (j) } < t^{2}, \ S^{4} [0, \tau^{4}_{\eta (j) } ] \cap \eta [0, j-1] = \emptyset, \notag \\ 
& \ \ \ \ \ \ \ \ LE ( S^{4} [ \tau^{4}_{\eta (j)}, t^{2} ] ) [0, m-j] = \eta [j, m], \ S^{4} [ \tau^{4}_{\eta (j)}, t^{2} ] \cap \eta [0, j-1] = \emptyset \Big) \notag \\
&= \sum_{w'' \in \partial B( w, \frac{\epsilon l n}{800} ) } P^{w'}_{4} \big( S^{4} ( t^{1} ) = w'' \big)  \times P^{w''}_{4} \Big( \tau^{4}_{\eta (j) } < t^{2}, \ S^{4} [0, \tau^{4}_{\eta (j) } ] \cap \eta [0, j-1] = \emptyset \Big) \notag \\ 
& \ \ \times P^{\eta (j)}_{4} \Big(  LE ( S^{4} [ 0, t^{2} ] ) [0, m-j] = \eta [j, m], \ S^{4} [ 0, t^{2} ] \cap \eta [0, j-1] = \emptyset \Big).
\end{align}
By the equation in line 10, page 199 of \cite{Law LERW99}, we can write the distribution of LERW in terms of Green's functions and non-intersecting probabilities of $\eta$ as follows.
\begin{align}\label{key-28}
&P^{\eta (j)}_{4} \Big(  LE ( S^{4} [ 0, t^{2} ] ) [0, m-j] = \eta [j, m], \ S^{4} [ 0, t^{2} ] \cap \eta [0, j-1] = \emptyset \Big) \notag \\
&= \prod_{q=j}^{m-1} G \big( \eta (q), \eta (q), B \setminus \eta [0, q-1] \big) \ P^{\eta (q)}_{4} \big( S^{4}(1) = \eta (q+1) \big)  \ G \big( \eta (m), \eta (m), B \setminus \eta [0, m-1] \big) \notag \\
&\ \   \ \times P^{\eta (m)}_{4} \big( S^{4} [ 1, t^{2} ] \cap \eta [0, m] = \emptyset \big),
\end{align}
where $B = B (w, \frac{ \epsilon l n}{4} )$ and $G (\cdot, \cdot, \cdot )$ is Green's function defined in Section \ref{kigou}.

Take $w'_{0} \in \partial B( w, \frac{ \epsilon l n }{4} )$ and $w''_{0} \in \partial B( w, \frac{\epsilon l n}{800} )$ such that 
\begin{equation*}
P^{w'_{0}}_{4} \big( S^{4} ( t^{1} ) = w''_{0} \big) = \max_{w' \in \partial B( w, \frac{ \epsilon l n }{4} ), \ w'' \in \partial B( w, \frac{\epsilon l n}{800} ) } P^{w'}_{4} \big( S^{4} ( t^{1} ) = w'' \big). 
\end{equation*}
Then by using Proposition 1.5.10 \cite{Law-book90}, we see that 
\begin{equation}\label{key-29}
P^{w'_{0}}_{4} \big( S^{4} ( t^{1} ) = w''_{0} \big) \le \frac{6400}{(\epsilon l n)^{2}}.
\end{equation}

By \eqref{key-26}-\eqref{key-28}, and by definition of $w'_{0}$ and $w''_{0}$, 
\begin{align}\label{unchi-1}
&P^{w}_{4} \Big( T_{2i} <  \tau_{n}^{4}, \ u'=j, \ LE ( S^{4} [ 0, T_{2i+1} ] ) [0, m] = \eta \Big) \notag \\
&\le \sum_{w' \in \partial B( w, \frac{ \epsilon l n }{4} ) } P^{w}_{4} \Big( T_{2i-1} <  \tau_{n}^{4}, \ S^{4} (T_{2i-1}) = w', \ LE ( S^{4} [ 0, T_{2i-1} ] ) [0, j] = \eta [0, j] \Big) \notag \\ 
& \ \ \times \sum_{w'' \in \partial B( w, \frac{\epsilon l n}{800} ) } P^{w'}_{4} \big( S^{4} ( t^{1} ) = w'' \big)  \times P^{w''}_{4} \Big( \tau^{4}_{\eta (j) } < t^{2}, \ S^{4} [0, \tau^{4}_{\eta (j) } ] \cap \eta [0, j-1] = \emptyset \Big) \notag \\
& \ \ \times \prod_{q=j}^{m-1} G \big( \eta (q), \eta (q), B \setminus \eta [0, q-1] \big) \ P^{\eta (q)}_{4} \big( S^{4}(1) = \eta (q+1) \big)  \ G \big( \eta (m), \eta (m), B \setminus \eta [0, m-1] \big) \notag \\
&\ \   \ \times P^{\eta (m)}_{4} \big( S^{4} [ 1, t^{2} ] \cap \eta [0, m] = \emptyset \big) \notag \\
% &\le \sum_{w' \in \partial B( w, \frac{ \epsilon l n }{4} ) } P^{w}_{4} \Big( T_{2i-1} <  \tau_{n}^{4}, \ S^{4} (T_{2i-1}) = w', \ LE ( S^{4} [ 0, T_{2i-1} ] ) [0, j] = \eta [0, j] \Big) \notag \\ 
% & \ \ \times \sum_{w'' \in \partial B( w, \frac{\epsilon l n}{800} ) } P^{w'_{0}}_{4} \big( S^{4} ( t^{1} ) = w''_{0} \big)  \times P^{w''}_{4} \Big( \tau^{4}_{\eta (j) } < t^{2}, \ S^{4} [0, \tau^{4}_{\eta (j) } ] \cap \eta [0, j-1] = \emptyset \Big) \notag \\
% & \ \ \times \prod_{q=j}^{m-1} G \big( \eta (q), \eta (q), B \setminus \eta [0, q-1] \big) \ P^{\eta (q)}_{4} \big( S^{4}(1) = \eta (q+1) \big)  \ G \big( \eta (m), \eta (m), B \setminus \eta [0, m-1] \big) \notag \\
% &\ \   \ \times P^{\eta (m)}_{4} \big( S^{4} [ 1, t^{2} ] \cap \eta [0, m] = \emptyset \big) \notag \\
&\le  P^{w}_{4} \Big( T_{2i-1} <  \tau_{n}^{4}, \  LE ( S^{4} [ 0, T_{2i-1} ] ) [0, j] = \eta [0, j] \Big) \notag \\ 
& \ \ \times \sum_{w'' \in \partial B( w, \frac{\epsilon l n}{800} ) } P^{w'_{0}}_{4} \big( S^{4} ( t^{1} ) = w''_{0} \big)  \times P^{w''}_{4} \Big( \tau^{4}_{\eta (j) } < t^{2}, \ S^{4} [0, \tau^{4}_{\eta (j) } ] \cap \eta [0, j-1] = \emptyset \Big) \notag \\
& \ \ \times \prod_{q=j}^{m-1} G \big( \eta (q), \eta (q), B \setminus \eta [0, q-1] \big) \ P^{\eta (q)}_{4} \big( S^{4}(1) = \eta (q+1) \big)  \ G \big( \eta (m), \eta (m), B \setminus \eta [0, m-1] \big) \notag \\
&\ \   \ \times P^{\eta (m)}_{4} \big( S^{4} [ 1, t^{2} ] \cap \eta [0, m] = \emptyset \big).  
\end{align}
In order to estimate the RHS of \eqref{unchi-1}, now we use the assumption of the induction for $\eta [0, j]$. By using it as well as the equation in line 10, page 199 of \cite{Law LERW99} for the distribution of $LE ( S^{4} [ 0, T_{1} ] ) [0, m]$, we see that
\begin{align}\label{key-30}
&\le c^{i-1}  P^{w}_{4} \Big( LE ( S^{4} [ 0, T_{1} ] ) [0, j] = \eta [0, j] \Big) \notag \\ 
& \ \ \times \sum_{w'' \in \partial B( w, \frac{\epsilon l n}{800} ) } P^{w'_{0}}_{4} \big( S^{4} ( t^{1} ) = w''_{0} \big)  \times P^{w''}_{4} \Big( \tau^{4}_{\eta (j) } < t^{2}, \ S^{4} [0, \tau^{4}_{\eta (j) } ] \cap \eta [0, j-1] = \emptyset \Big) \notag \\
& \ \ \times \prod_{q=j}^{m-1} G \big( \eta (q), \eta (q), B \setminus \eta [0, q-1] \big) \ P^{\eta (q)}_{4} \big( S^{4}(1) = \eta (q+1) \big)  \ G \big( \eta (m), \eta (m), B \setminus \eta [0, m-1] \big) \notag \\
&\ \   \ \times P^{\eta (m)}_{4} \big( S^{4} [ 1, t^{2} ] \cap \eta [0, m] = \emptyset \big) \notag \\
&= c^{i-1}  P^{w}_{4} \Big( LE ( S^{4} [ 0, T_{1} ] ) [0, m] = \eta [0, m] \Big) \notag \\ 
& \ \ \times \sum_{w'' \in \partial B( w, \frac{\epsilon l n}{800} ) } P^{w'_{0}}_{4} \big( S^{4} ( t^{1} ) = w''_{0} \big)  \times P^{w''}_{4} \Big( \tau^{4}_{\eta (j) } < t^{2}, \ S^{4} [0, \tau^{4}_{\eta (j) } ] \cap \eta [0, j-1] = \emptyset \Big) \notag \\
& \ \ \times  G \big( \eta (j), \eta (j), B \setminus \eta [0, j-1] \big) \ P^{\eta (j)}_{4} \big( S^{4} [ 1, t^{2} ] \cap \eta [0, j] = \emptyset \big).
\end{align}
Let $u^{\star} := \inf \{ t \ | \ \eta (t) \in S^{4} [0, t^{2} ] \}$. In order for $u^{\star}$ to be $j$, first $S^{4}$ hits $\eta (j)$ before $t^{2}$ and intersecting $\eta [0, j-1]$, then $S^{4}$ does not hit $\eta [0, j-1]$ from $\tau^{4}_{\eta (j) }$ to the last visit of $\eta (j)$, and finally it exits $B (w, \frac{ \epsilon l n}{4} )$ without intersecting $\eta [0, j-1]$. Thus we have
\begin{align*}
&P^{w''}_{4} \big( u^{\star} = j \big) \notag = \sum_{k=0}^{\infty} P^{w''}_{4} \Big( \tau^{4}_{\eta (j) } < t^{2}, \ t^{\star} = \tau^{4}_{\eta (j) } + 2k, \ S^{4} [0, t^{2}] \cap \eta [0, j-1] = \emptyset \Big) \notag \\
&(\text{ where } t^{\star} := \max \{ t \le t^{2} \ | \ S^{4} (t) = \eta (j) \} ) \notag \\
&= \sum_{k=0}^{\infty} P^{w''}_{4} \Big(  \tau^{4}_{\eta (j) } < t^{2}, \ S^{4} [0, \tau^{4}_{\eta (j) }] \cap \eta [0, j-1] = \emptyset \Big) \notag \\
&\ \ \times P^{\eta (j)}_{4} \Big( S^{4} (2k) = \eta (j), \ S^{4} [0, 2k ] \cap \eta [0, j-1] = \emptyset, \ S^{4} [0, 2k ] \subset B \Big) \ P^{\eta (j)}_{4} \Big( S^{4} [1, t^{2} ] \cap \eta [0, j] = \emptyset \Big) \notag \\
&= P^{w''}_{4} \Big( \tau^{4}_{\eta (j) } < t^{2}, \ S^{4} [0, \tau^{4}_{\eta (j) } ] \cap \eta [0, j-1] = \emptyset \Big) \ G \big( \eta (j), \eta (j), B \setminus \eta [0, j-1] \big) \notag \\
&\ \ \times P^{\eta (j)}_{4} \big( S^{4} [ 1, t^{2} ] \cap \eta [0, j] = \emptyset \big).
\end{align*}
Combining this with \eqref{key-30}, we have 
\begin{align}\label{key-31}
&P^{w}_{4} \Big( T_{2i} <  \tau_{n}^{4}, \ u'=j, \ LE ( S^{4} [ 0, T_{2i+1} ] ) [0, m] = \eta \Big) \notag \\
&\le c^{i-1}  P^{w}_{4} \Big( LE ( S^{4} [ 0, T_{1} ] ) [0, m] = \eta [0, m] \Big) \ \sum_{w'' \in \partial B( w, \frac{\epsilon l n}{800} ) } P^{w'_{0}}_{4} \big( S^{4} ( t^{1} ) = w''_{0} \big) \ P^{w''}_{4} \big( u^{\star} = j \big).
\end{align}
Clearly events $\{ u'=j \}$ are disjoint, and the same thing holds for events $\{ u^{\star} = j \}$. So taking sum for $j$ in \eqref{key-31}, we have 
\begin{align}\label{key-32}
&P^{w}_{4} \Big( T_{2i} <  \tau_{n}^{4}, \  LE ( S^{4} [ 0, T_{2i+1} ] ) [0, m] = \eta, \ S^{4} [T_{2i}, T_{2i+1}] \cap \eta \neq \emptyset \Big) \notag \\
&\le c^{i-1} P^{w}_{4} \Big( LE ( S^{4} [ 0, T_{1} ] ) [0, m] = \eta [0, m] \Big) \notag \\
& \ \ \times \sum_{w'' \in \partial B( w, \frac{\epsilon l n}{800} ) } P^{w'_{0}}_{4} \big( S^{4} ( t^{1} ) = w''_{0} \big) \ P^{w''}_{4} \Big( S^{4} [0, t^{2} ] \cap \eta [0, m ] \neq \emptyset \Big).
\end{align}

The estimate of the case that $S^{4} [T_{2i}, T_{2i+1}] \cap \eta \neq \emptyset$ was given as in \eqref{key-32}. For the case that $S^{4} [T_{2i}, T_{2i+1}] \cap \eta = \emptyset$, by \eqref{key-24} and the assumption of the induction, 
\begin{align}\label{key-33}
&P^{w}_{4} \Big( T_{2i} <  \tau_{n}^{4}, \  LE ( S^{4} [ 0, T_{2i+1} ] ) [0, m] = \eta, \ S^{4} [T_{2i}, T_{2i+1}] \cap \eta = \emptyset \Big) \notag \\
% &\le P^{w}_{4} \Big( T_{2i-1} <  \tau_{n}^{4}, \  LE ( S^{4} [ 0, T_{2i-1} ] ) [0, m] = \eta \Big) \notag \\
% & \ \ \times \sum_{w'' \in \partial B( w, \frac{\epsilon l n}{800} ) }  P^{w'_{0}}_{4} \big( S^{4} ( t^{1} ) = w''_{0} \big)  P^{w''}_{4} \Big( S^{4} [0, t^{2} ] \cap \eta [0, m ] = \emptyset \Big) \notag \\
&\le c^{i-1} P^{w}_{4} \Big( LE ( S^{4} [ 0, T_{1} ] ) [0, m] = \eta [0, m] \Big) \notag \\
& \ \ \times \sum_{w'' \in \partial B( w, \frac{\epsilon l n}{800} ) }  P^{w'_{0}}_{4} \big( S^{4} ( t^{1} ) = w''_{0} \big)  P^{w''}_{4} \Big( S^{4} [0, t^{2} ] \cap \eta [0, m ] = \emptyset \Big).
\end{align}

But \eqref{key-29} shows that $P^{w'_{0}}_{4} \big( S^{4} ( t^{1} ) = w''_{0} \big)$ is small enough compared with the number of lattice points in $\partial B( w, \frac{\epsilon l n}{800} )$. Since we assume $c \in (\frac{1}{2}, 1)$ in the assumption of the induction, this leads to finish the proof of the induction as follows.    
\begin{align}\label{key-34}
&P^{w}_{4} \Big( T_{2i+1} <  \tau_{n}^{4}, \  LE ( S^{4} [ 0, T_{2i+1} ] ) [0, m] = \eta \Big) \notag \\
&\le c^{i-1} P^{w}_{4} \Big( LE ( S^{4} [ 0, T_{1} ] ) [0, m] = \eta [0, m] \Big) \notag \\ 
& \ \ \times  \sum_{w'' \in \partial B( w, \frac{\epsilon l n}{800} ) }  P^{w'_{0}}_{4} \big( S^{4} ( t^{1} ) = w''_{0} \big) \Big\{ P^{w''}_{4} \Big( S^{4} [0, t^{2} ] \cap \eta [0, m ] \neq \emptyset \Big) + P^{w''}_{4} \Big( S^{4} [0, t^{2} ] \cap \eta [0, m ] = \emptyset \Big) \Big\} \notag \\
&\le c^{i-1} P^{w}_{4} \Big( LE ( S^{4} [ 0, T_{1} ] ) [0, m] = \eta [0, m] \Big) \frac{6400}{(\epsilon l n)^{2}} \frac{ 50 (\epsilon l n)^{2}}{640000} \le c^{i} P^{w}_{4} \Big( LE ( S^{4} [ 0, T_{1} ] ) [0, m] = \eta [0, m] \Big),
\end{align}
which finishes the proof of Lemma \ref{key-21}. \end{proof}

\medskip

Recall the strategy in Remark \ref{polish-rem-2}. Since $LE ( S^{4} [ 0, T_{2i+1} ] ) [\overline{\sigma}^{\star}_{4}, u_{4}'] \subset \overline{B( w, \frac{\epsilon l n}{1600} )}$, by Lemma \ref{key-21}, Lemma \ref{escape-25} and \eqref{escape-4}
\begin{align}\label{key-35}
&\sum_{i=0}^{\infty}  P^{w}_{4} \otimes P^{w}_{3} \Big( T_{2i+1} <  \tau_{n}^{4}, \  LE ( S^{4} [ 0, T_{2i+1} ] ) [\overline{\sigma}^{\star}_{4}, u_{4}'] \cap S^{3} [0, T^{3}_{w, \frac{l \epsilon n}{4}} ] = \emptyset \Big) \notag \\
&\le \sum_{i=0}^{\infty}  c^{i} P^{w}_{4} \otimes P^{w}_{3} \Big( LE ( S^{4} [ 0, T_{1} ] ) [t_{2}^{\star}, t_{1}^{\star} ] \cap S^{3} [0, T^{3}_{w, \frac{l \epsilon n}{4}} ] = \emptyset \Big) \notag \\
&\Big( \text{where } t_{1}^{\star} := \inf \{ t \ | \ LE ( S^{4} [ 0, T_{1} ] ) (t) \in \partial B( w, \frac{\epsilon l n}{1600} ) \} \notag \\ 
&\text{ and } t_{2}^{\star} := \max \{ t \le t_{1}^{\star} \ | \ LE ( S^{4} [ 0, T_{1} ] ) (t) \in \partial B( w, 8 \epsilon  n ) \} \Big) \notag \\
&\le C Es ( \epsilon n, \epsilon l n ).
\end{align}
Therefore, by \eqref{key-19},
\begin{align}\label{key-36}
&P^{w}_{4} \otimes P^{w}_{3} \Big( T'' \le \sigma^{4}_{z'} < \tau_{n}^{4}, \ S^{4} ( T'') \in \gamma_{1} [\overline{\sigma}_{1}, \text{len} \gamma_{1} ] \cup \gamma_{2}, \ LE ( S^{4} [ 0, T'' ] ) [\overline{\sigma}_{4}'', u_{4}] \cap S^{3} [0, T^{3}_{w, \frac{l \epsilon n}{4}} ] = \emptyset \Big) \notag \\
&\le \frac{C}{\epsilon l n} Es ( \epsilon n, \epsilon l n ),
\end{align}
where we used $l^{-1} \le Es ( \epsilon n, \epsilon l n )$ in the last inequality (see \eqref{escape-7}).

Thus by \eqref{key-16},
\begin{equation}\label{key-37}
\tilde{p}_{1} \le \frac{C}{\epsilon  n} \frac{1}{\epsilon l n} Es ( \epsilon n, \epsilon l n ) \max_{w_{1} \in \partial B ( w, \frac{\epsilon l n}{4} ) } P^{w_{1}}_{3} \Big( LE ( S^{1} [0, \sigma^{1}_{z} ] ) [0, \overline{\sigma}_{1}] \cap S^{3} [ 0 , \tau^{3}_{n}] = \emptyset \Big).
\end{equation}

\medskip

Combining \eqref{key-37} with \eqref{key-12}, we have 
\begin{align}\label{key-38}
&P^{0}_{1} \otimes P^{z}_{2} \otimes P^{w}_{3} \Big( F^{1}, \ q^{(1)}= r, \ S^{2} ( T^{2}_{z, \frac{l \epsilon n}{2}} ) = z' \Big) \le E^{0}_{1} \otimes E^{z}_{2} \Big( \tilde{p}_{1} {\bf 1}_{F^{2}} \Big) \notag \\
&(\text{Recall that } F^{2} \text{ was defined in } \eqref{key-13}.) \notag \\
&\le \frac{C}{\epsilon  n} \frac{1}{\epsilon l n} Es ( \epsilon n, \epsilon l n ) E^{0}_{1} \otimes E^{z}_{2} \Big\{ {\bf 1}_{F^{2}} \max_{w_{1} \in \partial B ( w, \frac{\epsilon l n}{4} ) } P^{w_{1}}_{3} \Big( LE ( S^{1} [0, \sigma^{1}_{z} ] ) [0, \overline{\sigma}_{1}] \cap S^{3} [ 0 , \tau^{3}_{n}] = \emptyset \Big) \Big\}.
\end{align}

\medskip

We need to estimate the expectation in RHS of \eqref{key-38}. Using the time reversibility of LERW (see Lemma \ref{reversal}), we can replace the loop erasure of $S^{1}$ from the origin to $z$ by the loop erasure of $S^{1}$ from $z$ to the origin. Therefore we have
\begin{align}\label{key-39}
&E^{0}_{1} \otimes E^{z}_{2} \Big\{ {\bf 1}_{F^{2}} \max_{w_{1} \in \partial B ( w, \frac{\epsilon l n}{4} ) } P^{w_{1}}_{3} \Big( LE ( S^{1} [0, \sigma^{1}_{z} ] ) [0, \overline{\sigma}_{1}] \cap S^{3} [ 0 , \tau^{3}_{n}] = \emptyset \Big) \Big\} \notag \\
&= E^{z}_{1} \otimes E^{z}_{2} \Big\{ {\bf 1}_{\tilde{F}^{2}} \max_{w_{1} \in \partial B ( w, \frac{\epsilon l n}{4} ) } P^{w_{1}}_{3} \Big( LE ( S^{1} [0, \tau^{1}_{0} ] ) [\tilde{u}_{1}, \tilde{u}_{2}] \cap S^{3} [ 0 , \tau^{3}_{n}] = \emptyset \Big) \Big\},
\end{align}
where 
\begin{equation}\label{TIMES}
\tilde{u}_{1} := \max \{ t \ | \ LE ( S^{1} [0, \tau^{1}_{0} ] ) (t) \in \partial (\epsilon n B'_{x} ) \}, \ \tilde{u}_{2} := \text{ len} LE ( S^{1} [0, \tau^{1}_{0} ] ),
\end{equation}
and 
\begin{align}\label{key-40}
&\tilde{F}^{2} := \Big\{ \tau^{1}_{0} < \tau^{1}_{n}, \ S^{2} [1, T^{2}_{z, \frac{l \epsilon n}{2}} ] \cap \overline{(\epsilon n B'_{x})} = \emptyset, \ S^{2} ( T^{2}_{z, \frac{l \epsilon n}{2}} ) = z', \notag \\
& \ \ \ \ \  \ \ \  \  LE ( S^{1} [0, \tau^{1}_{0} ] ) [\tilde{u}_{1}, \tilde{u}_{2}] \cap S^{2} [ 0 , T^{2}_{z, \frac{l \epsilon n}{2}} ] = \emptyset, \ LE ( S^{1} [0, \tau^{1}_{0} ] ) [0, \tilde{u}_{1}] \cap A^{r}_{z} \neq \emptyset, \notag \\
& \ \ \ \  \ \ \ \ \  LE ( S^{1} [0, \tau^{1}_{0} ] ) [0, \tilde{u}_{1}] \subset B (z, 2^{r} \epsilon n ) \Big\}.
\end{align}

%With Remark \ref{polish-rem-4} in mind, we will show \eqref{rem-4-1} in the next lemma.
\medskip

We have to estimate the expectation in RHS of \eqref{key-39}. As we discussed, we want to deal with all events in $\tilde{F}^{2}$ as if they were independent. That will be done in the next lemma. In order to control the independence of LERW, we will use Proposition 4.6 \cite{Mas}, which states that $\eta^{1}_{0, R} (S[0, \tau_{n} ])$ and $\eta^{2}_{0, 4R, n} ( S[0, \tau_{n} ])$ are ``independent up to constant". 

\medskip

\begin{lem}\label{hosoku-3}
Suppose that $r \le \log_{2} l -3$. Then there exist universal constants $C < \infty$ and $\delta > 0$ such that 
\begin{align}\label{hosoku-3-1}
&E^{z}_{1} \otimes E^{z}_{2} \Big\{ {\bf 1}_{\tilde{F}^{2}} \max_{w_{1} \in \partial B ( w, \frac{\epsilon l n}{4} ) } P^{w_{1}}_{3} \Big( LE ( S^{1} [0, \tau^{1}_{0} ] ) [\tilde{u}_{1}, \tilde{u}_{2}] \cap S^{3} [ 0 , \tau^{3}_{n}] = \emptyset \Big) \Big\} \notag \\
&\le \frac{C}{n} \frac{1}{ \epsilon n} 2^{- \delta r} Es (\epsilon n, n ) P^{z}_{2} \Big( S^{2} ( T^{2}_{z, \frac{l \epsilon n}{2}} ) = z' \Big).
\end{align}
\end{lem}

% \vspace{1zw}
% \begin{rem}\label{polish-rem-5}
% The key step for the proof of Lemma \ref{hosoku-3} is to establish ``up to constant independence" between two events $ LE ( S^{1} [0, \tau^{1}_{0} ] ) [\tilde{u}_{1}, \tilde{u}_{2}] \cap S^{2} [ 0 , T^{2}_{z, \frac{l \epsilon n}{2}} ] = \emptyset$ and $LE ( S^{1} [0, \tau^{1}_{0} ] ) [\tilde{u}_{1}, \tilde{u}_{2}] \cap S^{3} [ 0 , \tau^{3}_{n}] = \emptyset $. In order to do it, we decompose $LE ( S^{1} [0, \tau^{1}_{0} ] )$ into two parts; $LE ( S^{1} [0, \tau^{1}_{0} ] )$ up to hitting $\partial B (z, R)$ and $LE ( S^{1} [0, \tau^{1}_{0} ] )$ after last exit from $ B (z,  4 R)$ for several $R$-s ($R= l \epsilon n$, for example). Then we use Proposition 4.6 \cite{Mas} to say that they are independent up to constant.
% \end{rem}

\medskip

\begin{proof}
Throughout the proof, let
\begin{equation*}
\eta^{2}_{R} (\lambda ) := \lambda [s, u],
\end{equation*}
where $u = \inf \{ t \ | \ \lambda (t) \in \partial B (z, R) \}$ and $s = \sup \{ t \le u \ | \ \lambda (t) \in  \partial (\epsilon n B'_{x} ) \}$. Suppose that $LE ( S^{1} [0, \tau^{1}_{0} ] ) [0, \tilde{u}_{1}] \cap A^{r}_{z} \neq \emptyset$ and $LE ( S^{1} [0, \tau^{1}_{0} ] ) [0, \tilde{u}_{1}] \subset B (z, 2^{r} \epsilon n )$. Since $r \le \log_{2} l -3$, we see that 
\begin{equation*}
\eta^{2}_{\frac{n}{16}} ( LE ( S^{1} [0, \tau^{1}_{0} ] ) ) \subset LE ( S^{1} [0, \tau^{1}_{0} ] ) [\tilde{u}_{1}, \tilde{u}_{2}], \ \eta^{2}_{\epsilon l n} ( LE ( S^{1} [0, \tau^{1}_{0} ] ) ) \subset LE ( S^{1} [0, \tau^{1}_{0} ] ) [\tilde{u}_{1}, \tilde{u}_{2}].
\end{equation*}
Therefore, if we write
\begin{align}\label{key-43}
&\tilde{F}^{3} := \Big\{ \tau^{1}_{0} < \tau^{1}_{n}, \ S^{2} [1, T^{2}_{z, \frac{l \epsilon n}{2}} ] \cap \overline{(\epsilon n B'_{x})} = \emptyset, \ S^{2} ( T^{2}_{z, \frac{l \epsilon n}{2}} ) = z', \notag \\
& \ \ \ \ \ \ \ \ \  \ \  \eta^{2}_{\epsilon l n} ( LE ( S^{1} [0, \tau^{1}_{0} ] ) ) \cap S^{2} [ 0 , T^{2}_{z, \frac{l \epsilon n}{2}} ] = \emptyset, \notag \\
&  \ \ \ \  \ \ \ \ \ \   \ \Big( \eta^{1}_{z, 2^{r} \epsilon n} \big( LE ( S^{1} [0, \tau^{1}_{0} ] ) \big) \setminus \eta^{1}_{z, 2^{r-1} \epsilon n} \big( LE ( S^{1} [0, \tau^{1}_{0} ] ) \big) \Big) \cap \overline{(\epsilon n B'_{x})} \neq \emptyset \Big\},
\end{align}
(recall that $\eta^{1}$ was defined as in Definition \ref{escape}) then we can replace $LE ( S^{1} [0, \tau^{1}_{0} ] ) [\tilde{u}_{1}, \tilde{u}_{2}]$ as follows. 
\begin{align}\label{key-44}
&E^{z}_{1} \otimes E^{z}_{2} \Big\{ {\bf 1}_{\tilde{F}^{2} \cap G} \max_{w_{1} \in \partial B ( w, \frac{\epsilon l n}{4} ) } P^{w_{1}}_{3} \Big( LE ( S^{1} [0, \tau^{1}_{0} ] ) [\tilde{u}_{1}, \tilde{u}_{2}] \cap S^{3} [ 0 , \tau^{3}_{n}] = \emptyset \Big) \Big\} \notag \\
&\le E^{z}_{1} \otimes E^{z}_{2} \Big\{ {\bf 1}_{\tilde{F}^{3} \cap G} \max_{w_{1} \in \partial B ( w, \frac{\epsilon l n}{4} ) } P^{w_{1}}_{3} \Big( \eta^{2}_{\frac{n}{16}} ( LE ( S^{1} [0, \tau^{1}_{0} ] ) ) \cap S^{3} [ 0 , \tau^{3}_{n}] = \emptyset \Big) \Big\}.
\end{align}
By Proposition 4.2,  4.4 \cite{Mas} and Proposition 1.5.10 \cite{Law-book90}, the distribution of the loop erasure of a random walk conditioned to be hit the origin before exiting $B (n)$ is equal to (up to multiplicative constants) the distribution of the loop erasure of $S^{1}$ up to exiting $B(z, \frac{n}{4})$. So we have
\begin{align}\label{key-45}
&E^{z}_{1} \otimes E^{z}_{2} \Big\{ {\bf 1}_{\tilde{F}^{3} \cap G} \max_{w_{1} \in \partial B ( w, \frac{\epsilon l n}{4} ) } P^{w_{1}}_{3} \Big( \eta^{2}_{\frac{n}{16}} ( LE ( S^{1} [0, \tau^{1}_{0} ] ) ) \cap S^{3} [ 0 , \tau^{3}_{n}] = \emptyset \Big) \Big\} \notag \\
&\le \frac{C}{n} E^{z}_{1} \otimes E^{z}_{2} \Big\{ {\bf 1}_{\tilde{F}^{4} \cap G} \max_{w_{1} \in \partial B ( w, \frac{\epsilon l n}{4} ) } P^{w_{1}}_{3} \Big( \eta^{2}_{\frac{n}{16}} ( LE ( S^{1} [0, T^{1}_{z, \frac{n}{4}} ] ) ) \cap S^{3} [ 0 , \tau^{3}_{n}] = \emptyset \Big) \Big\},
\end{align}
where
\begin{align}\label{key-46}
&\tilde{F}^{4} := \Big\{  S^{2} [1, T^{2}_{z, \frac{l \epsilon n}{2}} ] \cap \overline{(\epsilon n B'_{x})} = \emptyset, \ S^{2} ( T^{2}_{z, \frac{l \epsilon n}{2}} ) = z', \ \eta^{2}_{\epsilon l n} \big( LE ( S^{1} [0, T^{1}_{z, \frac{n}{4}} ] ) \big) \cap S^{2} [ 0 , T^{2}_{z, \frac{l \epsilon n}{2}} ] = \emptyset, \notag \\
&  \ \ \ \  \ \ \ \ \ \   \ \Big( \eta^{1}_{z, 2^{r} \epsilon n} \big( LE ( S^{1} [0, T^{1}_{z, \frac{n}{4}} ] ) \big) \setminus \eta^{1}_{z, 2^{r-1} \epsilon n} \big( LE ( S^{1} [0, T^{1}_{z, \frac{n}{4}} ] ) \big) \Big) \cap \overline{(\epsilon n B'_{x})} \neq \emptyset \Big\}.
\end{align}

\medskip

We will estimate the expectation in the RHS of \eqref{key-45}. To do it, let $\gamma := LE ( S^{1} [0, T^{1}_{z, \frac{n}{4}} ] )$ and $\tau^{\gamma}_{R} := \inf \{ t \ | \ \gamma (t ) \in \partial B (z, R ) \}.$ Suppose that $\tilde{F}^{4}$ and $\gamma [ \tau^{\gamma}_{\frac{n}{16}}, \tau^{\gamma}_{\frac{n}{4}} ] \cap B ( z, 2^{r+ 4} \epsilon n ) \neq \emptyset$ occur. Then $S^{1}$ returns to $B ( z, 8 \epsilon n )$ after hitting $\partial B (z, 2^{r-1} \epsilon n )$. After $S^{1}$ returns to $B ( z, 8 \epsilon n )$ and goes to $\partial B (z, \frac{n}{16} )$, $S^{1}$ must return to $B (z, 2^{r+ 4} \epsilon n )$. By Proposition 1.5.10 \cite{Law-book90}, that probability is bounded above by $C \frac{ 2^{r} \epsilon n}{n} \frac{\epsilon n}{ 2^{r} \epsilon n} = C \epsilon$. Thus by the strong Markov property, Proposition 1.5.10 \cite{Law-book90}, and \eqref{escape-7},
\begin{align}\label{pol-1}
&\frac{C}{n} P^{z}_{1} \otimes P^{z}_{2} \Big( \tilde{F}^{4}, \ \gamma [ \tau^{\gamma}_{\frac{n}{16}}, \tau^{\gamma}_{\frac{n}{4}} ] \cap B ( z, 2^{r+ 4} \epsilon n ) \neq \emptyset \Big)
\le \frac{C}{n}  \frac{1}{\epsilon n} P^{z}_{2} \big(  S^{2} ( T^{2}_{z, \frac{l \epsilon n}{2}} ) = z' \big) \times \epsilon \notag \\
&\le \frac{C}{n}  \frac{1}{\epsilon n} 2^{-c r} Es ( \epsilon n, n ) P^{z}_{2} \big(  S^{2} ( T^{2}_{z, \frac{l \epsilon n}{2}} ) = z' \big),
\end{align}
for some $c > 0$. So it suffices to consider the case that $\gamma [ \tau^{\gamma}_{\frac{n}{16}}, \tau^{\gamma}_{\frac{n}{4}} ] \cap B ( z, 2^{r+ 4} \epsilon n ) = \emptyset$. With this in mind, define $k_{0} := \min \{ k \ | \ \gamma [ \tau^{\gamma}_{\frac{n}{16}}, \tau^{\gamma}_{\frac{n}{4}} ] \cap B ( z, 2^{-k} n ) = \emptyset \}$. Then we may assume that $2^{- k_{0}} n \ge 2^{r+4} \epsilon n$. Now we consider two cases.

\medskip

\underline{Case-1:} $2^{-k_{0}} n \ge 4 l \epsilon n$. \\
In this case, we have $\eta^{2}_{z, 4 \epsilon l n, \frac{n}{16}} ( \gamma ) = \eta^{1}_{z, \frac{n}{16}} ( \eta^{2}_{z, 4 \epsilon l n, \frac{n}{4} } ( \gamma ) )$ (see Definition \ref{escape} for $\eta^{i}$). Thus by the Harnack principle (see Theorem 1.7.6 \cite{Law-book90}),
\begin{align}\label{pol-2}
&\frac{C}{n} E^{z}_{1} \otimes E^{z}_{2} \Big\{ {\bf 1}_{\tilde{F}^{4} \cap \text{Case-1}} \max_{w_{1} \in \partial B ( w, \frac{\epsilon l n}{4} ) } P^{w_{1}}_{3} \Big( \eta^{2}_{\frac{n}{16}} (\gamma ) \cap S^{3} [ 0 , \tau^{3}_{n}] = \emptyset \Big) \Big\} \notag \\
&\le \frac{C}{n} E^{z}_{1} \otimes E^{z}_{2} \Big\{ {\bf 1}_{\tilde{F}^{4} \cap \text{Case-1}} \max_{w_{1} \in \partial B ( w, \frac{\epsilon l n}{4} ) } P^{w_{1}}_{3} \Big( \eta^{1}_{z, \frac{n}{16}} ( \eta^{2}_{z, 4 \epsilon l n, \frac{n}{4} } ( \gamma ) ) \cap S^{3} [ 0 , \tau^{3}_{n}] = \emptyset \Big) \Big\} \notag \\
&\le \frac{C}{n} E^{z}_{1} \otimes E^{z}_{2} \Big\{ {\bf 1}_{\tilde{F}^{4} \cap \text{Case-1}}  P^{z}_{3} \Big( \eta^{1}_{z, \frac{n}{16}} ( \eta^{2}_{z, 4 \epsilon l n, \frac{n}{4} } ( \gamma ) ) \cap S^{3} [ 0 , \tau^{3}_{n}] = \emptyset \Big) \Big\}.
\end{align}
But by Proposition 4.6 \cite{Mas}, the distribution of $\gamma$ from $z$ to $\partial B (z, \epsilon l n )$ and the distribution of $\gamma$ after last visit to $B(z, 4 \epsilon l n)$ are independent (up to multiplicative constants). Therefore the RHS of \eqref{pol-2} is bounded above by
\begin{equation}\label{pol-3}
\frac{C}{n} P^{z}_{1} \otimes P^{z}_{2} \Big( \tilde{F}^{4} \Big) \times P^{z}_{1} \otimes P^{z}_{3} \Big( \eta^{1}_{z, \frac{n}{16}} ( \eta^{2}_{z, 4 \epsilon l n, \frac{n}{4} } ( \gamma ) ) \cap S^{3} [ 0 , \tau^{3}_{n}] = \emptyset \Big).
\end{equation}
By Lemma \ref{escape-25}, Proposition \ref{escape-3}, and the strong Markov property, we see that \eqref{pol-3} is bounded above by 
\begin{equation}\label{pol-4}
\frac{C}{n} \frac{1}{\epsilon n}  2^{-c r} Es ( \epsilon n, n ) P^{z}_{2} \big(  S^{2} ( T^{2}_{z, \frac{l \epsilon n}{2}} ) = z' \big).
\end{equation}
So we finish Case-1.

\medskip

\underline{Case-2:} $ 2^{r+4} \epsilon n \le 2^{-k_{0}} n \le 4 l \epsilon n$. \\
Suppose that $k_{0} = k$ with $ 2^{r+4} \epsilon n \le 2^{-k} n \le 4 l \epsilon n$. Note that $\eta^{2}_{z, 2^{-k} n, \frac{n}{16}} ( \gamma ) = \eta^{1}_{z, \frac{n}{16}} ( \eta^{2}_{z, 2^{-k} n, \frac{n}{4} } ( \gamma ) )$, and $\Big( \eta^{2}_{z, 2^{-k} n, \frac{n}{4} } ( \gamma ) \setminus \eta^{1}_{z, \frac{n}{16}} ( \eta^{2}_{z, 2^{-k} n, \frac{n}{4} } ( \gamma ) ) \Big) \cap B (z, 2^{-(k-1)} n ) \neq \emptyset$. So by similar arguments using Proposition 4.6 \cite{Mas} as above, we see that 
\begin{align}\label{pol-4}
&\frac{C}{n} E^{z}_{1} \otimes E^{z}_{2} \Big\{ {\bf 1}_{\tilde{F}^{4} \cap \{ k_{0} = k \}} \max_{w_{1} \in \partial B ( w, \frac{\epsilon l n}{4} ) } P^{w_{1}}_{3} \Big( \eta^{2}_{\frac{n}{16}} (\gamma ) \cap S^{3} [ 0 , \tau^{3}_{n}] = \emptyset \Big) \Big\} \notag \\
&\le \frac{C}{n} P^{z}_{1} \otimes P^{z}_{2} \Big( \Big( \eta^{2}_{z, 2^{-k} n, \frac{n}{4} } ( \gamma ) \setminus \eta^{1}_{z, \frac{n}{16}} ( \eta^{2}_{z, 2^{-k} n, \frac{n}{4} } ( \gamma ) ) \Big) \cap B (z, 2^{-(k-1)} n ) \neq \emptyset, \ \tilde{F}^{4}_{k} \Big) \notag \\
&\le \frac{C}{n} P^{z}_{1} \otimes P^{z}_{2} \Big( \Big( \eta^{2}_{z, 2^{-k} n, \frac{n}{4} } ( \gamma ) \setminus \eta^{1}_{z, \frac{n}{16}} ( \eta^{2}_{z, 2^{-k} n, \frac{n}{4} } ( \gamma ) ) \Big) \cap B (z, 2^{-(k-1)} n ) \neq \emptyset \Big) \times P^{z}_{1} \otimes P^{z}_{2} \Big( \tilde{F}^{4}_{k} \Big) \notag \\
&\le \frac{C}{n} Es ( 2^{-k} n, n ) 2^{- c k} P^{z}_{1} \otimes P^{z}_{2} \Big( \tilde{F}^{4}_{k} \Big),
\end{align}
for some $c > 0$. Here $\tilde{F}^{4}_{k}$ is defined by
\begin{align*}
&\tilde{F}^{4}_{k} := \Big\{  S^{2} [1, T^{2}_{z, \frac{l \epsilon n}{2}} ] \cap \overline{(\epsilon n B'_{x})} = \emptyset, \ S^{2} ( T^{2}_{z, \frac{l \epsilon n}{2}} ) = z', \ \eta^{2}_{2^{-(k+2)} n} \big( \gamma \big) \cap S^{2} [ 0 , T^{2}_{z, \frac{l \epsilon n}{2}} ] = \emptyset, \notag \\
&  \ \ \ \  \ \ \ \ \ \   \ \Big( \eta^{1}_{z, 2^{r} \epsilon n} \big( \gamma \big) \setminus \eta^{1}_{z, 2^{r-1} \epsilon n} \big( \gamma \big) \Big) \cap \overline{(\epsilon n B'_{x})} \neq \emptyset \Big\}.
\end{align*}
But by Lemma \ref{escape-25}, Proposition \ref{escape-3}, and the strong Markov property, RHS of \eqref{pol-4} is bounded above by
\begin{equation}\label{pol-5}
\frac{C}{n} \frac{1}{\epsilon n}  2^{- c k}  2^{-c r} Es ( \epsilon n, n ) P^{z}_{2} \big(  S^{2} ( T^{2}_{z, \frac{l \epsilon n}{2}} ) = z' \big),
\end{equation}
for some $c > 0$. Taking sum for $k$, we have 
\begin{align}\label{pol-6}
&\frac{C}{n} E^{z}_{1} \otimes E^{z}_{2} \Big\{ {\bf 1}_{\tilde{F}^{4} \cap \text{Case-2}} \max_{w_{1} \in \partial B ( w, \frac{\epsilon l n}{4} ) } P^{w_{1}}_{3} \Big( \eta^{2}_{\frac{n}{16}} (\gamma ) \cap S^{3} [ 0 , \tau^{3}_{n}] = \emptyset \Big) \Big\} \notag \\
&\le \frac{C}{n} \frac{1}{\epsilon n}  2^{-c r} Es ( \epsilon n, n ) P^{z}_{2} \big(  S^{2} ( T^{2}_{z, \frac{l \epsilon n}{2}} ) = z' \big).
\end{align}
So we finish Case-2, and Lemma \ref{hosoku-3} is proved.
\end{proof}

\medskip

Now we return to the proof of Theorem \ref{key thm}. Using Lemma \ref{hosoku-1}, \ref{hosoku-2}, \ref{key-21}, and \ref{hosoku-3}, by \eqref{key-38},
\begin{align}\label{key-64}
&P^{0}_{1} \otimes P^{z}_{2} \otimes P^{w}_{3} \Big( F^{1}, \ q^{(1)}= r, \ S^{2} ( T^{2}_{z, \frac{l \epsilon n}{2}} ) = z' \Big) \le E^{0}_{1} \otimes E^{z}_{2} \Big( \tilde{p}_{1} {\bf 1}_{F^{2}} \Big) \notag \\
&\le \frac{C}{\epsilon n} \frac{1}{\epsilon l n} \frac{1}{n} \frac{1}{\epsilon n} 2^{-\delta r} Es (\epsilon n,  n ) Es (\epsilon n, \epsilon l  n ) P^{z}_{2} \Big( S^{2} ( T^{2}_{z, \frac{l \epsilon n}{2}} ) = z' \Big).
\end{align}
Taking sum for $z' \in \partial B (z, \frac{ \epsilon l n}{2} )$ and $0 \le r \le \log_{2} l -3 $, by \eqref{key-12}, we see that 
\begin{equation}\label{key-65}
P^{0}_{1} \otimes P^{z}_{2} \otimes P^{w}_{3} (F^{1}, \ q^{(1)} \le \log_{2} l -3) \le \frac{C}{\epsilon n} \frac{1}{\epsilon l n} \frac{1}{n} \frac{1}{\epsilon n}  Es (\epsilon n,  n ) Es (\epsilon n, \epsilon l  n ).
\end{equation}
(Recall that $F^{1}$ was defined in \eqref{key-10}.) For the case that $q^{(1)} \ge \log_{2} l -3$, by the same argument as above, one can prove that
\begin{equation}\label{key-66}
P^{0}_{1} \otimes P^{z}_{2} \otimes P^{w}_{3} (F^{1}, \ q^{(1)} \ge \log_{2} l -3) \le \frac{C}{\epsilon n} \frac{1}{\epsilon l n} \frac{1}{n} \frac{1}{\epsilon n}  Es (\epsilon n,  n ) Es (\epsilon n, \epsilon l  n ).
\end{equation} 
(We shall omit the proof of \eqref{key-66} and leave it to the reader.) Taking sum for $z \in \partial (\epsilon n B'_{x})$ and $w \in \partial (\epsilon n B'_{y})$, by \eqref{key-9}, we have
\begin{align}\label{key-67}
&P \big( LE ( S[0, T^{x} ] ) [0, \overline{\sigma}'_{1}] \cap S[ T^{x}, \tau_{n} ] = \emptyset, \ LE ( S[0, T^{y} ] ) [0, \overline{\sigma}'_{2}] \cap S[ T^{y}, \tau_{n} ] = \emptyset, \ T^{x} < T^{y} < \tau_{n} \big) \notag \\
&\le \frac{C \epsilon }{l} Es (\epsilon n,  n ) Es (\epsilon n, \epsilon l  n ).
\end{align}
(Note that $\sharp \{ z \in \partial (\epsilon n B'_{x}) \cap \mathbb{Z}^{3} \} \le C (\epsilon n )^{2}$.)  Combining \eqref{key-67} with \eqref{key-5} and \eqref{key-8}, we finish the proof of Theorem \ref{key thm}.
\end{proof}

\subsection{Estimates of the number of boxes hit by ${\cal K}$}

Now we are ready to estimate the first and the second moment of the number of cubes hit by ${\cal K}$. For $\epsilon > 0$, let 
\begin{equation}\label{number-box}
Y^{\epsilon} := \sharp \big\{ x \in \mathbb{Z}^{3} \ | \ \epsilon B_{x} \subset D_{\frac{2}{3}} \setminus D_{\frac{1}{3}}, \ {\cal K} \cap \epsilon B_{x} \neq \emptyset \big\}
\end{equation}
(Recall that $B_{x}$ was defined in \eqref{box}.) In this subsection, we will give a lower bound of $Y^{\epsilon}$ in Corollary \ref{positive-low}. In order to prove it, we first estimate the second moment of $Y^{\epsilon}$ (see Corollary \ref{second-moment} below) using Theorem \ref{key thm}. Then we also give a lower bound of $E ( Y^{\epsilon })$ in Proposition \ref{first-moment-low}, and using the second moment method we get Corollary \ref{positive-low} in the end of this subsection.

\medskip

Theorem \ref{key thm} and estimates of escape probabilities introduced as in Section 2.2 immediately show the following corollary, which gives a second moment estimate of $Y^{\epsilon}$.

\begin{cor}\label{second-moment}
Take $\epsilon > 0$ and fix $n=n_{\epsilon} = 2^{j_{\epsilon }}$ such that \eqref{skorohod-2} holds. Then there exists an absolute constant $C < \infty$ such that
\begin{equation}\label{second-1}
E \big( (Y^{\epsilon})^{2} \big) \le C \Big\{ \epsilon^{-2} Es ( \epsilon n , n ) \Big\}^{2}.
\end{equation}
\end{cor}

\begin{proof}
By Theorem \ref{key thm}, we have
\begin{align}\label{second-2}
&E \big( (Y^{\epsilon})^{2} \big) \le \sum_{|x|, |y| \in [\frac{1}{3 \epsilon }, \frac{2}{3 \epsilon } ] } P \Big( {\cal K} \cap \epsilon B_{x} \neq \emptyset, \ {\cal K} \cap \epsilon B_{y} \neq \emptyset \Big) \notag \\
&\le C \sum_{|x| \in [\frac{1}{3 \epsilon }, \frac{2}{3 \epsilon } ] } \sum_{l=1}^{\frac{2}{\epsilon}} l^{2} Es ( \epsilon n , l \epsilon n) Es ( \epsilon n , n )  \frac{\epsilon}{l} + \sum_{|x| \in [\frac{1}{3 \epsilon }, \frac{2}{3 \epsilon } ] } P \Big( {\cal K} \cap \epsilon B_{x} \neq \emptyset \Big) \notag \\
&\le C \epsilon^{-2}  Es ( \epsilon n , n ) \sum_{l=1}^{\frac{2}{\epsilon}} l Es ( \epsilon n , l \epsilon n) 
\end{align}

By Lemma \ref{escape-8}, for any $\delta > 0$ there exists $C= C_{\delta } < \infty$ such that
\begin{equation*}
(\epsilon l n )^{\alpha + \delta} Es ( \epsilon l n ) \le C_{\delta } n^{\alpha + \delta } Es (n).
\end{equation*}
Dividing both sides by $(\epsilon l n)^{\alpha + \delta} Es ( \epsilon n )$ and using \eqref{escape-4}, we have
\begin{equation*}
Es ( \epsilon n, \epsilon l n ) \le C_{\delta} \epsilon^{-\alpha - \delta } l^{-\alpha - \delta} Es ( \epsilon n, n ).
\end{equation*}
Fix $\delta > 0$ so that $1- \alpha - \delta > 0$. Combining this with \eqref{second-2}, we have 
\begin{align}\label{second-3}
&\epsilon^{-2}  Es ( \epsilon n , n ) \sum_{l=1}^{\frac{2}{\epsilon}} l Es ( \epsilon n , l \epsilon n) \le C_{\delta} \epsilon^{-2}  Es ( \epsilon n , n ) \sum_{l=1}^{\frac{2}{\epsilon}} l \epsilon^{-\alpha - \delta } l^{-\alpha - \delta} Es ( \epsilon n, n ) \notag \\
&= C_{\delta} \epsilon^{-2-\alpha - \delta } Es ( \epsilon n, n )^{2} \sum_{l=1}^{\frac{2}{\epsilon}} l^{1-\alpha - \delta } \le C_{\delta} \epsilon^{-2-\alpha - \delta } Es ( \epsilon n, n )^{2} \epsilon^{-2 + \alpha + \delta} = C_{\delta} \epsilon^{-4} Es ( \epsilon n, n )^{2},
\end{align}
which finishes the proof.
\end{proof}

\medskip

Take $\epsilon > 0$. Fix $n=n_{\epsilon} = 2^{j_{\epsilon }}$ such that \eqref{skorohod-2} holds. Recall that $B'_{x} := \prod_{i=1}^{3} [x_{i} -2, x_{i} + 2]$ was defined just before \eqref{key-4}. In the proof of Lemma 7.1 \cite{SS}, it was shown that 
\begin{equation*}
P \big( LE ( S [0, \tau_{n} ] ) \cap \epsilon n B'_{x} \neq \emptyset \big) \le C \epsilon Es ( \epsilon n , n ),
\end{equation*}
for $x \in \mathbb{Z}^{3}$ with $ \frac{1}{3} \le | \epsilon x | \le \frac{2}{3}$. Using this and \eqref{skorohod-2}, we have 
\begin{align*}
& P \big( {\cal K} \cap \epsilon B_{x} \neq \emptyset \big) \\
&\le P \big( {\cal K} \cap \epsilon B_{x} \neq \emptyset, \   d_{\text{H}} ( \text{LEW}_{n}, {\cal K} ) < \epsilon^{2} \big) + P \big( {\cal K} \cap \epsilon B_{x} \neq \emptyset, \   d_{\text{H}} ( \text{LEW}_{n}, {\cal K} ) \ge \epsilon^{2} \big) \\
&\le P \big( \text{LEW}_{n} \cap \epsilon B'_{x} \neq \emptyset \big) + \epsilon^{100} \le C \epsilon Es ( \epsilon n , n ).
\end{align*}
So we see that 
\begin{equation}\label{first-moment-up}
E ( Y^{\epsilon} ) \le C \epsilon^{-2} Es ( \epsilon n, n ).
\end{equation}

\medskip

In the next proposition, we will give the lower bound of $E ( Y^{\epsilon} ) $. As Remark 7.2 \cite{SS} states, its proof is almost included in the proof of Lemma 7.1 \cite{SS}. However we will give the proof for completeness.

\begin{prop}\label{first-moment-low}
Take $\epsilon > 0$ and fix $n=n_{\epsilon} = 2^{j_{\epsilon }}$ such that \eqref{skorohod-2} holds. Then there exists an absolute constant $c > 0$ such that 
\begin{equation}\label{moment-1}
E ( Y^{\epsilon} ) \ge c \epsilon^{-2} Es ( \epsilon n, n ).
\end{equation}
\end{prop}

\begin{proof}
Take $x=(x_{1}, x_{2}, x_{3} ) \in \mathbb{Z}^{3}$ with $ \frac{1}{3} \le | \epsilon x | \le \frac{2}{3}$. Let $x'= (x_{1} + \frac{1}{2}, x_{2}+ \frac{1}{2}, x_{3} + \frac{1}{2})$ be the center of $B_{x}$. Let $y = \epsilon n x'$. We write $B_{1} = B (y, \frac{\epsilon n}{1000} )$, $B_{2} = B (y, \frac{\epsilon n}{3} )$ and $B_{3} = B (y, \frac{\epsilon n}{2} )$  throughout the proof.

Using \eqref{skorohod-2}, it suffices to show that 
\begin{equation}\label{moment-2}
P \big( LE ( S [0, \tau_{n} ] ) \cap B_{3} \neq \emptyset \big) \ge c \epsilon Es ( \epsilon n , n ),
\end{equation}
Let $T := \max \{ t \ | \ S(t) \in B_{1} \}$ and $\tau := \min \{ t \ | \ LE ( S[0, T ] ) (t) \in B_{3} \}$. Then we have 
\begin{equation*}
P \big( LE ( S [0, \tau_{n} ] ) \cap B_{3} \neq \emptyset \big) \ge P \big( T < \tau_{n}, \  LE ( S[0, T ] ) [0, \tau ] \cap S[T+1, \tau_{n} ] = \emptyset \big).
\end{equation*}
By the last exit decomposition as in \eqref{key-9} and reversing a path, we have 
\begin{align*}
&P \big( T < \tau_{n}, \  LE ( S[0, T ] ) [0, \tau ] \cap S[T+1, \tau_{n} ] = \emptyset \big) \\
&\ge \sum_{z \in \partial_{i} B_{1} } P^{z}_{1} \otimes P^{z}_{2} \Big( \tau^{1}_{0} < \tau^{1}_{n}, \ S^{2}[1, \tau^{2}_{n}] \cap B_{1} = \emptyset, \ LE ( S^{1} [0, \tau^{1}_{0} ] ) [ \sigma_{1}, \sigma_{2} ] \cap S^{2} [0, \tau^{2}_{n} ] = \emptyset \Big),
\end{align*}
where $\tau^{1}_{0} = \inf \{ t \ | \ S^{1} (t) = 0 \}$, $\sigma_{1} = \max \{ t \ | \ LE ( S^{1} [0, \tau^{1}_{0} ] ) (t) \in B_{3} \}$ and $\sigma_{2} = \text{len} LE ( S^{1} [0, \tau^{1}_{0} ] )$. 

Let $\tau^{2}_{B_{2}} := \{ t \ | \ S^{2} (t) \in \partial B_{2} \}$. Then for each $z \in \partial_{i} B_{1}$,
\begin{align*}
&P^{z}_{1} \otimes P^{z}_{2} \Big( \tau^{1}_{0} < \tau^{1}_{n}, \ S^{2}[1, \tau^{2}_{n}] \cap B_{1} = \emptyset, \ LE ( S^{1} [0, \tau^{1}_{0} ] ) [ \sigma_{1}, \sigma_{2} ] \cap S^{2} [0, \tau^{2}_{n} ] = \emptyset \Big) \\
&\ge P^{z}_{2} \Big(  S^{2}[1, \tau^{2}_{B_{2}}] \cap B_{1} = \emptyset \Big) \\
&\ \ \times E^{z}_{1} \Big\{ {\bf 1}_{ \{ \tau^{1}_{0} < \tau^{1}_{n} \} } \min_{w \in \partial B_{2}} P^{w}_{2} \Big( S^{2} [0, \tau^{2}_{n}] \cap B_{1} = \emptyset, \ LE ( S^{1} [0, \tau^{1}_{0} ] ) [ \sigma_{1}, \sigma_{2} ] \cap S^{2} [0, \tau^{2}_{n} ] = \emptyset \Big) \Big\}.
\end{align*}
However, by the Harnack principle (see Theorem 1.7.6 \cite{Law-book90}) and Proposition 1.5.10 \cite{Law-book90}, for any $w \in \partial B_{2}$, 
\begin{align*}
&P^{w}_{2} \Big( S^{2} [0, \tau^{2}_{n}] \cap B_{1} = \emptyset, \ LE ( S^{1} [0, \tau^{1}_{0} ] ) [ \sigma_{1}, \sigma_{2} ] \cap S^{2} [0, \tau^{2}_{n} ] = \emptyset \Big) \\
&= P^{w}_{2} \Big( LE ( S^{1} [0, \tau^{1}_{0} ] ) [ \sigma_{1}, \sigma_{2} ] \cap S^{2} [0, \tau^{2}_{n} ] = \emptyset \Big) \\
& - P^{w}_{2} \Big( S^{2} [0, \tau^{2}_{n}] \cap B_{1} \neq \emptyset, \ LE ( S^{1} [0, \tau^{1}_{0} ] ) [ \sigma_{1}, \sigma_{2} ] \cap S^{2} [0, \tau^{2}_{n} ] = \emptyset \Big) \\
&\ge c P^{z}_{2} \Big( LE ( S^{1} [0, \tau^{1}_{0} ] ) [ \sigma_{1}, \sigma_{2} ] \cap S^{2} [0, \tau^{2}_{n} ] = \emptyset \Big) \\
& - P^{w}_{2} \Big( S^{2} [0, \tau^{2}_{n}] \cap B_{1} \neq \emptyset, \ LE ( S^{1} [0, \tau^{1}_{0} ] ) [ \sigma_{1}, \sigma_{2} ] \cap S^{2} [0, \tau^{2}_{n} ] = \emptyset \Big) \\
&\ge \frac{c}{2} P^{z}_{2} \Big( LE ( S^{1} [0, \tau^{1}_{0} ] ) [ \sigma_{1}, \sigma_{2} ] \cap S^{2} [0, \tau^{2}_{n} ] = \emptyset \Big).
\end{align*}
Note that the last inequality holds since we let the radius of $B_{1}$ be $\frac{\epsilon n}{1000}$.

Thus,
\begin{align*}
&P^{z}_{1} \otimes P^{z}_{2} \Big( \tau^{1}_{0} < \tau^{1}_{n}, \ S^{2}[1, \tau^{2}_{n}] \cap B_{1} = \emptyset, \ LE ( S^{1} [0, \tau^{1}_{0} ] ) [ \sigma_{1}, \sigma_{2} ] \cap S^{2} [0, \tau^{2}_{n} ] = \emptyset \Big) \\
&\ge \frac{c}{\epsilon n} P^{z}_{1} \otimes P^{z}_{2} \Big( \tau^{1}_{0} < \tau^{1}_{n}, \ LE ( S^{1} [0, \tau^{1}_{0} ] ) [ \sigma_{1}, \sigma_{2} ] \cap S^{2} [0, \tau^{2}_{n} ] = \emptyset \Big).
\end{align*}

Let $B = B (z, \frac{n}{4} )$ and $\tau_{i} = \inf \{ t \ | \ S^{i} (t) \in \partial B \}$ for $i=1, 2$. We write $\gamma = LE ( S^{1} [0, \tau_{1} ] )$ and let $\sigma := \max \{ t \ | \ \gamma (t) \in B_{3} \}$. We define events $F$ and $G$ by 
\begin{align*}
&F = \Big\{ \text{dist} \Big( \gamma ( \text{len} \gamma ), S^{2} [0, \tau_{2} ] \Big) \ge \frac{n}{12}, \ \text{dist} \Big( \gamma [\sigma ,  \text{len} \gamma ), S^{2} ( \tau_{2} ) \Big) \ge \frac{n}{12} \Big\}, \\
& G = \Big\{ \gamma [0, \sigma ] \cap B \big( \gamma (\text{len} \gamma ) , \frac{n}{12} \big) = \emptyset \Big\}.
\end{align*}
By Proposition 1.5.10 \cite{Law-book90}, we have 
\begin{equation*}
P^{z}_{1} \otimes P^{z}_{2} \Big(  G^{c} \Big) \le  \frac{C \epsilon }{n}
\end{equation*}

By the strong Markov property,
\begin{align*}
&P^{z}_{1} \otimes P^{z}_{2} \Big( \tau^{1}_{0} < \tau^{1}_{n}, \ LE ( S^{1} [0, \tau^{1}_{0} ] ) [ \sigma_{1}, \sigma_{2} ] \cap S^{2} [0, \tau^{2}_{n} ] = \emptyset \Big) \\
&\ge P^{z}_{1} \otimes P^{z}_{2} \Big( \tau^{1}_{0} < \tau^{1}_{n}, \ F, \ G, \ \gamma [\sigma ,  \text{len} \gamma ) \cap S^{2} [0, \tau_{2} ] = \emptyset, \ S^{1} [ \tau_{1}, \tau^{1}_{0} ] \cap S^{2} [ \tau_{2}, \tau^{2}_{n} ] = \emptyset, \\
& \ \ \ \ \ \ \ \ \ \ \ \ \ \ \ \ \ \big(  S^{1} [ \tau_{1}, \tau^{1}_{0} ] \cap B \big) \subset B \big( S^{1} ( \tau_{1} ), \frac{n}{12} \big), \  \big(  S^{2} [ \tau_{2}, \tau^{2}_{n} ] \cap B \big) \subset B \big( S^{2} ( \tau_{2} ), \frac{n}{12} \big) \Big) \\
&\ge \frac{c}{n} P^{z}_{1} \otimes P^{z}_{2} \Big( F, \ G, \ \gamma [\sigma ,  \text{len} \gamma ) \cap S^{2} [0, \tau_{2} ] = \emptyset \Big) \\
&\ge \frac{c}{n} P^{z}_{1} \otimes P^{z}_{2} \Big( F,  \ \gamma [\sigma ,  \text{len} \gamma ) \cap S^{2} [0, \tau_{2} ] = \emptyset \Big) - \frac{C \epsilon}{n}.
\end{align*}

However, by Lemma \ref{escape-17} and \ref{escape-25}, we have 
\begin{equation*}
P^{z}_{1} \otimes P^{z}_{2} \Big( F,  \ \gamma [\sigma ,  \text{len} \gamma ) \cap S^{2} [0, \tau_{2} ] = \emptyset \Big) \ge c P^{z}_{1} \otimes P^{z}_{2} \Big( \gamma [\sigma ,  \text{len} \gamma ) \cap S^{2} [0, \tau_{2} ] = \emptyset \Big) \ge c Es ( \epsilon n, n ).
\end{equation*}
Combining these estimates, we see that 
\begin{equation}\label{sansho}
P \big( LE ( S [0, \tau_{n} ] ) \cap B_{3} \neq \emptyset \big) \ge \sum_{z \in \partial_{i} B_{1} } \frac{c}{\epsilon n} \frac{1}{n} Es ( \epsilon n, n ) \ge c \epsilon Es ( \epsilon n, n ),
\end{equation}
which finishes the proof.
\end{proof}

\medskip

By Corollary \ref{second-moment}, Proposition \ref{first-moment-low} and the second moment method, we get the following lower bound of $Y^{\epsilon}$.

\begin{cor}\label{positive-low}
Take $\epsilon > 0$ and fix $n=n_{\epsilon} = 2^{j_{\epsilon }}$ such that \eqref{skorohod-2} holds. Then there exists an absolute constant $c > 0$ such that 
\begin{equation}\label{moment-1}
P \Big( Y^{\epsilon} \ge c \epsilon^{-2} Es ( \epsilon n, n ) \Big) \ge c.
\end{equation}
\end{cor}

\section{Tightness of $\frac{Y^{\epsilon}}{ E(Y^{\epsilon} ) }$}

As we discussed in Section 1.2, in order to show that the Hausdorff dimension of ${\cal K}$ is equal to $2- \alpha$ almost surely ($\alpha$ is the exponent as in Theorem \ref{escape-5}), we need to improve Corollary \ref{positive-low}, i.e., we have to prove that for all $r > 0$ there exists $c_{r} > 0$ such that 
\begin{equation}\label{imp}
P \Big( Y^{\epsilon} \ge c_{r} \epsilon^{-2} Es ( \epsilon n, n ) \Big) \ge 1-r,
\end{equation}
 where $\epsilon > 0$ is an arbitrary positive number and $n=n_{\epsilon} = 2^{j_{\epsilon }}$ is an integer satisfying \eqref{skorohod-2}. 
 
 In order to prove \eqref{imp}, again we use the coupling of ${\cal K}$ and $\text{LEW}_{n}$ explained as in Section 1.2. Then \eqref{imp} boils down to the corresponding estimates for LERW as follows. Let $Y^{\epsilon}_{n}$ be the number of $\epsilon n$-cubes $n b_{x}$ with $\frac{1}{3} \le |\epsilon x| \le \frac{2}{3}$ such that $LE (S[ 0, \tau_{n} ] )$ hits $n b_{x}$. Then \eqref{imp} is reduced to proving that for all $r > 0$ there exists $c_{r} > 0$ such that
\begin{equation}\label{rough-1-1}
P \Big(  Y^{\epsilon}_{n} \ge c_{r} \epsilon^{-2} Es ( \epsilon n, n) \Big) \ge 1- r.
\end{equation}

To show \eqref{rough-1-1}, we will use ``iteration arguments" as in the proof of Theorem 6.7 of \cite{BM} and Theorem 8.2.6 of \cite{Shi-gr} where exponential lower tail bounds of $M_{n}$ were established for $d=2$ (\cite{BM}) and $d=3$ (\cite{Shi-gr}). We explain it here. Take integer $N$. Define a sequence of boxes $A_{i}$ by $A_{i} = [- \frac{n}{3} - \frac{i n }{N}, \frac{n}{3} + \frac{i n }{N} ]^{3}$ for $0 \le i \le \frac{N}{6}$. We write $\gamma = LE (S[ 0, \tau_{n} ] )$ and let $\tau (i) = \tau^{\gamma} (i)$ be the first time that $\gamma$ exits from $A_{i}$. It turns out that the expected number of $\epsilon n$-cubes hit by $\gamma_{i} := \gamma [\tau (i), \tau (i+1 ) ]$ is of order $( \epsilon N )^{-2} Es (\epsilon n, \frac{n}{N} )$ for each $i$. Conditioned on $\gamma [0, \tau (i)] = \lambda$ for a given path $\lambda$, we are interested in the probability that the number of $\epsilon n$-cubes hit by $\gamma_{i}$ is bigger than $c_{1} ( \epsilon N )^{-2} Es (\epsilon n, \frac{n}{N} )$ (we denote this probability by $p (\lambda)$). The domain Markov property (see Lemma \ref{domain}) tells that we need to study a random walk conditioned not to intersect $\lambda$. We will study such a conditioned random walk in Section 4.1 and show that there exists a universal constant $c_{1} > 0$ which does not depend on $\lambda$ such that the probability $p ( \lambda ) $ above is larger than $c_{1}$ for every $i$ (see Lemma \ref{imp-38}). Using this and the domain Markov property, we have 

\begin{equation}\label{rough-2-1}
P \Big(  Y^{\epsilon}_{n} \le c_{1} ( \epsilon N )^{-2} Es (\epsilon n, \frac{n}{N} ) \Big) \le (1- c_{1})^{\frac{N}{6}}.
\end{equation}
Since $Es ( \epsilon n, \frac{n}{N} ) \ge c_{2} Es ( \epsilon n, n)$ for some absolute constant $c_{2} >0$, taking $N = N_{r}$ such that $(1- c_{1})^{\frac{N_{r}}{6}} < r$ first, then letting $c_{r} := c_{1} c_{2} N_{r}^{-2}$, we get \eqref{rough-1-1} and \eqref{imp} (see Proposition \ref{imp-43} and \ref{imp-46}).

%  In order to prove \eqref{imp}, we will use ``iteration arguments" as we explained in Section 1.2. The arguments were used in Theorem 6.7 \cite{BM} and in Theorem 8.3.1 \cite{Shi-gr} where exponential lower tail bounds for $\text{len} LE ( S[0, \tau_{n} ] )$ are established for $d=2$ and $d=3$, respectively. To follow such strategies, we need to study random walks conditioned not to hit a given path. In Section 4.1, we will consider those conditioned random walks.
%  
%  Using estimates derived in Section 4.1, we then prove \eqref{imp} in Section 4.2. In the proof of \eqref{imp}, we will use the iteration arguments explained in Section 1.2.
 
 \subsection{Loop-erasure of conditioned random walks}
 
Given a box and a simple path $\gamma$ contained in the inside of the box except the end point $\gamma ( \text{len} \gamma )$ which is lying on the boundary of the box. Following same spirits of Theorem 6.7 \cite{BM} and Theorem 8.2.6 \cite{Shi-gr}, we are interested in a random walk $X$ staring from $\gamma ( \text{len} \gamma )$ conditioned that $X [1, \tau ] \cap \gamma = \emptyset$ for some stopping time $\tau$. Estimates of such a conditioned random walk $X$ are crucial to prove \eqref{imp}. In this subsection, we will study $X$.

\medskip

We begin with some notation. 

\begin{dfn}\label{imp-notation}
Let $M \ge 20$. Fix $\epsilon > 0$ and take $n=n_{\epsilon} = 2^{j_{\epsilon }}$ such that \eqref{skorohod-2} holds. Define
\begin{align}\label{imp-1}
&k_{i} = \frac{1}{3} + \frac{i}{M} \text{ for } i=0, 1, \cdots , \frac{M}{20}, \notag \\
&D(i) = [-k_{i}, k_{i}]^{3}, \ A(i) = D(i+1) \setminus D(i), \ D_{i, n} = n D(i) \cap \mathbb{Z}^{3}, \ A_{i, n} = n A(i) \cap \mathbb{Z}^{3}.
\end{align}

Take $i \in \{ 0, 1, \cdots , \frac{M}{20} \}$. Suppose that $\gamma = \gamma_{i}$ is a simple path in $\mathbb{Z}^{3}$ with $\gamma (0) =0$, $\gamma [0, \text{len} \gamma -1 ] \subset D_{i, n}$ and $\gamma ( \text{len} \gamma ) \in \partial D_{i, n}$. Let $v = \gamma ( \text{len} \gamma )$. We denote a face of $\partial D_{i, n}$ containing $v$ by $\pi_{1}$. Let $\ell_{1}$ be the line segment starting at $v$ and terminating at $\partial D_{i+1, n}$ which is perpendicular to $\partial D_{i, n}$. We denote the middle point of $\ell_{1}$ by $o_{1}$. We define a set $F^{\gamma }_{i, n}$ by $F^{\gamma }_{i, n} := \big( o_{1} + [ - \frac{n}{8M}, \frac{n}{8M} ]^{3} \big) \cap \mathbb{Z}^{3}$.

Let $X= X^{\gamma }$ be the random walk conditioned to hit $\partial B_{n}$ before hitting $\gamma $, i.e., $X$ is the simple random walk $S$ started at $v$ conditioned on $\{ S[1, \tau_{n} ] \cap \gamma = \emptyset \}$.
\end{dfn}

\medskip

Suppose that $x, y \in \mathbb{Z}^{3}$ satisfy $\epsilon n B'_{x} \subset F^{\gamma }_{i, n}$ and $\epsilon n B'_{y} \subset F^{\gamma }_{i, n}$. ($B'_{x}$ was defined just before \eqref{key-4}.) Let $l:= |x-y|$. 
 
As in \eqref{key-5}, we are interested in 
\begin{equation}\label{imp-2}
P_{X} \Big(  LE ( X [0, \tau^{X}_{n} ] ) \cap \epsilon n B'_{x} \neq \emptyset, \  LE ( X[0, \tau^{X}_{n} ] ) \cap \epsilon n B'_{y} \neq \emptyset \Big),
\end{equation}
where we write $P_{X}$ for the probability law of $X$ and let $\tau^{X}_{n} := \inf \{ t \ | \ X(t) \in \partial B_{n} \}$. For this probability, we have the following lemma, which is an analog of Theorem \ref{key thm} for the probability that the loop erasure of $X$ hits two distinct cubes.

\begin{lem}\label{imp-3}
Let $X$ be the conditioned random walk defined in Definition \ref{imp-notation} and suppose that $x, y \in \mathbb{Z}^{3}$ satisfy $\epsilon n B'_{x} \subset F^{\gamma }_{i, n}$ and $\epsilon n B'_{y} \subset F^{\gamma }_{i, n}$ (see Definition \ref{imp-notation} for $F^{\gamma }_{i, n}$). Then there exists an absolute constant $C < \infty$ such that 
\begin{equation}\label{imp-4}
P_{X} \Big(  LE ( X [0, \tau^{X}_{n} ] ) \cap \epsilon n B'_{x} \neq \emptyset, \  LE ( X[0, \tau^{X}_{n} ] ) \cap \epsilon n B'_{y} \neq \emptyset \Big) \le \frac{C \epsilon M}{l} Es (\epsilon n, \epsilon l n ) Es (\epsilon n, \frac{n}{M} ),
\end{equation}
where $l= |x-y|$.
\end{lem}

\begin{proof}
Throughout the proof, we will use same notation defined in the proof of Theorem \ref{key thm}.
 
Define
\begin{equation}\label{imp-5}
T^{x}_{X} := \max \{ t \le \tau^{X}_{n} \ | \ X(t) \in \partial (\epsilon n B'_{x} ) \}, \ T^{y}_{X} := \max \{ t \le \tau^{X}_{n} \ | \ X(t) \in \partial (\epsilon n B'_{y} ) \}.
\end{equation}
As in \eqref{key-5}, it suffices to estimate
\begin{align}\label{imp-6}
&P_{X} \Big( LE ( X [0, \tau^{X}_{n} ] ) \cap \epsilon n B'_{x} \neq \emptyset, \  LE ( X[0, \tau^{X}_{n} ] ) \cap \epsilon n B'_{y} \neq \emptyset, \ T^{x}_{X} < T^{y}_{X} \Big) \notag \\
&= \frac{ P^{v} \Big( LE ( S [0, \tau_{n} ] ) \cap \epsilon n B'_{x} \neq \emptyset, \  LE ( S[0, \tau_{n} ] ) \cap \epsilon n B'_{y} \neq \emptyset, \ T^{x} < T^{y}, \ S[1, \tau_{n} ] \cap \gamma = \emptyset \Big) }{ P^{v} \Big( S[1, \tau_{n} ] \cap \gamma = \emptyset \Big) }.
\end{align}

By the last exit decomposition as in \eqref{key-9}, we have 
\begin{align}\label{imp-7}
&P^{v} \Big( LE ( S [0, \tau_{n} ] ) \cap \epsilon n B'_{x} \neq \emptyset, \  LE ( S[0, \tau_{n} ] ) \cap \epsilon n B'_{y} \neq \emptyset, \ T^{x} < T^{y}, \ S[1, \tau_{n} ] \cap \gamma = \emptyset \Big) \notag \\
&\le C \sum_{z \in \partial (\epsilon n B'_{x})} \sum_{w \in \partial (\epsilon n B'_{y})}  P^{v}_{1} \otimes P^{z}_{2} \otimes P^{w}_{3} \Big( F^{1}, \ S^{1} [1, \sigma^{1}_{z} ] \cap \gamma = \emptyset, \ S^{3} [0, \tau^{3}_{n} ] \cap \gamma = \emptyset \Big),
\end{align}
where $\sigma^{1}_{z}$ and  $F^{1}$ were defined as in \eqref{mendoi-1} and \eqref{key-10}, respectively.

Let 
\begin{equation}\label{imp-8}
W:= \big( o_{1} +  [ - \frac{n}{4M}, \frac{n}{4M} ]^{3} \big) \cap \mathbb{Z}^{3},
\end{equation}
where $o_{1}$ was defined in Definition \ref{imp-notation}. Then 
\begin{align}\label{imp-9}
&P^{v}_{1} \otimes P^{z}_{2} \otimes P^{w}_{3} \Big( F^{1}, \ S^{1} [1, \sigma^{1}_{z} ] \cap \gamma = \emptyset, \ S^{3} [0, \tau^{3}_{n} ] \cap \gamma = \emptyset \Big) \notag \\
&\le P^{v}_{1} \otimes P^{z}_{2} \otimes P^{w}_{3} \Big( F^{1}_{\star}, \ S^{1} [1, \sigma^{1}_{z} ] \cap \gamma = \emptyset, \ S^{3} [ \tau^{3}_{\partial W}, \tau^{3}_{n} ] \cap \gamma = \emptyset \Big),
\end{align}
where $\tau^{3}_{\partial W} = \{ t \ | \ S^{3} ( t) \in \partial W \}$ and 
\begin{align}\label{imp-10}
&F^{1}_{\star} := \Big\{  \sigma^{1}_{z} < \tau^{1}_{n}, \  \sigma^{2}_{w} < \tau^{2}_{n}, \ S^{2} [1, T^{2}_{z, 6 \epsilon n} ] \cap \overline{(\epsilon n B'_{x})} = \emptyset, \ S^{3} [ 1, T^{3}_{w, 6 \epsilon n} ] \cap  \overline{(\epsilon n B'_{y})}  = \emptyset \notag \\ 
&    LE ( S^{1} [0, \sigma^{1}_{z} ] ) [0, \overline{\sigma}_{1}] \cap ( S^{2} [ 0 , \sigma^{2}_{w} ] \cup S^{3} [0, \tau^{3}_{\partial W}]) = \emptyset, \ LE ( S^{1} [0, \sigma^{1}_{z} ] + S^{2} [0, \sigma^{2}_{w}] ) [0, \overline{\sigma}_{2}] \cap S^{3} [ 0 , \tau^{3}_{\partial W} ] = \emptyset  \Big\}.
\end{align}
(See \eqref{mendoi-2} for $\overline{\sigma}_{1}$ and $\overline{\sigma}_{2}$.)
By the strong Markov property,
\begin{align}\label{imp-11}
&P^{v}_{1} \otimes P^{z}_{2} \otimes P^{w}_{3} \Big( F^{1}_{\star}, \ S^{1} [1, \sigma^{1}_{z} ] \cap \gamma = \emptyset, \ S^{3} [ \tau^{3}_{\partial W}, \tau^{3}_{n} ] \cap \gamma = \emptyset \Big) \notag \\
&\le P^{v}_{1} \otimes P^{z}_{2} \otimes P^{w}_{3} \Big( F^{1}_{\star}, \ S^{1} [1, \sigma^{1}_{z} ] \cap \gamma = \emptyset \Big) \max_{ w_{\star } \in \partial W} P^{w_{\star }}_{3} \Big( S^{3} [ 0, \tau^{3}_{n} ] \cap \gamma = \emptyset \Big).
\end{align}

Recall that $q^{(1)}$ was defined as in \eqref{q1} (we use the same notation here). Suppose that $0 \le r \le \log_{2} l -3$. We will first deal with $P^{v}_{1} \otimes P^{z}_{2} \otimes P^{w}_{3} \Big( F^{1}_{\star}, \ S^{1} [1, \sigma^{1}_{z} ] \cap \gamma = \emptyset, \ q^{(1)} = r \Big)$. However, as in \eqref{key-38}, we have
\begin{align}\label{imp-12}
&P^{v}_{1} \otimes P^{z}_{2} \otimes P^{w}_{3} \Big( F^{1}_{\star}, \ S^{1} [1, \sigma^{1}_{z} ] \cap \gamma = \emptyset, \ q^{(1)} = r \Big) \notag \\
&\le \frac{C}{\epsilon  n} \frac{1}{\epsilon l n} Es ( \epsilon n, \epsilon l n ) \notag \\
& \times \sum_{z' \in \partial B (z, \frac{\epsilon l n}{2} ) } E^{v}_{1} \otimes E^{z}_{2} \Big\{ {\bf 1}_{F^{2}_{\star } \cap \{ S^{1} [1, \sigma^{1}_{z} ] \cap \gamma = \emptyset \}} \max_{w_{1} \in \partial B ( w, \frac{\epsilon l n}{4} ) } P^{w_{1}}_{3} \Big( LE ( S^{1} [0, \sigma^{1}_{z} ] ) [0, \overline{\sigma}_{1}] \cap S^{3} [ 0 , \tau^{3}_{\partial W}] = \emptyset \Big) \Big\},
\end{align}
where 
\begin{align}\label{imp-13}
&F^{2}_{\star } = \Big\{ \sigma^{1}_{z} < \tau^{1}_{n}, \ S^{2} [1, T^{2}_{z, \frac{l \epsilon n}{2}} ] \cap \overline{(\epsilon n B'_{x})} = \emptyset, \notag \\
& \ \ \ \ \  \ \ \ \ LE ( S^{1} [0, \sigma^{1}_{z} ] ) [0, \overline{\sigma}_{1}] \cap  S^{2} [ 0 , T^{2}_{z, \frac{l \epsilon n}{2}} ] = \emptyset, \ q^{(1)}= r, \ S^{2} ( T^{2}_{z, \frac{l \epsilon n}{2}} ) = z' \Big\}.
\end{align}
(Recall that in order to show \eqref{imp-12}, we have to estimate $\tilde{p}_{1} $ defined as in \eqref{key-16}. Note that we don't need to care about the ``non-intersecting with $\gamma$" conditions as long as we deal with $\tilde{p}_{1} $.)

Using the time reversibility of LERW (see Lemma \ref{reversal}) as in \eqref{key-39}, we see that
\begin{align}\label{imp-14}
&E^{v}_{1} \otimes E^{z}_{2} \Big\{ {\bf 1}_{F^{2}_{\star } \cap \{ S^{1} [1, \sigma^{1}_{z} ] \cap \gamma = \emptyset \}} \max_{w_{1} \in \partial B ( w, \frac{\epsilon l n}{4} ) } P^{w_{1}}_{3} \Big( LE ( S^{1} [0, \sigma^{1}_{z} ] ) [0, \overline{\sigma}_{1}] \cap S^{3} [ 0 , \tau^{3}_{ \partial W}] = \emptyset \Big) \Big\} \notag \\
&= E^{z}_{1} \otimes E^{z}_{2} \Big\{ {\bf 1}_{\tilde{F}^{2}_{\star } \cap \{ S^{1} [0, \tau^{1}_{v} ] \cap ( \gamma \setminus \{v \} ) = \emptyset \}} \max_{w_{1} \in \partial B ( w, \frac{\epsilon l n}{4} ) } P^{w_{1}}_{3} \Big( LE ( S^{1} [0, \tau^{1}_{v} ] ) [\tilde{u}_{1}, \tilde{u}_{2}] \cap S^{3} [ 0 , \tau^{3}_{ \partial W}] = \emptyset \Big) \Big\},
\end{align}
where 
\begin{equation}\label{imp-15}
\tilde{u}_{1} := \max \{ t \ | \ LE ( S^{1} [0, \tau^{1}_{v} ] ) (t) \in \partial (\epsilon n B'_{x} ) \}, \ \tilde{u}_{2} := \text{ len} LE ( S^{1} [0, \tau^{1}_{v} ] ),
\end{equation}
and 
\begin{align}\label{imp-16}
&\tilde{F}^{2}_{\star } := \Big\{ \tau^{1}_{v} < \tau^{1}_{n}, \ S^{2} [1, T^{2}_{z, \frac{l \epsilon n}{2}} ] \cap \overline{(\epsilon n B'_{x})} = \emptyset, \ S^{2} ( T^{2}_{z, \frac{l \epsilon n}{2}} ) = z', \notag \\
& \ \ \ \ \  \ \ \  \  LE ( S^{1} [0, \tau^{1}_{v} ] ) [\tilde{u}_{1}, \tilde{u}_{2}] \cap S^{2} [ 0 , T^{2}_{z, \frac{l \epsilon n}{2}} ] = \emptyset, \ LE ( S^{1} [0, \tau^{1}_{v} ] ) [0, \tilde{u}_{1}] \cap A^{r}_{z} \neq \emptyset, \notag \\
& \ \ \ \  \ \ \ \ \  LE ( S^{1} [0, \tau^{1}_{v} ] ) [0, \tilde{u}_{1}] \subset B (z, 2^{r} \epsilon n ) \Big\}.
\end{align}

Let $\beta = LE ( S^{1} [0, \tau^{1}_{v} ] )$ and $\tau^{\beta }_{z, \frac{n}{12 M} } := \inf \{ t \ | \ \beta (t) \in \partial B (z, \frac{n}{12 M} ) \}$ . Since $r \le \log_{2} l -3$, on $\tilde{F}^{2}_{\star }$,  we have
\begin{equation}\label{imp-17}
\beta [ \tau^{\beta }_{z, \frac{n}{12 M} }, \tilde{u}_{2} ] \cap \overline{(\epsilon n B'_{x})} = \emptyset.
\end{equation}
Indeed, if $\beta [ \tau^{\beta }_{z, \frac{n}{12 M} }, \tilde{u}_{2} ] \cap \overline{(\epsilon n B'_{x})} \neq \emptyset$, then $\tau^{\beta }_{z, \frac{n}{12 M} } < \tilde{u}_{1}$, which implies that $\beta [0, \tilde{u}_{1} ] \cap \partial B (z, \frac{n}{12 M} ) \neq \emptyset$. Since $\epsilon n B'_{x} \subset F^{\gamma }_{i, n}$ and $\epsilon n B'_{y} \subset F^{\gamma }_{i, n}$, we have $|\epsilon n x - \epsilon n y | \le \frac{\sqrt{3} n }{4 M}$ which implies that $\epsilon l \le \frac{\sqrt{3}}{4M}$. Thus
\begin{equation*}
2^{r} \epsilon n \le 2^{\log_{2} -3 } \epsilon n \le \frac{1}{8} \epsilon l n \le \frac{\sqrt{3} n}{32 M} < \frac{n}{12 M},
\end{equation*}
which contradicts $\beta [0, \tilde{u}_{1}] \subset B (z, 2^{r} \epsilon n )$. Therefore, we get \eqref{imp-17}. Thus if we define $\tilde{u}_{1}^{\star }$ by 
\begin{equation*}
\tilde{u}_{1}^{\star } := \max \{ t \le \tau^{\beta }_{z, \frac{n}{12 M} } \ | \ LE ( S^{1} [0, \tau^{1}_{v} ] ) (t) \in \partial (\epsilon n B'_{x} ) \},
\end{equation*}
then $\tilde{u}_{1} = \tilde{u}_{1}^{\star }$. Since $\beta [0, \tilde{u}_{1}^{\star }]$ is $\beta [0, \tau^{\beta }_{z, \frac{n}{12 M} } ]$-measurable, by Proposition 4.2 and 4.4 \cite{Mas}, we see that 
\begin{align}\label{imp-18}
&E^{z}_{1} \otimes E^{z}_{2} \Big\{ {\bf 1}_{\tilde{F}^{2}_{\star } \cap \{ S^{1} [0, \tau^{1}_{v} ] \cap ( \gamma \setminus \{v \} ) = \emptyset \}} \max_{w_{1} \in \partial B ( w, \frac{\epsilon l n}{4} ) } P^{w_{1}}_{3} \Big( LE ( S^{1} [0, \tau^{1}_{v} ] ) [\tilde{u}_{1}, \tilde{u}_{2}] \cap S^{3} [ 0 , \tau^{3}_{ \partial W}] = \emptyset \Big) \Big\} \notag \\
&\le C P^{z}_{1} \Big( \tau^{1}_{v} < \tau^{1}_{n}, \  S^{1} [0, \tau^{1}_{v} ] \cap ( \gamma \setminus \{v \} ) = \emptyset \Big) \notag \\
&\times E^{z}_{1} \otimes E^{z}_{2} \Big\{ {\bf 1} \Big[ S^{2} [1, T^{2}_{z, \frac{l \epsilon n}{2}} ] \cap \overline{(\epsilon n B'_{x})} = \emptyset, \ S^{2} ( T^{2}_{z, \frac{l \epsilon n}{2}} ) = z', \notag \\
&\ \ \ \ \ \ \ \ \ \ \ \ \ \ \  \ \ \ \ \ \tilde{\beta } [\tilde{u}_{3}, \tilde{u}_{4}] \cap S^{2} [ 0 , T^{2}_{z, \frac{l \epsilon n}{2}} ] = \emptyset, \ \tilde{ \beta } [0, \tilde{u}_{3}] \cap A^{r}_{z} \neq \emptyset, \ \tilde{\beta } [0, \tilde{u}_{3}] \subset B (z, 2^{r} \epsilon n ) \Big]  \notag \\
&\ \ \ \ \ \ \ \ \ \ \ \ \ \ \  \ \ \ \times \max_{w_{1} \in \partial B ( w, \frac{\epsilon l n}{4} ) } P^{w_{1}}_{3} \Big( \tilde{\beta } [\tilde{u}_{3}, \tilde{u}_{4}] \cap S^{3} [ 0 , \tau^{3}_{ \partial W}] = \emptyset \Big) \Big\},
\end{align}
where $\tilde{\beta } = LE ( S^{1} [0, T^{1}_{z, \frac{n}{3M}} ] )$, $\tilde{u}_{4} = \inf \{ t \ | \ \tilde{\beta } (t) \in \partial B (z, \frac{n}{12 M} ) \}$ and $\tilde{u}_{3} = \max \{ t \le \tilde{u}_{4} \ | \ \tilde{\beta } (t) \in \partial (\epsilon n B'_{x} ) $.

However, as in \eqref{key-64}, we have 
\begin{align}\label{imp-19}
&(\text{RHS of \eqref{imp-18}}) \notag \\
&\le C P^{z}_{1} \Big( \tau^{1}_{v} < \tau^{1}_{n}, \  S^{1} [0, \tau^{1}_{v} ] \cap ( \gamma \setminus \{v \} ) = \emptyset \Big) \notag \\
&\times \frac{1}{\epsilon n} P^{z}_{2} \Big( S^{2} ( T^{2}_{z, \frac{l \epsilon n}{2}} ) = z' \Big) ES ( \epsilon n, \epsilon l n ) 2^{- \delta r } Es (\epsilon l n, \frac{n}{M} ).
\end{align}

Therefore, by \eqref{imp-12},
\begin{align}\label{imp-20}
&P^{v}_{1} \otimes P^{z}_{2} \otimes P^{w}_{3} \Big( F^{1}_{\star}, \ S^{1} [1, \sigma^{1}_{z} ] \cap \gamma = \emptyset, \ q^{(1)} = r \Big) \notag \\
&\le \frac{C}{\epsilon  n} \frac{1}{\epsilon l n} \frac{1}{\epsilon n} 2^{- \delta r } ES ( \epsilon n, \epsilon l n ) Es (\epsilon  n, \frac{n}{M} ) P^{z}_{1} \Big( \tau^{1}_{v} < \tau^{1}_{n}, \  S^{1} [0, \tau^{1}_{v} ] \cap ( \gamma \setminus \{v \} ) = \emptyset \Big).
\end{align}
Taking sum for $ r \le \log_{2} l -3$, by \eqref{imp-11},
\begin{align}\label{imp-21}
&P^{v}_{1} \otimes P^{z}_{2} \otimes P^{w}_{3} \Big( F^{1}_{\star}, \ S^{1} [1, \sigma^{1}_{z} ] \cap \gamma = \emptyset, \ S^{3} [ \tau^{3}_{\partial W}, \tau^{3}_{n} ] \cap \gamma = \emptyset, \ q^{(1)} \le \log_{2} l -3 \Big) \notag \\
&\le \frac{C}{\epsilon  n} \frac{1}{\epsilon l n} \frac{1}{\epsilon n}  ES ( \epsilon n, \epsilon l n ) Es (\epsilon  n, \frac{n}{M} ) P^{z}_{1} \Big( \tau^{1}_{v} < \tau^{1}_{n}, \  S^{1} [0, \tau^{1}_{v} ] \cap ( \gamma \setminus \{v \} ) = \emptyset \Big) \notag \\
& \times \max_{ w_{\star } \in \partial W} P^{w_{\star }}_{3} \Big( S^{3} [ 0, \tau^{3}_{n} ] \cap \gamma = \emptyset \Big).
\end{align}
Similar argument gives that 
\begin{align}\label{imp-22}
&P^{v}_{1} \otimes P^{z}_{2} \otimes P^{w}_{3} \Big( F^{1}_{\star}, \ S^{1} [1, \sigma^{1}_{z} ] \cap \gamma = \emptyset, \ S^{3} [ \tau^{3}_{\partial W}, \tau^{3}_{n} ] \cap \gamma = \emptyset \Big) \notag \\
&\le \frac{C}{\epsilon  n} \frac{1}{\epsilon l n} \frac{1}{\epsilon n}  ES ( \epsilon n, \epsilon l n ) Es (\epsilon  n, \frac{n}{M} ) P^{z}_{1} \Big( \tau^{1}_{v} < \tau^{1}_{n}, \  S^{1} [0, \tau^{1}_{v} ] \cap ( \gamma \setminus \{v \} ) = \emptyset \Big) \notag \\
& \times \max_{ w_{\star } \in \partial W} P^{w_{\star }}_{3} \Big( S^{3} [ 0, \tau^{3}_{n} ] \cap \gamma = \emptyset \Big).
\end{align}

By reversing the path, we see that 
\begin{equation*}
P^{z}_{1} \Big( \tau^{1}_{v} < \tau^{1}_{n}, \  S^{1} [0, \tau^{1}_{v} ] \cap ( \gamma \setminus \{v \} ) = \emptyset \Big) \le \frac{C M}{n} P^{v} \Big( S [1, T_{v, \frac{n}{8M}}] \cap \gamma = \emptyset \Big).
\end{equation*}
Recall that $\ell_{1}$ was defined in Definition \ref{imp-notation}. Note that $\ell_{1}$ intersects with $\partial B ( v, \frac{n}{8M} )$ at only one point. We call the point $v'$. Let $A := \partial B ( v, \frac{n}{8M} ) \cap B( v', \frac{n}{16M} )$. By Proposition 6.1.1 \cite{Shi-gr} and by the Harnack principle (see Theorem 1.7.6 \cite{Law-book90}), 
\begin{align}\label{imp-23}
&P^{v} \Big( S [1, \tau_{n} ] \cap \gamma = \emptyset \Big) \notag \\
&\ge P^{v} \Big( S[1, T_{v, \frac{n}{8M}} ] \cap \gamma = \emptyset, \ S ( T_{v, \frac{n}{8M}} ) \in A, \ \tau_{\partial W} < \tau_{n}, \ S[T_{v, \frac{n}{8M}}, \tau_{\partial W} ] \cap \gamma = \emptyset, \ S[\tau_{\partial W}, \tau_{n} ] \cap \gamma = \emptyset \Big) \notag \\
&\ge c P^{v} \Big( S[1, T_{v, \frac{n}{8M}} ] \cap \gamma = \emptyset \Big) \max_{ w_{\star } \in \partial W} P^{w_{\star }}_{3} \Big( S^{3} [ 0, \tau^{3}_{n} ] \cap \gamma = \emptyset \Big). 
\end{align}
Thus by \eqref{imp-22},
\begin{align}\label{imp-24}
&P^{v}_{1} \otimes P^{z}_{2} \otimes P^{w}_{3} \Big( F^{1}_{\star}, \ S^{1} [1, \sigma^{1}_{z} ] \cap \gamma = \emptyset, \ S^{3} [ \tau^{3}_{\partial W}, \tau^{3}_{n} ] \cap \gamma = \emptyset \Big) \notag \\
&\le \frac{C}{\epsilon  n} \frac{1}{\epsilon l n} \frac{1}{\epsilon n} \frac{M}{n} ES ( \epsilon n, \epsilon l n ) Es (\epsilon  n, \frac{n}{M} ) P^{v} \Big( S [1, T_{v, \frac{n}{8M}}] \cap \gamma = \emptyset \Big) \notag \\
&\times \max_{ w_{\star } \in \partial W} P^{w_{\star }}_{3} \Big( S^{3} [ 0, \tau^{3}_{n} ] \cap \gamma = \emptyset \Big) \notag \\
&\le \frac{C}{\epsilon  n} \frac{1}{\epsilon l n} \frac{1}{\epsilon n} \frac{M}{n} ES ( \epsilon n, \epsilon l n ) Es (\epsilon  n, \frac{n}{M} ) P^{v} \Big( S [1, \tau_{n} ] \cap \gamma = \emptyset \Big).
\end{align}
Taking sum for $z \in \partial (\epsilon n B'_{x})$ and $w \in \partial (\epsilon n B'_{y})$, by \eqref{imp-7},
\begin{align*}
&P^{v} \Big( LE ( S [0, \tau_{n} ] ) \cap \epsilon n B'_{x} \neq \emptyset, \  LE ( S[0, \tau_{n} ] ) \cap \epsilon n B'_{y} \neq \emptyset, \ T^{x} < T^{y}, \ S[1, \tau_{n} ] \cap \gamma = \emptyset \Big) \notag \\
&\le \frac{C \epsilon M}{l} Es (\epsilon n, \epsilon l n ) Es (\epsilon n, \frac{n}{M} ) P^{v} \Big( S [1, \tau_{n} ] \cap \gamma = \emptyset \Big).
\end{align*}
Combining this with \eqref{imp-6}, we finish the proof.
\end{proof}

\medskip

Next we will consider the lower bound of the probability that $LE ( X [0, \tau^{X}_{n} ] ) \cap \epsilon n B'_{x} \neq \emptyset$. 
Assume that $\gamma$ is a simple path and $X$ is a conditioned random walk defined as in Definition \ref{imp-notation}. Let
\begin{equation}\label{stopping}
t^{\gamma }_{i, n} := \inf \{ t \ | \ LE ( X [0, \tau^{X}_{n} ] ) (t) \in \partial D_{i+1, n} \},
\end{equation} 
where $D_{i+1, n}$ was defined in Definition \ref{imp-notation}. Then we have the following lemma, which is an analog of \eqref{moment-2} for the probability that the loop erasure of $X$ hits a cube.

\begin{lem}\label{imp-25}
Suppose that $x \in \mathbb{Z}^{3}$ satisfies $\epsilon n B'_{x} \subset F^{\gamma }_{i, n}$ (see Definition \ref{imp-notation} for $F^{\gamma }_{i, n}$). Then there exists an absolute constant $c > 0$ such that 
\begin{equation}\label{imp-26}
P_{X} \Big(  LE \big( X [0, \tau^{X}_{n} ]  \big) [0, t^{\gamma }_{i, n} ] \cap \epsilon n B'_{x} \neq \emptyset \Big) \ge c \epsilon M Es (\epsilon n, \frac{n}{M} ).
\end{equation}
\end{lem}

\begin{proof}
Take $w \in \partial_{i} B ( \epsilon n x, \frac{\epsilon n}{1000} )$. Throughout the proof, we write
\begin{equation}\label{imp-27}
B_{1} := B ( \epsilon n x, \frac{\epsilon n}{1000} ), \ B_{2} := B ( w, \frac{\epsilon n}{8} ), \ B_{3} := B ( w, \frac{ n}{4M}).
\end{equation}
Suppose that $S^{1}$ and $S^{2}$ are independent simple random walks started at $w$. Let 
\begin{equation}\label{imp-28}
t^{i} := \inf \{ t \ | \ S^{i} (t) \in \partial B_{3} \},
\end{equation}
for each $i = 1, 2$. Recall that the line segment $\ell_{1}$ was defined in Definition \ref{imp-notation}. We define random sets $A_{i}$ as follows. Let $\ell^{1}$ be the line segment started at $y_{1} := S^{1} ( t^{1} )$ and terminated at $v$. Define $A_{1}$ by $A_{1} := \{ y \ | \ \text{dist} ( y, \ell^{1} ) \le \frac{n}{20 M } \}$. Let $w^{1}$ be the intersection point of the line segment connecting $v$ with $w$ and $\partial B_{3}$, and let $w^{2} \in \partial B_{3}$ be the point such that $\frac{ w^{1} + w^{2}}{2} = w$. Let $\ell^{2}$ be the line segment starting from $ w^{2}$ terminated at $\partial B (v, \frac{L_{0} n}{M} )$ which is parallel to $\ell_{1}$. Here $L_{0}$ is a (large) constant which will be defined later. Define $A_{2}$ by $A_{2} :=  \{ y \ | \ \text{dist} ( y, \ell^{2} ) \le \frac{n}{20 M } \}$. Let $\partial_{2} := \partial A_{2} \cap \{ y \ | \ \text{dist} ( y, \partial B (v, \frac{L_{0} n}{M} ) ) \le \frac{n}{20 M } \}$. For each $i=1, 2$, we write $u^{i} := \{ t \ge t^{i} \ | \ S^{i} (t) \in \partial A_{i} \}$. Finally, let $H^{i} := B ( w, \frac{n}{6M} ) \cup B(w^{i}, \frac{n}{8M} )$ for each $i=1, 2$ and $\partial_{1} := \partial B_{3} \cap B (w^{1}, \frac{n}{8M} )$.

We write 
\begin{equation*}
\sigma^{1} := \max \{ t \ | \ LE ( S^{1} [0, t^{1} ] ) (t) \in B_{2} \} \text{ and } \sigma^{2} := \text{len} LE ( S^{1} [0, t^{1} ] ).
\end{equation*}
Let 
\begin{equation*}
\tau_{i, n} := \inf \{ t \ | \ LE ( S [0, \tau_{n} ] ) (t) \in \partial D_{i+1, n} \}.
\end{equation*}
Then we have 
\begin{equation}\label{imp-29}
P_{X} \Big(  LE \big( X [0, \tau^{X}_{n} ] \big) [0, t^{\gamma }_{i, n} ] \cap \epsilon n B'_{x} \neq \emptyset \Big) = \frac{P^{v} \Big(  LE \big( S [0, \tau_{n} ]\big) [0, \tau_{i, n} ]  \cap \epsilon n B'_{x} \neq \emptyset, \ S[1, \tau_{n} ] \cap \gamma = \emptyset \Big) }{P^{v} \Big( S[1, \tau_{n} ] \cap \gamma = \emptyset \Big) }.
\end{equation}
By considering the last exit from $B_{1}$ and by reversing a path, we have 
\begin{align}\label{imp-30}
&P^{v} \Big(  LE \big( S [0, \tau_{n} ] \big) [0, \tau_{i, n} ]  \cap \epsilon n B'_{x} \neq \emptyset, \ S[1, \tau_{n} ] \cap \gamma = \emptyset \Big) \notag \\
&\ge \sum_{w \in \partial_{i} B_{1} } P^{w}_{1} \otimes P^{w}_{2} \Big( S^{2} [1, t^{2} ] \cap B_{1} = \emptyset, \ LE ( S^{1} [0, t^{1} ] ) [ \sigma^{1}, \sigma^{2} ] \cap S^{2} [0, t^{2} ] = \emptyset, \notag \\
& LE ( S^{1} [0, t^{1} ] ) [ \sigma^{1}, \sigma^{2} ] \subset H^{1}, \ LE ( S^{1} [0, t^{1} ] ) [ 0, \sigma^{1} ] \cap B(w^{1}, \frac{n}{8M} ) = \emptyset, \ S^{2} [0, t^{2} ] \subset H^{2}, \notag \\
& S^{1} [t^{1}, \tau^{1}_{v} ] \subset A_{1} \cap (\gamma \setminus \{ v \} )^{c}, \ S^{2} (u^{2} ) \in \partial_{2}, \ S^{2} [u^{2}, \tau^{2}_{n} ] \cap \gamma = \emptyset, \ S^{2} [u^{2}, \tau^{2}_{n} ] \cap B(v, \frac{n}{M} ) = \emptyset \Big)
\end{align}

By using Lemma \ref{escape-17} and \ref{escape-25} as in \eqref{sansho} and by the strong Markov property, we have 
\begin{align}\label{imp-31}
&(\text{The probability in RHS of \eqref{imp-30}}) \notag \\
&\ge \frac{c_{L_{0}}}{\epsilon n} Es ( \epsilon n, \frac{n}{M} ) \notag \\
&\ \times \min_{y_{1} \in \partial_{1},  y_{2} \in \partial_{2} }  P^{y_{1}}_{1}  \Big( S^{1} [0, \tau^{1}_{v} ] \subset A_{1} \cap (\gamma \setminus \{ v \} )^{c} \Big) P^{y_{2}}_{2} \Big(   S^{2} [0, \tau^{2}_{n} ] \cap \gamma = \emptyset, \ S^{2} [0, \tau^{2}_{n} ] \cap B(v, \frac{n}{M} ) = \emptyset \Big).
\end{align}
But by reversing a path and by Proposition 6.1.1 \cite{Shi-gr}, we see that for each $y_{1} \in \partial_{1}$, 
\begin{equation}\label{imp-32}
P^{y_{1}}_{1}  \Big( S^{1} [0, \tau^{1}_{v} ] \subset A_{1} \cap (\gamma \setminus \{ v \} )^{c} \Big) \ge \frac{c M}{n} P^{v} \Big( S[1, T_{v, \frac{n}{30M} }] \cap \gamma = \emptyset \Big).
\end{equation}

On the other hand, by Lemma 6.1.2 \cite{Shi-gr} and the Harnack principle (see Theorem 1.7.6 \cite{Law-book90}), there exists an absolute constant $C_{0} < \infty$ such that 
\begin{equation}\label{imp-33}
\max_{y \in B (v, \frac{n}{M} ) } P^{y} \big( S[0, \tau_{n} ] \cap \gamma = \emptyset \big) \le C_{0} \min_{y \in \partial_{2} } P^{y} \big( S[0, \tau_{n} ] \cap \gamma = \emptyset \big).
\end{equation}
Now take $L_{0}$ such that $\frac{2 C_{0}}{L_{0}} < \frac{1}{2}$. Then by the strong Markov property and Proposition 1.5.10 \cite{Law-book90}, we see that for each $y_{2} \in \partial_{2}$,
\begin{align}\label{imp-34}
&P^{y_{2}}_{2} \Big(   S^{2} [0, \tau^{2}_{n} ] \cap \gamma = \emptyset, \ S^{2} [0, \tau^{2}_{n} ] \cap B(v, \frac{n}{M} ) = \emptyset \Big) \notag \\
&\ge P^{y_{2}}_{2} \Big(   S^{2} [0, \tau^{2}_{n} ] \cap \gamma = \emptyset \Big) - P^{y_{2}}_{2} \Big(   S^{2} [0, \tau^{2}_{n} ] \cap \gamma = \emptyset, \ S^{2} [0, \tau^{2}_{n} ] \cap B(v, \frac{n}{M} ) \neq \emptyset \Big) \notag \\
&\ge P^{y_{2}}_{2} \Big(   S^{2} [0, \tau^{2}_{n} ] \cap \gamma = \emptyset \Big) - \frac{2}{L_{0}} \max_{y' \in B (v, \frac{n}{M} ) } P^{y'} \big( S[0, \tau_{n} ] \cap \gamma = \emptyset \big) \notag \\
&\ge P^{y_{2}}_{2} \Big(   S^{2} [0, \tau^{2}_{n} ] \cap \gamma = \emptyset \Big) - \frac{2 C_{0}}{L_{0}} \min_{y' \in \partial_{2} } P^{y'} \big( S[0, \tau_{n} ] \cap \gamma = \emptyset \big) \notag \\
&\ge P^{y_{2}}_{2} \Big(   S^{2} [0, \tau^{2}_{n} ] \cap \gamma = \emptyset \Big) - \frac{2 C_{0}}{L_{0}}  P^{y} \big( S[0, \tau_{n} ] \cap \gamma = \emptyset \big) \notag \\
&\ge \frac{1}{2} \min_{y_{2} \in \partial_{2} } P^{y_{2}}_{2} \Big(   S^{2} [0, \tau^{2}_{n} ] \cap \gamma = \emptyset \Big).
\end{align}
Again by Lemma 6.1.2 \cite{Shi-gr} and the Harnack principle (see Theorem 1.7.6 \cite{Law-book90}), we have 
\begin{equation}\label{imp-35}
\max_{y \in  B (v, \frac{n}{30M} ) } P^{y} \big( S[0, \tau_{n} ] \cap \gamma = \emptyset \big) \le C \min_{y_{2} \in \partial_{2} } P^{y_{2}}_{2} \Big(   S^{2} [0, \tau^{2}_{n} ] \cap \gamma = \emptyset \Big).
\end{equation}

Combining these estimates, we have 
\begin{align}\label{imp-36}
&P^{v} \Big(  LE \big( S [0, \tau_{n} ] \big) [0, \tau_{i, n} ]  \cap \epsilon n B'_{x} \neq \emptyset, \ S[1, \tau_{n} ] \cap \gamma = \emptyset \Big) \notag \\
&c \epsilon M Es (\epsilon n, \frac{n}{M} ) P^{v} \Big( S[1, T_{v, \frac{n}{30M} }] \cap \gamma = \emptyset \Big) \max_{y \in  B (v, \frac{n}{30M} ) } P^{y} \Big( S[0, \tau_{n} ] \cap \gamma = \emptyset \Big) \notag \\
&c \epsilon M Es (\epsilon n, \frac{n}{M} ) P^{v} \Big( S[1, \tau_{n}] \cap \gamma = \emptyset \Big),
\end{align}
which finishes the proof.
\end{proof}

\subsection{Proof of \eqref{imp}}

Suppose that $\gamma$ is a simple path and $X$ is a conditioned random walk not to hit $\gamma$ as in Definition \ref{imp-notation}. Let 
\begin{equation}\label{imp-37}
J^{\gamma}_{i, n} := \sharp \Big\{ x \in \mathbb{Z}^{3} \ | \ \epsilon n B'_{x} \subset F^{\gamma}_{i, n}, \ LE \big( X [0, \tau^{X}_{n} ] \big) [0, t^{\gamma}_{i, n} ] \cap \epsilon n B'_{x} \neq \emptyset \Big\},
\end{equation}
where $F^{\gamma}_{i, n}$ and $t^{\gamma}_{i, n}$ were defined in Definition \ref{imp-notation} and \eqref{stopping}, respectively. We are interested in the lower bound of $J^{\gamma}_{i, n}$. Using Lemma \ref{imp-3}, \ref{imp-25}, and the second moment method as in Corollary \ref{positive-low}, we will prove Lemma \ref{imp-38} below. Lemma \ref{imp-38} is an analog of Corollary \ref{positive-low} for $X$. Then using iteration arguments as in Theorem 6.7 \cite{BM} and Proposition 8.2.5 \cite{Shi-gr}, we will prove Proposition \ref{imp-43} which immediately concludes \eqref{imp}.

\medskip

We begin with the following lemma, which shows that the number of cubes hit by the loop erasure of the conditioned random walk is bigger than the expected number of such cubes with positive probability. We may think of the next lemma as an analog of Corollary \ref{positive-low} for the number of cubes hit by the loop erasure of $X$.

\begin{lem}\label{imp-38}
There exists an absolute constant $c_{1} > 0$ such that 
\begin{equation}\label{imp-39}
P_{X} \Big( J^{\gamma}_{i, n} \ge c_{1} (\epsilon M)^{-2} Es ( \epsilon n, \frac{n}{M} ) \Big) \ge c_{1}.
\end{equation}
\end{lem}

\begin{proof}
By Lemma \ref{imp-3} and \eqref{escape-7}
\begin{align}\label{imp-40}
&E_{X} \big( (J^{\gamma}_{i, n})^{2} \big) \notag \\
&\le \sum_{x, y \in \mathbb{Z}^{3}, \epsilon n B'_{x}, \epsilon n B'_{y} \subset F^{\gamma }_{i, n}} P_{X} \Big(  LE ( X [0, \tau^{X}_{n} ] ) \cap \epsilon n B'_{x} \neq \emptyset, \  LE ( X[0, \tau^{X}_{n} ] ) \cap \epsilon n B'_{y} \neq \emptyset \Big)  \notag \\
&\le C \sum_{x \in \mathbb{Z}^{3}, \epsilon n B'_{x} \subset F^{\gamma }_{i, n}} \sum_{l=1}^{\frac{1}{\epsilon M}} l^{2} \frac{ \epsilon M}{l} Es (\epsilon n, \epsilon l n ) Es (\epsilon n, \frac{n}{M} ) \notag \\
&=C \sum_{x \in \mathbb{Z}^{3}, \epsilon n B'_{x} \subset F^{\gamma }_{i, n}}  \epsilon M  Es (\epsilon n, \frac{n}{M} )  \sum_{l=1}^{\frac{1}{\epsilon M}} l Es (\epsilon n, \epsilon l n ) \notag \\
&\le C \sum_{x \in \mathbb{Z}^{3}, \epsilon n B'_{x} \subset F^{\gamma }_{i, n}}  \epsilon M  Es (\epsilon n, \frac{n}{M} ) (\epsilon M)^{-\alpha - \delta } Es ( \epsilon n, \frac{n}{M} ) \sum_{l=1}^{\frac{1}{\epsilon M}} l^{1- \alpha - \delta } \notag \\
&(\text{Here we used \eqref{escape-7} to say that } Es (\epsilon n, \epsilon l n ) \le C (\epsilon M)^{-\alpha - \delta } l^{-\alpha - \delta} Es (\epsilon n, \frac{n}{M} ) \notag \\
& \text{ for some } \delta > 0 ) \notag \\
&\le C \Big\{ (\epsilon M )^{-2} Es (\epsilon n, \frac{n}{M} ) \Big\}^{2}.
\end{align}

On the other hand, by Lemma \ref{imp-25}, we have
\begin{align}\label{imp-41}
&E_{X} \big( J^{\gamma}_{i, n} \big) \notag \\
&\ge \sum_{x \in \mathbb{Z}^{3}, \epsilon n B'_{x} \subset F^{\gamma }_{i, n}} P_{X} \Big(  LE \big( X [0, \tau^{X}_{n} ]  \big) [0, t^{\gamma }_{i, n} ] \cap \epsilon n B'_{x} \neq \emptyset \Big) \notag \\
&\ge c (\epsilon M)^{-2} Es (\epsilon n, \frac{n}{M} ).
\end{align}

Therefore, we see that $E_{X} \big( (J^{\gamma}_{i, n})^{2} \big) \le C \Big( E_{X} \big( J^{\gamma}_{i, n} \big) \Big)^{2}$. By the second moment method, we finish the proof.
\end{proof}

\medskip

Let
\begin{equation}\label{imp-42}
J_{\epsilon, n} := \sharp \Big\{ x \in \mathbb{Z}^{3} \ | \ \epsilon n B'_{x} \subset B (\frac{2n}{3}) \setminus B (\frac{n}{3}), \ LE \big( S [0, \tau_{n} ] \big) \cap \epsilon n B'_{x} \neq \emptyset \Big\},
\end{equation}

Now we use iteration arguments explained in the beginning of Section 4. Using the iteration argument, we prove next proposition which gives \eqref{rough-1-1}.

\begin{prop}\label{imp-43}
For every $r > 0$, there exists $c_{r} > 0$ such that 
\begin{equation}\label{imp-44}
P \Big( J_{\epsilon, n} \ge c_{r} \epsilon^{-2} Es ( \epsilon n, n ) \Big) \ge 1-r,
\end{equation}
for all $\epsilon > 0$ and $n=n_{\epsilon} = 2^{j_{\epsilon }}$ satisfying \eqref{skorohod-2}.
\end{prop}

\begin{proof}
Take $r > 0$. Let $M_{r}$ be an integer satisfying $(1- c_{1} )^{\lfloor \frac{M_{r}}{20} \rfloor} < r$ where $c_{1} > 0$ is a constant as in Lemma \ref{imp-38}. We write $N_{r} := \lfloor \frac{M_{r}}{20} \rfloor$.

Let $\beta := LE \big( S [0, \tau_{n} ] \big)$ and recall that $\tau_{i, n} = \inf \{ t \ | \ \beta (t) \in \partial D_{i+1, n} \}$ where $D_{i, n}$ was defined as in Definition \ref{imp-notation}. For each $i=1, \cdots , N_{r}$, define
\begin{equation}\label{imp-44}
J_{i} := \sharp \Big\{  x \in \mathbb{Z}^{3} \ | \ \epsilon n B'_{x} \subset F^{\beta [0, \tau_{i-1, n} ] }_{i, n}, \ \beta [ \tau_{i-1, n}, \tau_{i, n} ] \cap \epsilon n B'_{x} \neq \emptyset \Big\}.
\end{equation}
(See Definition \ref{imp-notation} for $F^{\gamma}_{i, n}$.)

Then by the domain Markov property of LERW (see Lemma \ref{domain}) and by Lemma \ref{imp-38}, 
\begin{align}\label{imp-45}
&P \Big( J_{\epsilon, n} <  c_{1} (\epsilon M_{r})^{-2} Es ( \epsilon n, \frac{n}{M_{r}} ) \Big) \notag \\
&\le P \Big( J_{i} <  c_{1} (\epsilon M_{r})^{-2} Es ( \epsilon n, \frac{n}{M_{r}} ) \text{ for all } i=1, \cdots ,N_{r}  \Big) \notag \\
&\le E \Big\{ \bigcap_{i=1}^{N_{r} -1} \big\{ J_{i} <  c_{1} (\epsilon M_{r})^{-2} Es ( \epsilon n, \frac{n}{M_{r}} ) \big\} P \Big( J_{N_{r} } < c_{1} (\epsilon M_{r})^{-2} Es ( \epsilon n, \frac{n}{M_{r}} ) \ \big| \ \beta [0, \tau_{N_{r} -1, n} ] \Big)  \Big\} \notag \\
&\le (1- c_{1} ) P \Big( \bigcap_{i=1}^{N_{r} -1} \big\{ J_{i} <  c_{1} (\epsilon M_{r})^{-2} Es ( \epsilon n, \frac{n}{M_{r}} ) \big\}  \Big) \notag \\
&\le (1 - c_{1} )^{N_{r}} < r.
\end{align}
Since $ Es ( \epsilon n, \frac{n}{M_{r}} ) \ge c Es ( \epsilon n, n )$ for some $c > 0$, if we let $c_{r} := c c_{1} M_{r}^{-2}$, we finish the proof.
\end{proof}

By Proposition \ref{imp-43}, we get the following proposition immediately. Recall that $Y^{\epsilon }$ was defined in \eqref{number-box}.

\begin{prop}\label{imp-46}
For all $r > 0$, there exists $c_{r} > 0$ such that 
\begin{equation}\label{imp-47}
P \Big( Y^{\epsilon } \ge c_{r} \epsilon^{-2} Es ( \epsilon n, n ) \Big) \ge 1- r- \epsilon^{100},
\end{equation}
for every $\epsilon > 0$ and $n=n_{\epsilon} = 2^{j_{\epsilon }}$ satisfying \eqref{skorohod-2}.
\end{prop}

\begin{proof}
Note that $c Y^{\epsilon} \le J_{\epsilon, n} \le \frac{1}{c} Y^{\epsilon}$ for some absolute constant $c > 0$ on $\{ d_{\text{H}} ({\cal K}, \text{LEW}_{n} ) < \epsilon^{2} \}$. Thus if $n=n_{\epsilon} = 2^{j_{\epsilon }}$ satisfies \eqref{skorohod-2}, by \eqref{skorohod-3} and Proposition \ref{imp-43},
\begin{align}\label{imp-48}
&P \Big( Y^{\epsilon } \ge c c_{r} \epsilon^{-2} Es ( \epsilon n, n ) \Big) \notag \\
&\ge P \Big( Y^{\epsilon } \ge c c_{r} \epsilon^{-2} Es ( \epsilon n, n ), \ d_{\text{H}} ( \text{LEW}_{n_{\epsilon}}, {\cal K} ) < \epsilon^{2} \Big) \notag \\
&\ge P \Big(  J_{\epsilon, n} \ge c_{r} \epsilon^{-2} Es ( \epsilon n, n ), \ d_{\text{H}} ( \text{LEW}_{n_{\epsilon}}, {\cal K} ) < \epsilon^{2} \Big) \notag \\
&\ge P \Big(  J_{\epsilon, n} \ge c_{r} \epsilon^{-2} Es ( \epsilon n, n ) \Big) - P \Big( d_{\text{H}} ( \text{LEW}_{n_{\epsilon}}, {\cal K} ) \ge \epsilon^{2} \Big) \notag \\
&\ge 1-r- \epsilon^{100},
\end{align}
which finishes the proof.
\end{proof}

\medskip

\begin{rem}\label{poli-1}
By \eqref{first-moment-up} and Markov's inequality, we see that for all $r > 0$, 
\begin{equation}\label{imp-49}
P \Big(  Y^{\epsilon} \ge \frac{C}{r} \epsilon^{-2} Es (\epsilon n, n ) \Big) \le \frac{E ( Y^{\epsilon} )}{\frac{C}{r} \epsilon^{-2} Es (\epsilon n, n )} \le r,
\end{equation}
where $C$ is a constant as in \eqref{first-moment-up}.
\end{rem}

\section{Lower bound of $\text{dim}_{\text{H}} ({\cal K})$}
In this section, we will prove that 
\begin{equation}\label{dim-1}
\text{dim}_{\text{H}} ({\cal K}) \ge 2- \alpha, \text{ almost surely}.
\end{equation}
Combining this with Theorem 1.4 \cite{SS}, we have 
\begin{equation}\label{dim-2}
\text{dim}_{\text{H}} ({\cal K}) = 2- \alpha, \text{ almost surely}.
\end{equation}
In order to prove \eqref{dim-1}, we will use a standard technique so called Frostman's lemma (see Lemma \ref{dim-3}). We will review that lemma in Section 5.1. We then give some energy estimates for suitable sequence of measures whose supports converge to ${\cal K}$ (see Lemma \ref{dim-6}). Using Lemma \ref{dim-6}, we will prove \eqref{dim-1} in Section 5.2.

\subsection{Preliminaries}
In order to give a lower bound of the Hausdorff dimension of a set in $\mathbb{R}^{d}$, the Lemma \ref{dim-3} below is a standard criterion referred to as Frostman's lemma. In this subsection, we first state it. Then in Lemma \ref{dim-6}, we will estimate $\beta$-energy for suitable measures $\mu_{k} $ defined below.

\begin{lem}\label{dim-3} (Theorem 4.13 \cite{Falconer})
Suppose that $K \subset \overline{D}$ is a closed set and let $\mu$ be a positive measure supported on $K$ with $\mu ( K)  > 0$. Define $\beta$-energy $I_{\beta} ( \mu )$ by 
\begin{equation}\label{dim-4}
 I_{\beta} ( \mu ) = \int_{\overline{D}} \int_{\overline{D}} |x- y|^{- \beta } d \mu (x ) d \mu (y).
 \end{equation}
 If $I_{\beta} ( \mu ) < \infty$, then $\text{dim}_{\text{H}} (K) \ge \beta$.
 \end{lem}

\medskip

According to Lemma \ref{dim-3}, we need to construct a positive (random) measure $\mu$ supported on ${\cal K}$ such that its $(\beta - \delta )$-energy $I_{\beta - \delta } ( \mu )$ is finite with high probability for any $\delta > 0$, where $\beta := 2- \alpha$ (see Theorem \ref{escape-5} for $\alpha$). 
With this in mind, let $\epsilon = \epsilon_{k} = 2^{-k}$ for $k \ge 1$ and let $n=n_{\epsilon} = 2^{j_{\epsilon }}$ be an integer satisfying \eqref{skorohod-2}. Now we define a sequence of measures $\mu_{k}$ which approximates $\mu$ as follows. Let $\mu_{k}$ be the (random) measure whose density, with respect to Lebesgue measure, is $ \frac{1}{ \epsilon Es ( \epsilon n , n ) }$ on each box $\epsilon B_{x}$ with $x \in \mathbb{Z}^{3}$ and $\epsilon B_{x} \subset D_{\frac{2}{3}} \setminus D_{\frac{1}{3}}$ or $\epsilon B_{x} \cap \partial D_{\frac{i}{3}} \neq \emptyset$ for $i=1,2$ such that ${\cal K} \cap \epsilon B_{x} \neq \emptyset$, and assigns measure zero elsewhere. Then it is easy to check that $\text{supp} (\mu_{k+1} ) \subset \text{supp} (\mu_{k} )$ and with probability one $\bigcap_{k=1}^{\infty} \text{supp} (\mu_{k} ) \subset {\cal K}$.

Therefore, as we discussed as in Section 1.2, we need to show that for every $\delta > 0$ and $r > 0$ there exist constants $c_{r} > 0, C_{\delta, r} < \infty$ which do not depend on $\epsilon$ such that 
\begin{align}
&P \Big( I_{\beta - \delta } ( \mu_{k} ) \le C_{\delta, r} \Big) \ge 1-r, \label{sketch-1-1-1} \\
&P \Big( \mu_{k} (\overline{D} ) \ge c_{r} \Big) \ge 1-r. \label{sketch-1-1-2}
\end{align}
for all $k > 0$. Once \eqref{sketch-1-1-1} and \eqref{sketch-1-1-2} are proved, let $\mu$ be any weak limit of the $\mu_{k}$. Then the measure $\mu$ is a positive measure satisfying that its support is contained in ${\cal K}$ and the $(\beta - \delta )$-energy is finite with probability at least $1-r$. Using Lemma \ref{dim-3}, we get $\text{dim}_{\text{H}} ({\cal K}) \ge \beta - \delta$ with probability $\ge 1-r$, and Theorem \ref{main result} is proved.

For \eqref{sketch-1-1-2}, we have the following. Take an arbitrary $r > 0$. By Proposition \ref{imp-46}, with probability at least $1- r- \epsilon^{100}$, we have 
\begin{equation}\label{dim-5}
\mu_{k} (\overline{D} ) \ge \sum_{x \in \mathbb{Z}^{3}, \epsilon B_{x} \subset D_{\frac{2}{3}} \setminus D_{\frac{1}{3}} } {\bf 1} \{ {\cal K} \cap \epsilon B_{x} \neq \emptyset \} \frac{\epsilon^{2}}{Es ( \epsilon n , n )} = Y^{\epsilon } \frac{\epsilon^{2}}{Es ( \epsilon n , n )} \ge c_{r},
\end{equation}
for all $k$, which proves \eqref{sketch-1-1-2}.

For \eqref{sketch-1-1-1}, we start with the following lemma which gives a first moment estimate of $I_{\beta - \delta} ( \mu_{k} )$ for an arbitrary positive number $\delta$.

\begin{lem}\label{dim-6}
For every $\delta > 0$, there exists $C_{\delta } < \infty $ such that 
\begin{equation}\label{dim-7}
E \Big( I_{\beta - \delta} ( \mu_{k} ) \Big) \le C_{\delta},
\end{equation}
for all $k$. Here $\beta := 2 - \alpha $.
\end{lem}

\begin{proof}
Recall that we write $\epsilon = \epsilon_{k} = 2^{-k}$ for $k \ge 1$ and let $n=n_{\epsilon} = 2^{j_{\epsilon }}$ be an integer satisfying \eqref{skorohod-2}. Then by Theorem \ref{key thm},
\begin{align}\label{dim-hosokudao}
&E \Big( I_{\beta - \delta} ( \mu_{k} ) \Big) \le \sum_{x, y \in \mathbb{Z}^{3}, \ \epsilon B_{x}, \epsilon B_{y} \subset D_{\frac{2}{3}} \setminus D_{\frac{1}{3}}} E \Big\{ \int_{\epsilon B_{x}} \int_{\epsilon B_{y}} |z- w|^{- (\beta - \delta) } d \mu_{k} (z) d \mu_{k} (w) \Big\} \notag \\
&= \sum_{x, y \in \mathbb{Z}^{3}, \ \epsilon B_{x}, \epsilon B_{y} \subset D_{\frac{2}{3}} \setminus D_{\frac{1}{3}}} \big( \epsilon Es (\epsilon n, n ) \big)^{-2} \int_{\epsilon B_{x}} \int_{\epsilon B_{y}} |z- w|^{- (\beta - \delta) } dz dw P \Big( {\cal K} \cap \epsilon B_{x} \neq \emptyset \text{ and } {\cal K} \cap \epsilon B_{y} \neq \emptyset \Big) \notag \\
&\le C \sum_{x  \in \mathbb{Z}^{3}, \ \epsilon B_{x} \subset D_{\frac{2}{3}} \setminus D_{\frac{1}{3}}} \sum_{l = 1}^{\frac{2}{\epsilon}} l^{2} \big( \epsilon Es (\epsilon n, n ) \big)^{-2} (\epsilon l )^{-(\beta - \delta)} \epsilon^{6} Es ( \epsilon n , l \epsilon n) Es ( \epsilon n , n )  \frac{\epsilon}{l}  \notag \\
&+ C \sum_{x  \in \mathbb{Z}^{3}, \ \epsilon B_{x} \subset D_{\frac{2}{3}} \setminus D_{\frac{1}{3}}} \big( \epsilon Es (\epsilon n, n ) \big)^{-2} \int_{\epsilon B_{x}} \int_{\epsilon B_{x}} |z- w|^{- (\beta - \delta) } dz dw \ Es (\epsilon n, n ) \epsilon.
\end{align}
But RHS of \eqref{dim-hosokudao} is bounded above by
\begin{align}\label{dim-8}
&\le C \epsilon^{2- (\beta - \delta)} \frac{1}{ Es (\epsilon n, n ) } \sum_{l = 1}^{\frac{2}{\epsilon}} l^{1 - (\beta - \delta) } Es ( \epsilon n , l \epsilon n) + \frac{C \epsilon^{\alpha + \delta }}{Es (\epsilon n, n )} \notag \\
&\le C \epsilon^{2- (\beta - \delta)} \frac{1}{ Es (\epsilon n, n ) } \sum_{l = 1}^{\frac{2}{\epsilon}} l^{1 - (\beta - \delta) } (\epsilon l )^{-\alpha - \eta } Es (\epsilon n, n ) + \frac{C \epsilon^{\alpha + \delta }}{Es (\epsilon n, n )} \notag \\
&(\text{Here } \eta := \frac{\delta }{2} \text{ and we used \eqref{escape-10}}.) \notag \\
&\le C \epsilon^{2- (\beta - \delta) -\alpha - \eta } \sum_{l = 1}^{\frac{2}{\epsilon}} l^{1 - (\beta - \delta) -\alpha - \eta} + \frac{C \epsilon^{\alpha + \delta }}{Es (\epsilon n, n )} \le C_{\delta}, \notag \\
&(\text{Here we used \eqref{escape-7}})
\end{align}
which finishes the proof.
\end{proof}

\subsection{Proof of \eqref{dim-1}}
Now we are ready to prove the following theorem.

\begin{thm}\label{main-1}
Let $d=3$. Then 
\begin{equation}\label{main-2}
\text{dim}_{\text{H}} ({\cal K} ) \ge 2- \alpha, \text{ almost surely}.
\end{equation}
\end{thm}

\begin{proof}
Recall that $r > 0$ is an arbitrary positive number. Let $\delta > 0$ be an arbitrary positive number also. Define $C_{\delta, r }:= \frac{C_{\delta}}{r}$ where $C_{\delta}$ is a constant as in Lemma \ref{dim-6}. Take a constant $c_{r} > 0$ as in Proposition \ref{imp-46}. Let $\beta:= 2- \alpha$. By Lemma \ref{dim-6} and Markov's inequality,
\begin{equation}\label{main-3}
P \Big( I_{\beta - \delta} ( \mu_{k} ) \ge C_{\delta, r } \Big) \le \frac{C_{\delta }}{C_{\delta, r }} = r,
\end{equation}
for all $k$. Combining this with \eqref{dim-5}, we have
\begin{equation*}
P \Big( \mu_{k} (\overline{D} ) \ge c_{r}, \ I_{\beta - \delta} ( \mu_{k} ) \le C_{\delta, r } \Big) \ge 1 - 2r - 2^{-100 k},
\end{equation*}
for all $k$. By Fatou's lemma, this implies 
\begin{equation*}
P \Big( \mu_{k} (\overline{D} ) \ge c_{r}, \ I_{\beta - \delta} ( \mu_{k} ) \le C_{\delta, r } \ \text{i.o.} \Big) \ge 1 - 2r.
\end{equation*}
On the event above, let $\mu$ be any weak limit of the $\mu_{k}$. Then it is easy to verify that $\mu$ is supported on ${\cal K}$, $\mu ( {\cal K} ) \ge c_{r}$, and $I_{\beta - \delta} ( \mu ) \le C_{\delta, r } $. By Lemma \ref{dim-3}, we have 
\begin{equation*}
P \Big( \text{dim}_{\text{H}} ({\cal K} ) \ge 2- \alpha - \delta \Big) \ge 1 - 2r.
\end{equation*}
Since this holds for every $r > 0$ which is independent of $\delta > 0$,
\begin{equation*}
P \Big( \text{dim}_{\text{H}} ({\cal K} ) \ge 2- \alpha - \delta \Big) = 1.
\end{equation*}
Since this holds for every $\delta > 0$, we see that 
\begin{equation*}
P \Big( \text{dim}_{\text{H}} ({\cal K} ) \ge 2- \alpha  \Big) = 1,
\end{equation*}
which finishes the proof.
\end{proof}

\medskip

\begin{rem}\label{final}
We expect that 
\begin{equation}\label{compara}
Es (n) \asymp n^{-\alpha },
\end{equation}
in 3 dimensions. Here we write $a_{n} \asymp b_{n}$ if there exists $c > 0$ such that $c b_{n} \le a_{n} \le \frac{1}{c} b_{n}$ for all $n$.

This is proved for $d=2$ \cite{Law14}. The main steps in \cite{Law14} are 
\begin{itemize}
\item Write $Es (n)$ in terms of simple random walk quantities.

\item Estimate the simple random walk quantities.
\end{itemize} 
The simple random walk quantities as above come from the random walk loop measure which is related to the winding number of loops (see \cite{Law14}). In \cite{Law14}, by estimating such simple random walk quantities carefully, not only the relation as in \eqref{compara} but the exact value of $\alpha$ were also obtained in two dimensions ($\alpha = \frac{3}{4}$ in two dimensions).

Is it possible to find suitable simple random walk quantities to calculate $Es (n)$ and to compute the exact value of $\alpha$ for $d=3$? 
\end{rem}

\begin{rem}\label{final-1}
Recall that we write $Y$ for the union of ${\cal K}$ and loops from independent Brownian loop soup in $D$ which intersect ${\cal K}$, see \eqref{decomp}. Theorem 1.1 of \cite{SS} shows that $Y$ has the same distribution as the trace of three-dimensional Brownian motion. In Conjecture 1.3 of \cite{SS}, we conjectured that the law of ${\cal K}$ would be characterized uniquely by this decomposition. Namely, if the union of a random simple path $\tilde{{\cal K}}$ and loops from independent Brownian loop soup in $D$ which intersect $\tilde{{\cal K}}$ has the same distribution as the trace of three-dimensional Brownian motion, then we expect that $\tilde{{\cal K}}$ has the same distribution as ${\cal K}$. Thanks to Theorem \ref{main-result-1}, in this characterization, we may add one additional assumption for $\tilde{{\cal K}}$, i.e., $\text{dim}_{\text{H}} ( \tilde{{\cal K}} ) = \beta$ almost surely. We believe that this might be useful to prove the conjecture.

\end{rem}

{\bf Acknowledgements.} Enormous thanks go to Alain-Sol Sznitman for a careful reading of the paper and fruitful comments. This project was carried out while the author was enjoying
the hospitality of the Forschungsinstitut f\"{u}r Mathematik of the ETH Z\"{u}rich. He wishes to thank the institution. The research of the author has been supported by the Japan Society for the Promotion of Science (JSPS). Finally, the author thanks Hidemi Aihara for all her understanding and support.

 \end{document}